\documentclass[a4paper,11pt]{amsart}

\usepackage[T1]{fontenc}
\usepackage{ifxetex}

\ifxetex \relax
\else
\usepackage[utf8]{inputenc}
\fi
\usepackage{lmodern}
\usepackage{enumerate}
\usepackage{url}
\usepackage{comment}
\usepackage{%
amsfonts, amsthm, amsbsy, amstext, amssymb, amsopn, amsfonts, latexsym,stmaryrd} 
\usepackage{mathtools}
\mathtoolsset{mathic=true}
\usepackage[widespace]{fourier}
\usepackage{hyperref}

\RequirePackage[dvipsnames]{xcolor} % [dvipsnames]
\definecolor{halfgray}{gray}{0.55}%chapter numbers will be semi
\definecolor{webgreen}{rgb}{0,0.5,0}
\definecolor{webbrown}{rgb}{.6,0,0} 
\definecolor{red}{rgb}{1,0,0}

\hypersetup{%
   bookmarks=true,        % show bookmarks bar?
   unicode=false,         % non-Latin characters in Acrobat’s bookmarks
   pdftoolbar=true,       % show Acrobat’s toolbar?
   pdfmenubar=true,       % show Acrobat’s menu?
   pdffitwindow=false,    % window fit to page when opened
   pdfstartview={FitH},   % fits the width of the page to the window
   pdfnewwindow=true,     % links in new PDF window
   colorlinks=true,       % false: boxed links; true: colored links
   linkcolor=webbrown,    % color of internal links (change box color with linkbordercolor)
   citecolor=webgreen,    % color of links to bibliography
   filecolor=magenta,     % color of file links
   urlcolor=cyan,         % color of external links
   pdfcreator={pdfLaTeX}, % creator of the document
   pdftitle={Approximate orthogonality of powers for ergodic affine
unipotent diffeomorphisms on nilmanifolds},%
   pdfauthor={L. Flaminio, K. Fraczek, J. Kulaga-Przymus, M. Lemanczyk},%
   pdfsubject={2000 Mathematical Subject Classification: Primary:
    37-XX, 37C15, 37C40},%
   pdfkeywords={Cohomological Equations, Homogeneous flows},%
   pdfcreator={pdfLaTeX},%
   pdfproducer={LaTeX with hyperref}%
}
\usepackage{mathrsfs,bbm}
\usepackage{tikz}
\usepackage{tikz-cd}
\usetikzlibrary{matrix,arrows,decorations.pathmorphing}
\providecommand{\noopsort}[1]{}

\newtheorem{Th}{Theorem}[section]
\newtheorem{Prop}[Th]{Proposition}
\newtheorem{Lemma}[Th]{Lemma}

\newcounter{ABC}

\newtheorem{ThA}[ABC]{Theorem}
\newtheorem*{Th*}{Theorem}
\newtheorem{CorA}[ABC]{Corollary}
\newtheorem{PropA}[ABC]{Proposition}
\newtheorem{Cor}[Th]{Corollary}

\theoremstyle{definition}
\newtheorem{Remark}[Th]{Remark}
\newtheorem{Def}{Definition}[section]
\newtheorem{Example}{Example}[section]

\def\scalar(#1,#2){(#1\mid#2)} %\newcommand{\Bar}{\overline}

\newcommand{\raz}{\mathbbm{1}}

\def\g{\mathfrak g}
\newcommand{\crr}{{\mathcal R}}
\newcommand{\cs}{{\mathcal S}}
\newcommand{\ca}{{\mathcal A}}
\newcommand{\cb}{{\mathcal B}}
\newcommand{\cc}{{\mathcal C}}
\newcommand{\cd}{{\mathcal D}}

\newcommand{\cf}{{\mathcal F}}

\newcommand{\ch}{{\mathcal H}}
\newcommand{\cm}{{\mathcal M}}

\newcommand{\xbm}{(X,{\mathcal B},\mu)}

\newcommand{\zdk}{(Z,{\mathcal D},\kappa)}
\newcommand{\ycn}{(Y,{\mathcal C},\nu)}
\newcommand{\ot}{\otimes}
\newcommand{\ov}{\overline}
\newcommand{\la}{\lambda}

\newcommand{\bs}{\mathbb{S}}

\newcommand{\Q}{\mathbb{Q}}

\newcommand{\R}{{\mathbb{R}}}
\newcommand{\T}{{\mathbb{T}}}
\def\C{{\mathbb{C}}}
\newcommand{\Z}{{\mathbb{Z}}}
\newcommand{\N}{{\mathbb{N}}}
\newcommand{\E}{{\mathbb{E}}} %\newcommand{\N}{\mbox{\bf N}}
\newcommand{\va}{\varphi}

\def\G{{\mathcal G}}
\newcommand{\mob}{\boldsymbol{\mu}}

\newcommand{\tfs}{T_{\va,\cs}}
\newcommand{\tf}{T_{\va}} %\newcommand{\mob}{\boldsymbol{\mu}}
\newcommand{\bfu}{\boldsymbol{u}}

\newcommand{\<}{{\langle}}
\renewcommand{\>}{{\rangle}} %
\newcommand{\defin}{:=}

\def\SL{\mathrm S\mathrm L}

\def\D{{\mathrm d}}
\def\<{\langle}
\def\>{\rangle}
\def\Ad{%
\operatorname{Ad}
}

\def\Span{%
\operatorname{span}
}

\def\NilMan{\text{\textbf{NilMan}}}
\def\Uaff{\text{\textbf{MpUniAff}}}
\def\Nilover{\text{\textbf{FiberedMpNilfl}}}
\def\Ab{\operatorname{\textit{\textbf{Ab}}}}
\def\Susp{\operatorname{\textit{\textbf{Susp}}}}

\numberwithin{equation}{section}

\newcommand{\supp}{%
\operatorname{supp}
}

\let\oldmarginpar
\marginpar
\renewcommand
\marginpar[1]{\-\oldmarginpar[
\raggedleft
\footnotesize #1]%
{%
\raggedright
\footnotesize #1}}

\begin{document}

\title[AOP for ergodic affine unipotent diffeomorphisms on
nilmanifolds]{Approximate orthogonality of powers for ergodic affine
  unipotent diffeomorphisms on nilmanifolds} \author{L.\ Flaminio}
\address[L. Flaminio]{Unit\'e Mixte de Recherche CNRS 8524 \\
  Unit\'e de Formation et Re\-cher\-che de Math\'ematiques\\
  Universit\'e de Lille\\
  F59655 Villeneuve d'Asq CEDEX\\
  FRANCE} \email{livio.flaminio@math.univ-lille1.fr}
\urladdr{math.univ-lille1.fr/$\sim$flaminio} % Delete if not wanted.
\thanks{L. Flaminio supported in part by the Labex~CEMPI (ANR-11-LABX-07)}

\author{K.\ Fr\k{a}czek} \address{Faculty of Mathematics and Computer
  Science, Nicolaus Copernicus University, ul. Chopina 12/18, 87-100
  Toruń, Poland} \email{fraczek@mat.umk.pl}
% \urladdr{www.math.sc.edu/$\sim$howard} % Delete if not wanted.

\author{J.\ Ku\l aga-Przymus}
% \address{Faculty of Mathematics and Computer Science, Nicolaus
% Copernicus University, ul. Chopina 12/18, 87-100 Toruń, Poland}
\email{joanna.kulaga@gmail.com}
% \urladdr{www.math.sc.edu/$\sim$howard} % Delete if not wanted.

\author{M.\ Lema\'nczyk}
% \address{Faculty of Mathematics and Computer Science, Nicolaus
% Copernicus University, ul. Chopina 12/18, 87-100 Toruń, Poland}
\email{mlem@mat.umk.pl}
% \urladdr{www.math.sc.edu/$\sim$howard} % Delete if not wanted.
\thanks{M.\ Lema\'nczyk supported by Narodowe Centrum Nauki grant
  2014/15/B/ST1/03736}
\begin{abstract}
  Let~$ G $ be a connected, simply connected  nilpotent Lie
  group and $ \Gamma < G $ a lattice.  We prove that each ergodic
  diffeomorphism~$ \phi(x\Gamma)=uA(x)\Gamma $ on the
  nilmanifold~$ G/\Gamma $, where~$ u\in G $ and~$ A:G\to G $ is a
  unipotent automorphism satisfying~$ A(\Gamma)=\Gamma $, enjoys the
  property of asymptotically orthogonal powers (AOP).  Two
  consequences follow:

  (i) Sarnak's conjecture on Möbius orthogonality holds in every
  uniquely ergodic model of an ergodic affine unipotent
  diffeomorphism;

  (ii) For ergodic affine unipotent diffeomorphisms themselves, the
  Möbius orthogonality holds on so called typical short interval:
  \[
  \frac1M\sum_{M\leq m<2M}\left|\frac1H\sum_{m\leq n<m+H}
    f(\phi^n(x\Gamma))\mob (n)\right|\to 0
  \]
  as~$ H\to\infty $ and~$ H/M\to0 $ for each~$ x\Gamma\in G/\Gamma $
  and each~$ f\in C(G/\Gamma) $.

  \begin{sloppypar}
    In particular, the results in (i) and (ii) hold for ergodic
    nil-translations. Moreover, we prove that each nilsequence is
    orthogonal to the Möbius function $\mob$ on a typical short
    interval.  \end{sloppypar}

  We also study the problem of lifting of the AOP property to induced
  actions and derive some applications on uniform distribution.
\end{abstract}


\maketitle

\tableofcontents

%%%%%%%%%%%%%%%%%%%%%%%%%%%%%%%%%%%%%%%%%%%%%%%%%%%%%%%%%%%%%%%%%%%%%%%%%%%%%%%%%%%%%%%%%%%%%%%%%%%%%%%%%%%%%%%%%%%%%%%%
\section{Introduction} Following \cite{Sa}, we say that a
homeomorphism $T$ of a compact metric space $X$ is {\em Möbius
  orthogonal} if for each~$ f\in C(X) $ and~$ x\in X $, we have
\[
\frac1N\sum_{n\leq N}f(T^nx)\, \mob(n)\to 0 ,\]
as~$ N\to\infty $.  Here~$ \mob:\N\to\C $ stands for the Möbius
function defined by: $ \mob(1)=1 $, $ \mob(n)=(-1)^k $,
if~$ n=p_1p_2\cdots p_k $ for distinct prime
numbers~$ p_1,\ldots,p_k $, and~$ \mob(n)=0 $ otherwise.  In 2010,
P.~Sarnak formulated the following conjecture, referred to as
\emph{Sarnak's conjecture on Möbius orthogonality}:
\begin{quote}
  {Every zero entropy homeomorphism of a compact metric
    space is Möbius orthogonal.}
\end{quote}

\begin{sloppypar}
  Since then, the Möbius orthogonality has been proved in many
  particular cases of zero entropy homeomorphisms: \cite{Ab-Ka-le},
  \cite{Ab-Le-Ru0}, \cite{Bo}, \cite{Bo-Sa-Zi}, \cite{De-Dr-Mu},
  \cite{Do-Ka}, \cite{Fe-Ku-Le-Ma}, \cite{Fe-Ma}, \cite{Gr-Ta},
  \cite{Kar}, \cite{Mu}, \cite{Pe}.
\end{sloppypar}

Sarnak's conjecture is mainly motivated by some open problems in
number theory, for example, the famous Chowla conjecture on
auto-correlations of the Möbius function implies Sarnak's conjecture
(see \cite{Ab-Ku-Le-Ru}, \cite{Sa}, \cite{Ta} for more details).

As stated above, Sarnak's conjecture deals with topological dynamical
systems.  However, we may consider it by looking at all invariant
measures for the homeomorphim~$ T $. (By the variational principle,
all such measures have measure-theore\-tic entropy zero.)  Then we
will deal with an ergodic theory problem considered in the class of
measure-theoretic dynamical systems with zero measure-theoretic
entropy.

\begin{sloppypar}
  Reversing the problem, we may start with an ergodic, zero entropy
  automorphism~$ R $ of a probability standard Borel
  space~$ (Z,\cd,\nu) $ and suppose that~$ T $ is a \emph{uniquely
    ergodic model of~$ R $}, that is, $ T $ is a homeomorphism of a
  compact metric space~$ X $ with a unique~$ T $-invariant probability
  Borel measure~$ \mu $ such that the measure-theoretic
  systems~$ (Z,\cd,\nu,R) $ and~$ (X,\cb (X),\mu,T) $ are
  measure-theoretically isomorphic.  We recall that, by the
  Jewett-Krieger theorem, each ergodic system~$ R $ has many uniquely
  ergodic models, some of which may even be topologically mixing (see
  also \cite{Leh}).  In such a setting it is natural to ask if there
  are some ergodic properties of~$ R $ which are sufficient to
  guarantee that {\bf every} uniquely ergodic model~$ T $ of~$ R $
  will be Möbius orthogonal.
\end{sloppypar}

\begin{sloppypar}
  For example, irrational rotations are well-known to be Möbius
  orthogonal \cite{Da} but the fact that the Möbius orthogonality
  holds in all uniquely ergodic topological models of irrational
  rotations has been proved only recently in \cite{Ab-Le-Ru}.  In
  \cite{Ab-Le-Ru}, such a result has been achieved by introducing a
  new invariant of measure-theoretic isomorphism, namely, the AOP
  property (which, in particular, holds for all irrational
  rotations). The Möbius orthogonality holds true in every uniquely
  ergodic model of a measure-theoretic automorphism~$ R $ satisfying
  the AOP property.
\end{sloppypar}

Let us recall briefly the definition of the AOP property (see
Section~%
\ref{joiAOP}, Definition~%
\ref{defAOP} and Remark~%
\ref{r:weaktop} for more details).  An ergodic automorphism~$ R $ of a
probability standard Borel space~$ (Z,\cd,\nu) $ is said to have {\em
  asymptotically orthogonal powers} (AOP), or to satisfy the \emph{AOP
  property}, if the Hausdorff distance of the set~$ J(R^p,R^q) $ of
joinings between~$ R^p $ and~$ R^q $ from the
singleton~$ \{\nu\ot\nu\} $ goes to zero when~$ p $ and~$ q $ are
different prime numbers tending to infinity.

The AOP property for the automorphism~$ R $ implies that all non-zero
powers of it are ergodic, i.e.~$ R $ is {\em totally ergodic}.

Denoting by~$ \mathscr{P} $ the set of prime numbers, a consequence of
the AOP property is that in each uniquely ergodic model~$ (X,T) $ of
$ (Z,\cd,\nu,R) $, we have
\begin{equation}
  \label{intr1} \limsup_{\substack{p, q\to\infty\\
      p,q\in\mathscr{P}\\
      p\neq q}}\,\, \limsup_{K\to\infty}\left|\frac1{b_{K+1}} \sum_{k\leq
      K}\sum_{b_k\leq n<b_{k+1}}f(T^{pn}x_k)\ov{f(T^{qn}x_k)}\right|=0
\end{equation}
for each increasing sequence of integers~$ (b_k) $ with
$ b_{k+1}-b_k\to\infty $, for each sequence of
points~$ (x_k)\subset X $ and a zero mean function~$ f\in C(X) $.  By
the K\'atai-Bourgain-Sarnak-Ziegler criterion%
\footnote{This criterion (in the form used in \cite{Ab-Le-Ru}) says
  that if~$ (a_n)\subset\C $ is bounded
  and~$ \limsup_{p\neq q\to\infty, p,q\in\mathscr{P}}\limsup_{N\to
    \infty}\left|\frac1N\sum_ {n\leq N}a_{pn}\ov{a}_{q_n}\right|=0 $
  then~$ \frac1N\sum_{n\leq N}a_n\bfu (n)\to 0 $ for each
  multiplicative $ \bfu:\N\to\C $, $ |\bfu|\leq1 $.

  We use this criterion by considering~$ (a_n) $ given by
  $ (f(x_1),\ldots,f (T^{b_1-1}x_1),f(T^{b_1}x_2),\ldots) $.}
(\cite{Bo-Sa-Zi}, \cite{Ka}), we obtain
\begin{equation}
  \label{intr2} \frac1{b_{K+1}}\sum_{k\leq K}\left(\sum_{b_k\leq n<b_{k+1}}f
    (T^nx_k)\bfu(n)\right)\to0,\text{ when }K\to\infty,
\end{equation}
for every multiplicative function~$ \bfu:\N\to\C $ bounded by~$ 1 $
and all~$ (b_k) $, $ (x_k) $ and~$ f $ as above. If we set~$ x_k=x $
for all $ k\geq1 $ and $\bfu=\mob$, then~\eqref{intr2} means that the
Möbius orthogonality holds for $T$.  However, in many concrete
situations, from~\eqref {intr2}, we can deduce the stronger conclusion
$$
\lim_{K\to \infty} \frac1{b_{K+1}}\sum_{k\leq K}\left|\sum_{b_k\leq
    n<b_{k+1}}f(T^nx)\bfu(n)\right|=0,
$$
or, equivalently (see \cite{Ab-Le-Ru}),
$$
\frac1M\sum_{M\leq m<2M}\left|\frac1H\sum_{m\leq
    n<m+H}f(T^nx)\bfu(n)\right|,\text { when }H\to\infty,H/M\to0,
$$
for each zero mean function~$ f\in C(X) $, $ x\in X $ and~$ \bfu $ as
above.

The AOP property has been proved in \cite{Ab-Le-Ru} for so-called
quasi-discrete spectrum automorphisms \cite{Ab}.  This implies the Möbius
orthogonality of all uni\-que\-ly ergodic models of quasi-discrete
spectrum automorphisms. Some uni\-que\-ly ergodic models of such
automorphisms are given by affine transformations on compact abelian
(metric) groups.  For {\bf these} particular models the Möbius
orthogonality has been proved earlier \cite{Gr-Ta}, \cite{Li-Sa}.
Recall that the affine transformations are examples of
distal homeomor\-phisms.%
\footnote{A homeomorphism~$ T $ is distal if the
  orbit~$ (T^nx,T^ny) $, $ n\in\N $, is bounded away from the
  diago\-nal in~$ X\times X $ for each~$ x\neq y $.} Another natural
class of distal (uniquely ergodic) homeomorphisms is given by
nil-tran\-sla\-tions and, more generally, affine unipotent diffeomorphisms
of nilmanifolds.  Recall that if~$ G $ is a connected, simply
connected, nilpotent Lie group, $ \Gamma < G $ is a lattice
and~$ u\in G $, then each rotation $ l_u(x\Gamma):=ux\Gamma $ acting
on~$ G/\Gamma $ is called a {\em nil-translation by~$ u $}.  More
generally, we may consider \emph{affine unipotent} diffeomorphisms
on~$ G/\Gamma $, that is maps of~$ G/\Gamma $ of the
form~$ \phi(x\Gamma):=uA (x)\Gamma $, where~$ A $ is a unipotent
automorphism of~$ G $ such that~$ A(\Gamma)=\Gamma $ and~$ u\in G $.
For such maps, the Möbius orthogonality has been proved in
\cite{Gr-Ta}, where even some quantitative version of it has been
proved.  Liu and Sarnak in \cite{Li-Sa} suggest that perhaps instead
of proving Sarnak's conjecture, it will be easier first to prove the
Möbius orthogonality in the distal case (see \cite{Ku-Le},
\cite{Li-Sa} and \cite{Wang} for some results in this direction).
Instead, we can ask if, assuming total ergodicity, we have the AOP
property.

In the present paper, we confirm that the AOP approach may bring some
fruits by proving the following.
\begin{ThA}
  \label{ThA} Let~$ G $ be a connected, simply connected, nilpotent
  Lie group and~$ \Gamma < G $ a lattice.  Then, every ergodic affine
  unipotent diffeomorphism~$ \phi:G/\Gamma\to G/\Gamma $ has the AOP
  property.
\end{ThA}
By the previous discussion, the following result is immediate:
\begin{CorA}
  \label{CorB}The Möbius orthogonality holds in every uniquely ergodic
  model of any affine unipotent diffeomorphism of a compact connected
  nilmanifold.
\end{CorA}
Moreover, the algebraic structure of the underlying space~$ G/\Gamma $
will let us prove the following.
\begin{CorA}
  \label{CorC} Let~$ G $ be a connected, simply connected, nilpotent
  Lie group and $ \Gamma < G $ a lattice.
  Let~$ \phi:G/\Gamma\to G/\Gamma $ be an ergodic affine unipotent
  diffeomorphism.  Then, for every~$ x\in G$, for every zero mean
  function~$ f\in C(G/\Gamma) $ and for every multiplicative
  function~$ \bfu:\N\to\C $ bounded by~$ 1 $, we have
  \begin{equation}
    \label{eq:zbkr} \frac1M\sum_{M\leq m<2M}\left|\frac1H \sum_{m\leq
        h<m+H}f(\phi^h(x\Gamma))\,\bfu(h)\right|\to 0
  \end{equation}
  as~$ H\to\infty $ and~$ H/M\to0 $.  For~$ \bfu=\mob $ the result
  holds for arbitrary~$ f\in C(G/\Gamma) $.
\end{CorA}

The property expressed by~\eqref{eq:zbkr} will be referred to as the
{\em Möbius orthogonality on typical short interval}.  One more
consequence of Theorem~%
\ref{ThA} is the following sample of a result when this property takes
place.

\begin{PropA}
  \label{PropD} Assume that~$ P\in\R[x] $ is a non-zero degree
  polynomial with irrational leading coefficient.  Then for all
  $ \gamma\in\R\setminus\{0\} $ and $\varrho\in\R$, we have
  \begin{equation}
    \label{eq:zast1} \frac1M\sum_{M\leq m<2M}\left|\frac1H\sum_{m\leq
        h<m+H} e^{2\pi iP([\gamma h+\varrho])}\mob(h)\right|\longrightarrow 0
  \end{equation}
  as~$ H\to\infty $ and~$ H/M\to0 $.
\end{PropA}

\bigskip

Recall that a sequence $(a_n)\subset\C$ is called a {\em nilsequence}
if it is a uniform limit of sequences of the form
$(f(l_u^n(x\Gamma)))$, where $G/\Gamma$ is a compact
nilmanifold.\footnote{$G/\Gamma$ need not be connected and $l_u$ need
  not be ergodic.} Green and Tao \cite{Gr-Ta} proved that all
nilsequences are orthogonal to $\mob$ (their result is
quantitative). We will prove the Möbius orthogonality on typical short
interval:

\begin{ThA}\label{TheoremE} For each nilsequence $(a_n)\subset\C$, we
  have
$$
\frac1M\sum_{M\leq m<2M}\left|\frac1H\sum_{m\leq
    h<m+H}a_h\mob(h)\right|\to 0$$ when $H\to\infty$, $H/M\to0$.
\end{ThA}

As Leibman in \cite{MR2643713} proved that all polynomial
multi-correlation se\-quen\-ces are nilsequences in the Weyl pseudo-metric
(see Section~\ref{pmc}), we obtain the following result.

% Let~$ T $ be an automorphism of~$ (X,{\mathcal B},\mu) $, $
% g_i\in %L^\infty
% (X,\mu) $,
% $ p_i\in\Z[x] $, $ i=1,\ldots,k $.  Then there exists a
% nilsequence~$ (d_n) $ such that
%$$
%\limsup_{N-M\to\infty} \frac1{N-M}\sum_{n=M}^{N-1} \left|d_n-\int_X
%  g_1\circ T^{p_1(n)}\cdot\ldots\cdot g_k\circ T^{p_k(n)}\,d\mu\right
% |=0.
%$$
% \end{Th*}

\begin{CorA}\label{CorollaryF}
  \begin{sloppypar}
    For every automorphism~$ T $ of a probability standard Borel space
    $\xbm$, the polynomial multi-correlations functions are orthogonal
    to the Möbius function on a typical short interval, that is, for
    every~$ g_i\in L^\infty(X,\mu) $, $ p_i\in\Z[x] $,
    $ i=1,\ldots,k $ ($k\geq1$), we have
  \end{sloppypar}
  $$
  \frac1M\sum_{M\leq m<2M}\left| \frac1H\sum_{m\leq h<m+H}
    \mob(h)\int_X g_1\circ T^{p_1(h)}\cdot\ldots\cdot g_k\circ
    T^{p_k(h)}\,d\mu\right|\longrightarrow 0
  $$
  when~$ H\to\infty $ and~$ H/M\to0 $.
\end{CorA}

In some classical cases (e.g.\ when~$ G $ is $ H_3(\R) $), it is
well-known that nil-trans\-la\-tions are time automorphisms of the
suspension flows over affine automorphisms of the torus.  Since affine
automorphisms on tori enjoy the AOP property \cite{Ab-Le-Ru}, a
natural question arises whether the AOP property of an automorphism
implies the AOP property of the suspension flow. This, in fact,
motivates a more general question.  The AOP property can be studied
for actions of (general) abelian groups, see
Section~\ref{liftingAOPia}.  Passing from automorphisms (i.e.\
$ \Z $-actions) to their suspensions ($ \R $-actions) is a particular
case of inducing \cite{Zi}.  Our aim will be to prove the following
result about the ``relative'' AOP property for induced actions.

\begin{PropA}\label{PropositionG}
  Let $ H $ be a closed cocompact subgroup of a locally compact second
  countable abelian group~$ G $. Assume that $H$ has no non-trivial
  compact subgroups. Assume, moreover, that an~$ H $-action~$ \cs $ on
  a probability standard Borel space $\ycn$ has the AOP property.
  Then, for each~$ E,F\in L^2(Y\times G/H)\ominus L^2(G/H) $, we have
$$\lim_{p\neq q, p,q\in\mathscr{P},p,q\to\infty}
\sup_{\kappa\in
  J^e(\widetilde\cs^{(p)},\widetilde\cs^{(q)})}\left|\int_{Y\times
    G/H\times Y\times G/H}E\ot F\,d\kappa\right|=0$$
for the induced~$ G $-action~$ \widetilde{\cs}, $ i.e.\ the induced
$G$-action has the ``relative'' AOP property.
\end{PropA}

An application of this result for~$ H=k\Z $ and~$ G=\Z $ yields the
following:

\begin{CorA}\label{CorollaryH}
  \label{c:km} Assume that~$ T $ is a uniquely ergodic homeomorphism
  of~$ X $, with the unique invariant measure~$ \mu $.  Assume
  moreover, that~$ (X,\mu,T) $ has the AOP property.  Then, for each
  multiplicative function~$ \bfu:\N\to\C $, $ |\bfu|\leq1 $, for
  each~$ k\geq 1 $ and~$ 0\leq j<k $, we have
  \begin{equation}
    \label{in24} \frac1N\sum_{n\leq N}f(T^nx)\bfu(kn+j)\to0
  \end{equation}
  for each~$ f\in C(X) $ of zero mean and each~$ x\in X $.  In
  particular,
  \begin{equation}
    \label{in25} \frac1N\sum_{n\leq N}f(T^nx)\mob(kn+j)\to0.
  \end{equation}
\end{CorA}

\begin{CorA}\label{CorollaryI}
  Let $(a_n)$ be a sequence $\big(f(\phi^n(x\Gamma))\big)$,
  $\big(e^{2\pi iP([\gamma n+\varrho])}\big)$, an arbitrary
  nil-se\-quen\-ce or
  $\big(\int_Xg_1\circ T^{P_1(n)}\cdot\ldots\cdot g_r\circ
  T^{p_r(n)}\,d\mu\big)$
  as in Corollary~\ref{CorC}, Proposition~\ref{PropD},
  Theorem~\ref{TheoremE} and Corollary~\ref{CorollaryF}, respectively.
  Then, for each $k,j\in\N$, we have
  \begin{equation}\label{eq:corI}
    \frac1M\sum_{M\leq m<2M}\left|\sum_{m\leq n<m+H} a_n\mob(kn+j)\right|\to 0
  \end{equation}
  when $H\to\infty$, $H/M\to0$.
\end{CorA}

Finally, we will also give new examples of AOP flows with partly
continuous singular spectra.

The second part of the paper has a form of an appendix in which we
provide some more information on Lie groups, but also we elucidate the
approach to prove the AOP property for nil-translations through Lie
group apparatus as special properties of measure-theoretical distal
automorphisms.

\section{On the AOP property and the Möbius orthogonality}

\subsection{Joinings.  The AOP property}%
\label{joiAOP} Assume that~$ T $ and~$ S $ are ergodic automorphisms
of probability standard Borel spaces $ \xbm $ and~$ \ycn $,
respectively.  A~$ T\times S $-invariant probability measure~$ \rho $
on~$ (X\times Y,\cb\ot\cc) $ is called a {\em joining} of~$ T $
and~$ S $ if the
marginals % $(\pi_X)_\ast(\rho)$, $(\pi_Y)_\ast(\rho)$
of~$ \rho $ on~$ X $ and~$ Y $ are equal to~$ \mu $ and~$ \nu $,
respectively.  The product measure~$ \mu\ot\nu $ is a joining of~$ T $
and~$ S $, called the \emph{product joining}.  In particular, the set
$ J(T,S) $ of joinings of~$ T $ and~$ S $ is not empty.
If~$ \rho\in J(T,S) $ is ergodic for~$ T\times S $, then~$ \rho $ is
called an {\em ergodic joining} and we write~$ \rho\in J^e(T,S) $.
The set~$ J(T,S) $, endowed with the vague topology, is a closed
simplex for the natural affine structure on the space of probability
Borel measures on~$ (X\times Y,\cb\ot\cc) $. Then~$ J^e(T,S) $ is the
set of extremal points of~$ J (T,S) $.

The automorphisms~$ T $ and~$ S $ are called {\em disjoint} if their
only joining is the product joining $ \mu\ot\nu $, (see \cite{Fu}); in
this case we write~$ T\perp S $.

All above notions have obvious generalizations to measure preserving
actions of groups and semi-groups.

The following definition can be introduced for ergodic actions of more
general groups (see Section~\ref{liftingAOPia}); later on, we shall
consider the AOP property for flows.

Let~$ L^2_0\xbm $ stand for the subspace of zero mean function
in~$ L^2\xbm $ and recall that $ \mathscr{P} $ stands for the set of
prime natural numbers.

\begin{Def}[\cite{Ab-Le-Ru}]
  \label{defAOP} A totally ergodic automorphism~$ T $ of a probability
  standard Borel space~$ \xbm $ has {\em asymptotically orthogonal
    powers} (AOP) if for each $ f,g\in L^2_0\xbm $, we have
  \begin{equation}
    \label{cond:aop} \limsup_{\substack{p,q\to\infty\\
        p, q \in \mathscr{P}\\
        p\neq q }}\sup_{\rho\in J^e(T^p,T^q)}\left|\int_{X\times X}f\ot g\,d\rho\right|=0.
  \end{equation}
  In this case, we also say that~$ T $ has the AOP property.
\end{Def}

Clearly, if the prime powers of~$ T $ are pairwise disjoint,
then~$ T $ enjoys the AOP property.  There are, however, other natural
examples, see \cite{Ab-Le-Ru}, \cite{Ku-Le}; in fact, it may even
happen that all non-zero powers of automorphism having AOP are
isomorphic.

\begin{Remark}
  \label{r:weaktop} By definition, $ \rho_n\to\rho $ for the vague
  topology if and only if we have
  $ \int_{X\times Y}f\ot g\,d\rho_n\to \int_{X\times Y}f\ot g\,d\rho $
  for all $ f,g\in L^2_0\xbm $. Then, since~$ L^2_0\xbm $ is
  separable, the vague topology on~$ J(T,S) $ is metrizable.  Then the
  AOP property for an automorphism~$ T $ states that, when~$ p $
  and~$ q $ are distinct large primes, all ergodic joinings (and hence
  all joinings) of~$ T^p $ and~$ T^q $ are uniformly close to product
  measure.

  Moreover, in order to show the AOP property for~$ T $, we only need
  to check the property~\eqref{defAOP} for~$ f,g $ belonging to a linearly dense
  subset in~$ L^2_0\xbm $.
\end{Remark}

Assume that automorphisms~$ T $ and~$ S $ have a {\em common factor},
i.e.\ there exists an ergodic automorphism~$ R $ defined on a
probability space~$ \zdk $ with equivariant factor
maps~$ \pi_{X,Z}:\xbm\to\zdk $ and~$ \pi_{Y,Z}:\ycn\to\zdk $.  Let
$$
\mu=\int_{Z}\mu_z\,d\kappa(z),\;\nu=\int_Z\nu_z\,d\kappa(z)
$$
be the disintegrations of the measures~$ \mu $ and~$ \nu $.  Then
every joining~$ \rho\in J(R,R) $ has a natural extension to a joining
$ \widehat{\rho}\in J(T,S) $ determined by%
\footnote{Recall that, up to an abuse of notation,
  $ \int_Xf\,d\mu_{z_1}=\E (f|\pi_{X,Z}^{-1}(\cd))(z_1) $ and
  similarly~$ \int_Y g\,d\nu_{z_2}=\E (g|\pi^{-1}_{Y,Z}(\cd))(z_2) $.}
\begin{equation}
  \label{rie} \int_{X\times Y}f\ot g\,d\widehat{\rho}=\int_{Z\times Z}\left
    (\int_Xf\,d\mu_{z_1}\int_Y g\,d\nu_{z_2}\right)\,d\rho(z_1,z_2)
\end{equation}
for all~$ f,g\in L^\infty\xbm $.  The joining $ \widehat\rho $ is
called the {\em relatively independent extension} of $ \rho $.  The
joining~$ \widehat{\rho} $ need not be ergodic, even if~$ \rho $ is;
however, the image via $ \pi_{X,Z}\times\pi_{Y,Z} $ of almost every
ergodic component of~$ \widehat{\rho} $ is equal to~$ \rho $.

% When~$S=T$, then~$J(T,T)$ is the set of {\em self-joinings} of~$T$.
% It contains the graph measures~$\mu_R$ for which:
%$$
%\int_{X\times X} f\ot g\,d\mu_{R}=\int_X f\cdot g\circ R\,d\mu$$
% for each~$f,g\in L^\infty\xbm$, where~$R$ is a member of the {\em
% centralizer}~$C(T)$ of~$T$, i.e.\~$R$ is another automorphism of
% $\xbm$ and~$RT=TR$. Since~$T$ is ergodic, $\mu_R\in J^e(T,T)$.  By
% some abuse of notation, $C(T)\subset J^e(T,T)$. It is not hard to
% see that~$C(T)$ is closed in the weak topology on~$J(T,T)$. In fact,
% $C(T)\ni R_n\to R\in C(T)$ if and only if~$f\circ R_n\to f\circ R$
% in~$L^2\xbm$ for each~$f\in L^2\xbm$.

\subsection{A criterion for AOP and Möbius orthogonality on typical
  short intervals}

Theorem~%
\ref{ThA} and Corollaries~%
\ref{CorB} and~%
\ref{CorC} will be proved by showing that affine uni\-potent
diffeomorphisms of compact nilmanifolds satisfy the hypotheses of the
following theorem.

\begin{Th}
  \label{thm:prop} Let~$ T:X\to X $ be a homeomorphism of a compact
  metric space.  Assume that~$ T $ is totally and uniquely ergodic for
  the~$ T $-invariant probability Borel measure~$ \mu $.
  Let~$ \mathscr{C}\subset C(X)\cap L^2_0(X,\mu) $ be a set whose
  linear span is dense in $ L^2_0 (X,\mu) $.
  \begin{itemize}
  \item[(i)] Assume that, for all~$ f_1,f_2\in\mathscr{C} $ and for
    all but a finite number of pairs of distinct prime
    numbers~$ (r,s) $, we have~$ \rho(f_1\otimes \bar{f}_2)=0 $ for
    all ergodic joinings~$ \rho $ of~$ T^r $ and~$ T^s $.

    Then~$ T $ satisfies the AOP property and the Möbius orthogonality
    holds in every uniquely ergodic model of~$ (X,\mu,T) $.

  \item[(ii)] Assume further that for all~$ f\in\mathscr{C} $ and
    all~$ \omega\in \C $ with~$ |\omega|=1 $ there exists a
    homeomorphism $ S:X\to X $ such
    that~$ f(T^n(Sx))=\omega f(T^n x) $ for every $ x\in X $
    and~$ n\in\Z $.

    Then, for every~$ x\in X $, for every every zero
    mean~$ f\in C(X) $ and every multiplicative~$ \bfu:\N\to\C $
    bounded by~$ 1 $, we have
    \begin{equation}
      \label{eq:zbkr1} \frac 1 M\sum_{M\leq m<2M}\left|\frac1H\sum_{m\leq
          n<m+H}f(T^nx)\bfu(n)\right| \longrightarrow 0
    \end{equation}
    when~$ H\to\infty $ and~$ H/M\to 0 $.  If~$ \bfu=\mob $ the result
    holds for arbitrary~$ f\in C(X) $.
  \end{itemize}
\end{Th}

In view of Remark~%
\ref{r:weaktop}, since the set~$ \mathscr{C} $ is linearly dense
in~$ L^2_0(X,\mu) $, Part (i) of the above theorem was proved in
\cite{Ab-Le-Ru} (Theorem 2).  So we only need to prove Part (ii).

\begin{Remark}\label{rem:mrt}  If~$ f $ is constant or more general $f(T^nx)=\exp(ina)$ for some $a\in\R$,
  then \eqref{eq:zbkr1} holds true, when~$ \bfu $ is the Möbius
  function, by Theorem~1.7 in \cite{Ma-Ra-Ta} (see also the discussion
  preceding this theorem).\end{Remark}

\begin{proof}[Proof of Theorem~\ref{thm:prop}]

  Suppose that the hypotheses of (i) and (ii) are satisfied.  Let
  $ (b_k)_ {k\geq 1} $ be an increasing sequence of natural numbers
  such that~$ b_{k+1}-b_k\to+\infty $.

  Fix~$ x\in X $,~$ f\in \mathscr{C}\subset C(X)\cap L^2_0(X,\mu) $
  and a multiplicative function~$ \bfu:\N\to\C $ bounded by~$ 1 $. For
  every~$ k\geq 1 $, let~$ \omega_k\in\C $ be the number of
  modulus~$ 1 $ such that
  \[
  \Big|\sum_{b_k\leq n< b_{k+1}}f( T^nx)\bfu(n)\Big|=
  \omega_k\sum_{b_k\leq n< b_{k+1}}f( T^nx)\bfu(n).
  \]
  By hypothesis, there exists a homeomorphism~$ S_k:X\to X $ such that
  $ f(T^n S_k x)=\omega_kf(T^nx) $ for every~$ n\in\Z $.
  Let~$ x_k:=S_k x $.  Then
  \begin{equation}
    \label{eq:convb0}
    \Big|\sum_{b_k\leq n< b_{k+1}}f( T^nx)\bfu(n)\Big|= \sum_{b_k\leq
      n< b_{k+1}}f( T^nx_k)\bfu(n).
  \end{equation}
  By Theorem~3 in \cite{Ab-Le-Ru}, the AOP property implies that for
  all zero mean~$ f\in C(X) $, for all sequences~$ (x_k)_{k\geq 1} $
  in~$ X $ and for all multiplicative functions~$ \bfu:\N\to\C $
  bounded by~$ 1 $, we have
  \[ \frac1{b_{K+1}}\sum_{k\leq K}\Big(\sum_{b_k\leq n<
    b_{k+1}}f(T^nx_k)\bfu(n)\Big) \to 0 \text{ when }K\to\infty.
  \]
  In view of \eqref{eq:convb0}, it follows that
  \begin{equation}
    \label{eq:convb1} \frac1{b_{K+1}}\sum_{k\leq K}\Big|\sum_{b_k\leq
      n< b_{k+1}}f( T^nx)\bfu(n)\Big|\to0\text{ when }K\to\infty.
  \end{equation}
  As~$ f \in \mathscr{C} $ was arbitrary, it follows that
  \eqref{eq:convb1} holds also for every~$ f\in \Span\mathscr{C} $.

  Let~$ f $ be an arbitrary continuous function on~$ X $ with zero
  mean.  Since the space~$ \Span\mathscr{C} $ is dense in~$ L^2_0(X,\mu) $, for
  every~$ \varepsilon>0 $ there
  exists~$ f_\varepsilon\in \Span\mathscr{C} $ such
  that~$ \int_{X}|f-f_\varepsilon|\,d\mu<\varepsilon $.  Then
  \begin{align*}
    \frac1{b_{K+1}}\sum_{k\leq K} \Big|\sum_{b_k\leq n< b_{k+1}}f(T^nx)\bfu
    (n)\Big| &\leq \frac1{b_{K+1}}\sum_{k\leq K} \Big|\sum_{b_k\leq n<
               b_{k+1}}f_\varepsilon( T^nx)\bfu(n)\Big|\\
             &\quad+\frac1{b_{K+1}}\sum_{ n\leq b_{K+1}}\big|f(T^nx)-f_\varepsilon(T^nx)\big|.
  \end{align*}
  Since~$ T $ is uniquely ergodic and~$ |f-f_\varepsilon| $ is
  continuous, we have
  \[
  \frac1{b_{K+1}}\sum_{ n\leq b_{K+1}}\big|f( T^nx)-f_\varepsilon(
  T^nx)\big|\to \int_{X}|f-f_\varepsilon|\,d\mu<\varepsilon\text{ as
  }K\to\infty.
  \]
  It follows that
  \[
  \limsup_{K\to\infty}\frac1{b_{K+1}}\sum_{k\leq K}\Big|\sum_{b_k\leq
    n< b_{k+1}}f( T^nx)\bfu(n)\Big|\leq \varepsilon
  \]
  for each~$ \varepsilon>0 $ which proves \eqref{eq:convb1} for all
  continuous functions~$ f $ with zero mean.  According to
  \cite{Ab-Le-Ru} (see the proof of Theorem~5) this implies
  \eqref{eq:zbkr} for every continuous function~$ f $ with zero mean.
\end{proof}

\begin{Remark}
  \label{rem:riemint} Let~$ T $ be a uniquely ergodic homeomorphism of
  a compact metric space~$ X $ and let~$ \mu $ be the
  unique~$ T $-invariant probability measure.  Let~$ \bfu $ be a
  multiplicative function bounded by~$ 1 $.  Suppose that the
  conclusion of Theorem~%
  \ref{thm:prop}~(ii) holds true: for each continuous
  functions~$ f:X\to\C $ with zero mean and~$ x\in X $, we have
  \begin{equation}
    \label{eq:sumhm} A(f,M,H):=\frac1M\sum_{M\leq m<2M}\left|\frac1H\sum_
      {m\leq n<m+H}f(T^nx)\bfu(n)\right| \longrightarrow 0
  \end{equation}
  as~$ H\to\infty $ and~$ H/M\to 0 $.  Then the approximation argument
  used in the proof of Theorem~%
  \ref{thm:prop}~(ii) yields the validity of~\eqref{eq:sumhm} for a
  more general class of functions.

  Let~$ D(X) $ be the space of bounded measurable
  functions~$ f:X\to\C $ such that the closure of the set of
  discontinuity point of~$ f $ has zero measure.  First note that
  every~$ f\in D(X) $ satisfies the following equidistribution
  condition
  \begin{equation}
    \label{eq:eqid} \frac{1}{N}\sum_{n\leq N}f(T^nx)\to\int_X f\,d\mu\text
    { for all }x\in X.
  \end{equation}
  Indeed, since the closure~$ F $ of the set of discontinuities
  of~$ f $ has zero measure, for every~$ \varepsilon>0 $ there exists
  an open set~$ F\subset U $ such
  that~$ \mu(\overline{U})<\varepsilon $.  By Tietze's extension
  theorem and the continuity of~$ f $ restricted to~$ X\setminus U $,
  there exists a continuous function~$ f_\varepsilon:X\to\C $ such
  that~$ f(x)=f_\varepsilon(x) $ for~$ x\in X\setminus U $
  and~$ \|f_\varepsilon\|_{\sup}\leq \|f\|_{\sup} $.  Then
  \begin{align*}
    \Big|\frac{1}{N}\sum_{n\leq N}f(T^nx)-\int_X f\,d\mu\Big|&\leq
                                                               \Big|\frac{1}{N}\sum_{n\leq N}f_\varepsilon(T^nx)-\int_X f_\varepsilon\,d\mu\Big|\\
                                                             &\quad+ 2\|f\|_{\sup}\Big(\frac{1}{N}\sum_{n\leq N}\chi_U(T^nx)+\mu
                                                               (U)\Big).
  \end{align*}
  By unique ergodicity,
  $ \frac{1}{N}\sum_{n\leq N}\delta_{T^nx}\to\mu $ weakly.  Therefore,
  \[
  \frac{1}{N}\sum_{n\leq N}f_\varepsilon(T^nx)\to\int_X
  f_\varepsilon\,d\mu
  \]
  and, by the regularity of~$ \mu $,
  \[
  \limsup_{N\to\infty}\frac{1}{N}\sum_{n\leq N}\chi_U(T^nx)\leq
  \limsup_{N\to\infty}\frac{1}{N}\sum_{n\leq N}\delta_{T^nx}(\overline
  {U})\leq \mu(\overline{U})<\varepsilon.
  \]
  It follows that
  \[
  \limsup_{N\to\infty}\Big|\frac{1}{N}\sum_{n\leq N}f(T^nx)-\int_X
  f\,d\mu\Big|\leq 4\varepsilon\|f\|_{\sup}
  \]
  for each~$ \varepsilon>0 $ which proves \eqref{eq:eqid}.

  Finally, for each~$ f\in D(X) $ with zero mean
  and~$ \varepsilon>0 $, we can find a continuous
  function~$ f_\varepsilon $ with zero mean such that
  $ \int_{X}|f-f_\varepsilon|\,d\mu<\varepsilon $.
  Since~$ |f-f_\varepsilon|\in D(X) $, by \eqref{eq:eqid}, we have
  \[
  \limsup_{N\to\infty}\frac{1}{N}\sum_{n\leq
    N}|f(T^nx)-f_\varepsilon(T^nx)| <\varepsilon.
  \]
  \begin{sloppypar}
    The quantity above bounds the asymptotical difference between
    $ A(f,M,H) $ and $ A(f_\varepsilon,M,H)$.  Therefore
    \eqref{eq:sumhm} holds for each~$ f\in D(X) $ with zero mean and
    for arbitrary~$ x\in X $.
  \end{sloppypar}
\end{Remark}

\section{AOP property for nil-translations}

% \begin{sloppypar}
In this section we shall prove that ergodic nil-translations on
compact connected nilmanifolds satisfy the hypotheses (i) and (ii) of
the criterion provided by Theorem~%
\ref{thm:prop}, thereby proving Theorem~%
\ref{ThA} and Corollaries~%
\ref{CorB} and~%
\ref{CorC} for~$ \phi $ a nil-translation on a compact connected
nilmanifold.  % \end{sloppypar}

\subsection{Background on nilpotent Lie groups}%
\label{sec:background}

Let~$ G $ be a connected simply connected~$ k $-step nilpotent Lie
group with Lie algebra~$ \mathfrak g $, and let~$ \Gamma $ be a
lattice in~$ G $.  The quotient~$ M=G/\Gamma $ is then a compact
nilmanifold on which~$ G $ acts on the left by translations. Denote
by~$ \lambda =\lambda_M $ the~$ G $-invariant probability measure on~$
M $ (locally given by a Haar measure of~$ G $).  Let
\[
  \mathfrak g=\mathfrak g^{(1)}\supset\mathfrak g^{(2)}\supset\dots
  \supset \mathfrak g^{(k)}\supset \mathfrak g^{(k+1)}=\{0\},\quad
  \mathfrak g^{(i+1)}=[\mathfrak g,\mathfrak g^{(i)}], \ i=1,\dots, k,
\]
be the descending central series of~$ \mathfrak g $ (with~$ \mathfrak
g^{(k)}\neq\{0\} $) and let
\[
  G=G^{(1)}\supset G^{(2)}\supset\dots \supset G^{(k)}\supset G^{(k+1)}=\{e_G\},
  \ G^{(i+1)}=[G,G^{(i)}], \ i=1,\dots, k,
\]
be the corresponding series for~$ G $.  In this setting, there exists
a strong Malcev basis through the filtration~$ (\mathfrak g^{(i)})_ {i=1}^k
$ strongly based at the lattice~$ \Gamma $, that is a basis~$ X_1,
\dots , X_{\dim \mathfrak g} $ of~$ \mathfrak g $ such that for an
increasing sequence of integers~$ 0=\ell_0< \ell_1 <\dots < \ell_k=\dim
\mathfrak g $ we have:
\begin{enumerate}
  \item
    The elements~$ X_{\ell_{i-1}+1}, X_{\ell_{i-1}+2}, \dots , X_{\dim
    \mathfrak g} $ form a basis of~$ \mathfrak g^{(i)} $;
  \item
    For each~$ i\in \{1, \dots, \dim \mathfrak g\} $ the elements~$ X_
    {i}, X_{i+1}, \dots , X_{\dim \mathfrak g} $ span an ideal of~$
    \mathfrak g $;
  \item
    The lattice~$ \Gamma $ is given by
    \[
      \Gamma=\big\{ \exp (n_1 X_1)\cdots \exp(n_{\dim \mathfrak g} X_{\dim
      \mathfrak g})\mid n_i\in \Z, i=1,\dots, \dim \mathfrak g\big\}.
    \]
\end{enumerate}
It will be convenient to set~$ d_i=\ell_{i}-\ell_{i-1} $ for $ i=1,\ldots,k
$, so that~$ d_1 $ denotes the dimension of the abelianized group~$ G/
[G,G] $ and~$ d_k $ is the dimension of the~$ k $ derived group~$ G^{(k)}
$, which, we recall, is included in the center~$ Z(G) $ of~$ G $.

Since, for all~$ i=1,\dots , k-1 $, the group~$ G^{(i+1)} $ is a
closed normal subgroup of~$ G $, we have natural epimorphisms~$ \pi^ {%
(i)} \colon G\to G/G^{(i+1)} $.  The group $ G^{(i+1)}\cap \Gamma $ is
a lattice in~$ G^{(i+1)} $ (Theorem~2.3 in
\cite[Corollary~1]{Rag}); equivalently, $ G^{(i+1)}\Gamma $ is a
closed subgroup of~$ G $. Moreover, $\pi^ {(i)}(\Gamma)$ is a lattice in  $G/G^{(i+1)}$  and $ M^{(i)}:=G/G^{(i+1)}\Gamma \approx (G/G^{(i+1)})/\pi^ {(i)}(\Gamma)$ is a nilmanifold.  It
follows that the short exact sequences
\[
  0 \hookrightarrow G^{(i+1)} \hookrightarrow G \twoheadrightarrow G/G^
  {(i+1)}\twoheadrightarrow 0, \qquad i=1,\dots,k-1,
\]
induce $ G $-equivariant fiber bundles of nilmanifolds
% \footnote{Observe
% that~$(G/G^{(i+1)})/\pi^{(i)}(\Gamma)$ is homeomorphic
% to~$G/G^{(i+1)}\Gamma$ via a~$G$-equivariant map
% and~$\Gamma G^{(i+1)}/\Gamma$ is homeomorphic
% to~$G^{(i+1)}/(G^{(i+1)}\cap \Gamma)$ via a~$G^{(i+1)}$-equivariant
% map. Both~$G/G^{(i+1)}\Gamma$ and~$\Gamma G^{(i+1)}/\Gamma$ are
% nilmanifolds.}
\begin{equation}
  \label{eq:fibration} \pi^{(i)} \colon M=G/\Gamma \twoheadrightarrow
  M^{(i)}=G/G^{(i+1)}\Gamma, \qquad i=1,\dots,k-1,
\end{equation}
whose fibers are the orbits of~$ G^{(i+1)} $ on~$ G/\Gamma $, hence
homeomorphic to the nil\-ma\-ni\-folds $ G^{(i+1)} / (G^{(i+1)} \cap \Gamma)
$.

Two cases are of most interest for us.  When~$ i=1 $ the base
nilmanifold~$ M^{(1)}=G/[G,G]\Gamma $ is a torus of dimension~$ d_1 $.
As~$ G/[G,G] $ is an abelian group, it is identified by the
exponential map with (the additive group)~$ \mathfrak g/\mathfrak g^{%
(2)}\approx \R^{d_1} $.  Then the vectors
\[
  \bar X_i = X_i \mod \mathfrak g^{(2)}, \qquad i=1,\dots, d_1,
\]
form a set of generators of the lattice~$ \pi^{(2)}(\Gamma)\approx\Gamma/
(G^{(2)}\cap\Gamma) $.  We shall use an additive notation for the
abelian group~$ G/[G,G] $; thus we identify~$ M^{(1)} $ with~$ \R^{d_1}/\Z^
{d_1} $ by means of the basis $ \bar X_i $, ($ i=1,\dots ,d_1 $).

At the opposite extreme, when~$ i=k-1 $, the group~$ G^{(k)} $ is
abelian, thus isomorphic to~$ \R^{d_k} $.
% Since~$ G^{(k)}\cap \Gamma$ is
% a lattice in~$ G^{(k)}$, the fibers of the
% fibration~\eqref{eq:fibration} are tori isomorphic to
% the~$ d_k$-dimensional torus group
% $ \mathbb T^{(k)}:=G^{(k)}/ (G^{(k)}\cap \Gamma)$.
Let us consider the action of~$ G^{(k)} $ by left translations on $
M=G/\Gamma $.  As the stability group of any point in~$ M $ is the
subgroup~$ G^{(k)}\cap \Gamma $, the action of~$ G^{(k)} $ on $ M $
induces a free action of the torus group~$ \mathbb T^{(k)}=G^{(k)}/ (G^
{(k)}\cap \Gamma) $ on~$ M $.  Ssnce the group $ \mathbb T^{(k)} $ acts
transitively on the fibers of the fibration $ \pi^{(k-1)}\colon M\to M^
{(k-1)} $, the map $ \pi^{(k-1)} $ is a principal~$ \mathbb T^{(k)} $-bundle.

It follows from the above that the Hilbert space~$ L^2(M,\lambda) $
decomposes into a sum of~$ G $-invariant orthogonal Hilbert subspaces
\[
  L^2(M,\lambda)=\bigoplus_{\chi \in \widehat{ {\mathbb T}^{(k)}}} H_\chi,
\]
where we have set, for each character~$ \chi $ of the torus~$ \mathbb
T^{(k)} $,
\begin{equation}
  \label{l2} H_\chi:=\{f\in L^2(M) \mid f(zx ) =\chi(z)f(x ), \,\forall
  x \in M,\, \forall z\in G^{(k)}\}.
\end{equation}
Let~$ \lambda^{(k)} $ be the probability Haar measure on~$ \mathbb T^{%
(k)} $.  Denote by~$ C_\chi $ the subspace of continuous functions in~$
H_\chi $.  The linear operator~$ \mathcal F_\chi:  L^2 (M,\lambda)\to
H_\chi $ given by
\[
  \mathcal F_\chi(f)(x)=\int_{\mathbb T^{(k)}}\chi(z)f(z^{-1}x)\,d\lambda^
  {(k)}(z)
\]
is the orthogonal projector on $ H_\chi $.  It follows that
\begin{Lemma}\label{lem:densc}
  The space $C_\chi$ is dense in $H_\chi$.
\end{Lemma}

\subsection{Dynamics and joinings for nil-translations}%
\label{sec:dynamics-joinings} For~$ u\in G $, let~$ l_u\colon M\to M $
be the left translation by~$ u $.  Then~$ l_u $ is a
measure preserving automorphism of~$ (M,\lambda) $, which for
simplicity we call a \emph{nil-translation}.

By the normality of~$ G^{(i+1)} $ in~$ G $, the projection~$ \pi^{(i)}
:M\to M^{(i)} $ intertwines the map $ l_u \colon M\to M $ with the
nil-translation~$ l_{\pi^{(i)}(u)} \colon M^{(i)}\to M^{(i)} $.

The celebrated theorem by Auslander, Green and Hahn (see
\cite{Au-Gr-Ha}), states that the nil-translation~$ l_u $ is ergodic
(minimal and uniquely ergodic) if and only if the
translation~$ l_{\pi^ {(1)}(u)} $ by~$ \pi^{(1)}(u) $ on the
torus~$ M^{(1)} $ is ergodic (minimal and uniquely ergodic).  The
latter condition is equivalent to saying that the rotation
vector~$ \alpha=(\alpha_1, \dots, \alpha_{d_1})\in \R^{d_1}/\Z^ {d_1}
$ defined by
\[
  \pi^{(1)}(u) = \exp (\alpha_1 \bar X_1+ \dots+ \alpha_{d_1}\bar X_{d_1})
\]
is irrational with respect to the lattice~$ \Z^{d_1} $;
equivalently, that the real numbers $ 1, \alpha_1', \dots, \alpha_{d_1}'
$ are linearly independent over~$ \Q $, for any lift~$ (\alpha_1',
\dots, \alpha_{d_1}') $ of~$ \alpha $ to~$ \R^{d_1} $.  Using an
additive notation we denote the translation~$ l_{\pi^{(1)} (u)}
$ of the torus~$ M^{(1)} $ by~$ R_{\alpha} $, as it is uniquely
determined by the rotation vector~$ \alpha $.  With a slight abuse of
language, and coherently with the choice of using an additive notation
for the torus group~$ M^{(1)} $, we shall then say that the
projection~$ \pi^{(1)} $ intertwines the translation~$ l_u $ with the
rotation~$ R_\alpha $.

Let~$ l_u \colon M\to M $ be  the nil-translation (not necessary
ergodic with respect to~$ \lambda $) by $u$ and let~$ \mu $ be an~$ l_u $-invariant
ergodic probability measure on~$ M $.  Let us consider the stabilizer
of the measure~$ \mu $
\[
  \Lambda(\mu)=\{g\in G:(l_g)_*\mu=\mu\}.
\]
Then~$ \Lambda(\mu)<G $ is a closed subgroup of~$ G $.  The following
celebrated theorem of Ratner
\cite[Theorem 1]{Rat} (first proved independently, in the context of
nilflows, by Starkov in
\cite{St} and Lesigne in
\cite{Les}) describes all such ergodic measures.

\begin{Th}[\cite{St}, \cite{Les},\cite{Rat}]
  \label{thm:ratner} If~$ \mu $ is a~$ l_u $-invariant ergodic
  measure on~$ M $ then there exists~$ x\in M $ such that the orbit~$
  \Lambda(\mu) x\subset M $ is closed and it is the topological
  support of~$ \mu $.

  Let~$ U\in \g $ and let~$ \mu $ be a probability measure on~$ M $
  which is invariant and ergodic for the nilflow~$ (l_{u^t})_{t\in \R}
  $ ($ u^t=\exp(tU) $).  Then the stabilizer~$ \Lambda(\mu) $ is a
  connected group and the topological support of~$ \mu $ is an orbit~$
  \Lambda(\mu) x\subset M $.
\end{Th}

From now on, we fix an ergodic translation~$ l_u $ of~$ M $ (with
respect to~$ \lambda $) projecting to the corresponding ergodic
rotation~$ R_\alpha $ of the torus~$ M^{(1)} $.

We now consider the group~$ G\times G $ with the Lie algebra $
\mathfrak g \oplus \mathfrak g $, the lattice~$ \Gamma\times\Gamma $
and the corresponding nilmanifold~$ M\times M= (G\times G)/ (\Gamma\times
\Gamma) $ on which the group~$ G\times G $ acts by left translations.
We have natural projections~$ p_1 $ and~$ p_2 $ of~$ M\times M $ onto~$
M $, by selecting the first or the second coordinate of a point $ (x_1,x_2)\in
M\times M $, respectively.

Then we have~$ (G\times G)^{(i)} = G^{(i)}\times G^{(i)} $, $ (\mathfrak
{g}\oplus\g)^{(i)}=\mathfrak g^{(i)}\oplus\mathfrak g^{(i)} $ for all~$
i=1,\ldots, k $.  Clearly $ M^{(i)} \times M^{(i)}= (G\times G)/(G\times
G)^{(i)}(\Gamma\times\Gamma) $ for all~$ i=1,\ldots, k-1 $.

Let~$ r,s\in \N $ be relatively prime numbers.  Let~$ \rho $ be an
ergodic joining of the measure theoretical ergodic systems~$ (M,l_ {u}^r,
\lambda) $ and~$ (M,l_{u}^s, \lambda) $.  We recall that, by
definition, $ \rho $ is a measure on~$ M\times M $, with marginals $ (p_i)_*\rho=\lambda
$, ($ i=1,2 $), which is invariant and ergodic for the product
transformation~$ l_{u}^r \times l_{u}^s $.  The transformation~$ l_u^r\times
l_u^s $ on~$ M\times M $ is the left translation on~$ M\times M $ by~$
(u^r,u^s) $:
\[
  l_u^r\times l_u^s (x_1,x_2)=(u^rx_1,u^sx_2), \quad \forall (x_1,x_2)\in
  M\times M.
\]
Then the projection map~$ \pi^{(1)}\times\pi^{(1)} $ of~$ M\times M $
onto the~$ 2d_1 $-dimensional torus $ M^{(1)}\times M^{(1)}$
intertwines the map~$ l_u^r\times l_u^s $ with the rotation~$ R_ {\alpha}^r\times
R_{\alpha}^s $ of~$ { M}^{(1)} \times{ M}^{(1)} $. Moreover, the image
measure $ \rho^{(1)}:=(\pi^{(1)}\times\pi^{(1)})_*\rho $ is~$ R_{\alpha}^r\times
R_{\alpha}^s $-invariant and ergodic.

\begin{Lemma}
  \label{lem:torus} Let~$R_\alpha$ a minimal rotation of the torus $d$-dimensional
  $\mathbb T^d$ and rotation vector $\alpha$.  Let
  $r,s$ relatively prime integers and let $ \eta $ be an~$ R_{\alpha}^r\times R_{\alpha}^s~
  $-invariant and ergodic probability measure on~$ \mathbb T^d \times
  \mathbb T^d $.  Then  the stabilizer~$ \Lambda(\eta)
  $ of $\eta$ is the group
   \[
    {\mathbb T}_{r,s}=\{(t_1,t_2)\in   {\mathbb T}^d\times {\mathbb T}^d \mid st_1=rt_2\},
  \]
  with Lie algebra
  $ \{(r v, s v)\mid v \in \R^d \} $.
\end{Lemma}
\begin{proof}
  As the map~$ R_{\alpha}^r\times R_{\alpha}^s $ is the rotation~$
  R_{(r\alpha, s\alpha)} $ with rotation vector $ (r\alpha,s\alpha)
  $ of the $2d$ dimensional  torus  $\mathbb T^d \times \mathbb T^d$,
  it preserves the orbits of the closed subgroup  ${\mathbb T}_{r,s}$
  that is, the cosets
  \[
    {\mathbb T}_{r,s,c}=\{(t_1,t_2+c)\in \T^d \times \T^d \mid st_1=rt_2\},
    \quad c\in {\mathbb T}^d.
  \]
  We claim that the action of~$ R_{(r\alpha, s\alpha)} $ on each
  coset~$ {\mathbb T}_ {r,s,c} $ is minimal and uniquely ergodic.  Indeed,
  let~$ a , b \in \Z $ be such that~$ ar+bs=1 $. The
  map~$ I_c\colon {\mathbb T}_{r,s,c}\to {\mathbb T}^d $, given
  by~$ I_c(t_1,t_2+c)=at_1+bt_2 $, is well defined.  Its inverse is
  given by the formula
  $ t\in \T^d\mapsto (rt,st+c)\in {\mathbb T}_{r,s,c} $.  It follows
  that the map~$ I_c $ is a homeomorphism intertwining the action
  of~$ R_{(r\alpha, s\alpha)} $ on~$ {\mathbb T}_{r,s,c} $ with the action
  of~$ R_\alpha $ on~$ {\mathbb T}^d$; it also intertwines also the action
  of $ {\mathbb T}_{r,s}$ on $ {\mathbb T}_{r,s,c} $ with the action of
  ${\mathbb T}^d$ on itself.  Since~$ R_\alpha $ is minimal and uniquely
  ergodic, so is the action of $ R_{(r\alpha, s\alpha)}$ on
  $ {\mathbb T}_{r,s,c} $. Since~$ \eta $ is an ergodic measure for
  $ R_{(r\alpha, s\alpha)}$, there exists a~$c\in {\mathbb T}$ such
  that~$ \supp(\eta)={\mathbb T}_{r,s,c}={\mathbb T}_ {r,s}(0,c) $.  It follows
  that~$\Lambda (\eta)= {\mathbb T}_{r,s} $,
% As the orbits of
%   the groups~$ M^ {(1)}_{r,s} $ and~$ \Lambda(\eta) $ are closed and
%   coincide, their Lie algebras are the same.  Since
%   \[
%     {\mathbb T}_{r,s}=I^{-1}_0{\mathbb T}=\{(rt,st)\in { M}^{(1)} \times{M}^{(1)}\mid
%     t\in {\mathbb T}\},
%   \]
%   the Lie algebra of~$ {\mathbb T}_{r,s} $ is~$ \mathfrak h_{r,s} $,
  which completes the proof.
\end{proof}

Applying the above lemma to our setting we have:
\begin{Cor}
  \label{lem:torus1} Let~$ \eta $ be an~$ R_{\alpha}^r\times R_{\alpha}^s~
  $-invariant and ergodic probability measure on $ M^{(1)}\times M^{%
  (1)} $.  Then the Lie algebra of the stabilizer~$ \Lambda(\eta) $ is t
  \[
    \mathfrak h_{r,s}:=\{(r \bar X, s \bar X) \mid \bar X \in
    \mathfrak g/\mathfrak g^{(2)}\}\subset (\g\oplus \g)/(\g^{(2)}\oplus
    \g^{(2)}).
  \]
\end{Cor}

Denote by~$ H<G\times G $ the stabilizer of the joining~$ \rho $ and
let $ \mathfrak h\subset\g\oplus\g $ be its Lie algebra.

\begin{Lemma}
  \label{lem:rsh} There exist elements~$ X_1' , X_1'' , \dots, X_ {d_1}'
  , X_{d_1}'' $ in~$ \mathfrak g $ satisfying
  \[
    X'_i\equiv X_i''\equiv X_i \mod \mathfrak g^{(2)}, \text{ for all
    } i=1,\dots, d_1
  \]
  and such that
  \[
    \overline{Y}_i=(r X'_i, sX''_i)\in \mathfrak h.
  \]
\end{Lemma}

\begin{proof}
  Let us consider the image measure~$ \rho^{(1)}:=(\pi^{(1)}\times\pi^
  {(1)})_*\rho $ which is~$ R_{\alpha}^r\times R_{\alpha}^s $-invariant
  and ergodic.  By Theorem~%
  \ref{thm:ratner}, the topological support of the measure~$ \rho $
  is a closed coset~$ H\bar x $ for some~$ \bar x\in M\times M $. Then
  the topological support of~$ \rho^{(1)} $ is~$ (\pi^{(1)}\times\pi^
  {(1)})(H) \bar t $ for some~$ \bar t\in M^{(1)}\times M^{(1)} $.

  Since~$ \rho^{(1)} $ is~$ R_{\alpha}^r\times R_{\alpha}^s $-invariant
  and ergodic, by Corollary~%
  \ref{lem:torus1}, the topological support of~$ \rho^{(1)} $ is~$
  \Lambda(\rho^{(1)}) \bar t $ and the Lie algebra of~$ \Lambda(\rho^
  {(1)}) $ is~$ \mathfrak h_{r,s} $.  As the orbits~$ (\pi^{(1)}\times\pi^
  {(1)})(H) \bar t $ and~$ \Lambda(\rho^{(1)}) \bar t $ are closed and
  equal, it follows that the Lie algebras of~$ (\pi^{(1)}\times\pi^ {%
  (1)})(H) $ and~$ \Lambda(\rho^{(1)}) $ are the same and equal to~$
  \mathfrak h_{r,s} $.  Hence the Lie algebra~$ \mathfrak h $ of the
  Lie group~$ H $ projects under the tangent map~$ \mathrm d( \pi^{(1)}\times\pi^
  {(1)}) $ onto the algebra~$ \mathfrak h_{r,s}= \{ (r \bar X, s \bar
  X) \mid \bar X \in \mathfrak g/\mathfrak g^{(2)}\} $.  It follows
  that for every~$ i=1,\ldots, d_1 $ there exists~$ (X_i',X_i'')\in
  \mathfrak h $ such that
  \[
    d( \pi^{(1)}\times\pi^{(1)})(X_i',X_i'')=(r\,d\pi^{(1)}(X_i),s\,d\pi^
    {(1)}(X_i)),
  \]
  which yields the required elements~$ X_1' , X_1'' , \dots, X_ {d_1}'
  , X_{d_1}'' \in \mathfrak g $.
\end{proof}

% The set~$S=\{ \overline Y_i \mid i=1, \dots d_1 \}~$ is a subset of
% the set of generators for the Lie algebra~$\mathfrak h$ since it
% projects onto the
% basis~$\{(r \bar X_i, s \bar X_i) \mid i=1, \dots d_1 \} \subset
% \mathfrak h/[\mathfrak h,\mathfrak h]\cap ( \mathfrak g/\mathfrak
% g^{(2)} \oplus \mathfrak g/\mathfrak g^{(2)})~$
% which is a subset of~$\mathfrak h/[\mathfrak h,\mathfrak h]$.
\begin{Lemma}
  \label{lem:aop_nil:1} The Lie algebra~$ \mathfrak h\subset\g\oplus\g
  $ satisfies
  \[
    \{ (r^k Z, s^k Z )\mid Z\in \mathfrak g^{(k)}\} \subset \mathfrak
    h \cap (\mathfrak g^{(k)}\oplus\mathfrak g^{(k)} ).
  \]
  It follows that the group~$ H $ contains the subgroup
  \[
    L\defin \{ \exp(r^k Z, s^k Z)\mid Z\in\mathfrak g^{(k)}\} <G^{(k)}\times
    G^{(k)}.
  \]
\end{Lemma}
\begin{proof}
  Let~$ X_i',X_i'',\overline{Y}_i $ for~$ i=1,\ldots, d_1 $ be given
  by Lemma~%
  \ref{lem:rsh}.  Consider the~$ k $-fold Lie products of elements of
  the sets~$ \overline S=\{ \overline Y_i \mid i=1, \dots ,d_1 \} $
  and~$ S=\{ X_i \mid i=1, \dots ,d_1 \} $; that is, for every~$ (i_1,\dots,i_k)\in
  \{1,\dots,d_1\}^k $, we set
  \[
    \overline S_{i_1,i_2,\dots,i_k}:=[\overline Y_{i_1},[\overline Y_{i_2},
    \dots, [\overline Y_{i_{k-1}},\overline Y_{i_k}]\dots]]\in \g^{(k)}\oplus\g^
    {(k)}
  \]
  and
  \[
    S_{i_1,i_2,\dots,i_k}:=[X_{i_1},[X_{i_2}, \dots, [X_{i_{k-1}},X_{i_k}]\dots]]\in
    \g^{(k)}.
  \]
  Then, by definition and Lemma~%
  \ref{lem:aop_nil:4}, we have
  \begin{align*}
    \overline S&_{i_1,i_2,\dots,i_k} =[(rX'_{i_1},sX_{i_1}'') ,[(rX'_{i_2},sX''_
    {i_2}) , \dots, [(rX'_{i_{k-1}},sX''_{i_{k-1}}),(rX'_{i_{k}},sX''_
    {i_{k}})]\dots]] \\
    &=\big( [rX'_{i_1},[ rX'_{i_2}, \dots, [rX'_{i_{k-1}},rX'_{i_{k}}]\dots]],
    [sX''_{i_1},[ sX''_{i_2}, \dots, [sX''_{i_{k-1}},sX''_{i_{k}}]\dots]]
    \big)\\
    &=\big( r^k \,[X_{i_1},[ X_{i_2}, \dots, [X_{i_{k-1}},X_{i_{k}}]\dots]]
    , s^k\,[X_{i_1},[ X_{i_2}, \dots, [X_{i_{k-1}},X_{i_{k}}]\dots]]
    \big)\\
    &=\big( r^k S_{i_1,i_2,\dots,i_k} , s^k S_{i_1,i_2,\dots,i_k} \big).
  \end{align*}
  \begin{sloppypar}Since~$ X_1+\g^{(2)},\ldots,X_{d_1}+\g^{(2)} $ is a basis of~$ \g/\g^
  {(2)} $, by Lemma~%
  \ref{lem:aop_nil:2}, the set $ \{X_1,\ldots,X_{d_1}\} $ generates
  the Lie algebra~$ \g $.  Therefore, in view of Lemma~%
  \ref{lem:aop_nil:3}, the family of~$ k $-fold products~$ S_{i_1,i_2,\dots,i_k}~
  $ spans~$ \mathfrak g^{(k)} $.

By Lemma~%
  \ref{lem:rsh}, $ \overline Y_{i}\in\mathfrak h $ for every~$ i=1,\ldots,
  d_1 $.  It follows that
  \[
    \big( r^k S_{i_1,i_2,\dots,i_k} , s^k S_{i_1,i_2,\dots,i_k} \big)=\overline
    S_{i_1,i_2,\dots,i_k}\in\mathfrak h\cap (\mathfrak g^{(k)}\oplus\mathfrak
    g^{(k)} )
  \]
  for every~$ (i_1,\dots,i_k)\in \{1,\dots,d_1\}^k $.  Consequently,
  for all~$ Z\in \mathfrak g^{(k)} $ we have $ (r^kZ,s^kZ)\in\mathfrak
  h\cap (\mathfrak g^{(k)}\oplus\mathfrak g^{(k)} ) $, which completes
  the proof.
\end{sloppypar}
\end{proof}

\begin{Prop}
  \label{prop:aop_nil:1} Let~$ \chi_1 $, $ \chi_2 $ be characters of
  the torus~$ \mathbb T^{(k)} $ so that at least one is non-trivial.
  Then, if~$ \chi_1^{r^k}\not= \chi_2^{s^k} $, for any~$ f_i\in H_{\chi_i}
  $, $ i=1,2 $, and any ergodic joining~$ \rho $ of~$ (M,l_{u}^r,
  \lambda) $ and~$ (M,l_{u}^s, \lambda) $, we have
  \[
    \rho(f_1\otimes \bar f_2)=0.
  \]
\end{Prop}
\begin{proof}
  Recall there exists a closed subgroup~$ H< G\times G $ such that the
  measure~$ \rho $ is~$ H $-invariant and~$ H $ contains the subgroup
  \[
    L= \{ \exp(r^k Z, s^k Z)\mid Z\in\mathfrak g^{(k)}\} < G^{(k)}\times
    G^{(k)}.
  \]
  Since~$ f_1\in H_{\chi_1} $ and~$ f_2\in H_{\chi_2} $, for every~$
  (\exp(r^k Z), \exp(s^k Z))\in L $ and all~$ ( x_1, x_2)\in
  M \times M $ we have
  \begin{align*}
    f_1\otimes \bar f_2&((\exp(r^k Z), \exp(s^k Z))(x_1, x_2))=f_1(\exp
    (r^k Z) x_1)\cdot \bar f_2(\exp(s^k Z) x_2)\\
    & = \chi_1(\exp(r^k Z)) \overline{\chi_2(\exp(s^k Z))}\cdot f_1\otimes
    \bar f_2(x_1, x_2)\\
    & = (\chi_1^{r^k}\cdot\chi_2^{-s^k})(\exp(Z))\cdot f_1\otimes \bar
    f_2(x_1, x_2).
  \end{align*}
  As the probability measure~$ \rho $ is invariant under the action of
  the group~$ L<H $, it follows that
  \[
    \rho(f_1\otimes \bar f_2)=(\chi_1^{r^k}\cdot\chi_2^{-s^k})(z)\rho(f_1\otimes
    \bar f_2)
  \]
  for every~$ z\in G^{(k)}/(G^{(k)}\cap \Gamma)=\mathbb T^{(k)} $. By
  assumption, the character~$ \chi_1^{r^k}\cdot\chi_2^{-s^k} $ is
  non-trivial, so we can find~$ z\in \mathbb T^{(k)} $ with~$ (\chi_1^
  {r^k}\cdot\chi_2^{-s^k})(z)\neq 1 $.  This yields~$ \rho(f_1\otimes
  \bar f_2)=0 $, which completes the proof.
\end{proof}

The following theorem shows that Theorem~%
\ref{thm:prop} applies to nil-translations.  Consequently, Theorem~\ref{ThA}
and Corollaries~\ref{CorB}~and~\ref{CorC} are true for nil-translations.

\begin{Th}
  \label{thm:mainrot} Let~$ l_u $ be an ergodic nil-translation of a
  compact connected nil\-ma\-ni\-fold~$ M=G/\Gamma $.  Then there exists a
  set~$ \mathscr{C}\subset C(M)\cap L^2_0(M,\lambda) $ whose span is
  dense in~$ L^2_0(M,\lambda) $ such that:
  \begin{itemize}
    \item[(i)]
      for any pair~$ f_1,f_2\in\mathscr{C} $, for all but at most one
      pair of relatively prime natural numbers~$ (r,s) $ we have~$
      \rho(f_1\otimes \bar{f}_2)=0 $ for all ergodic joinings~$ \rho $
      of~$ l_u^r $ and~$ l_u^s $;

    \item[(ii)]
      for every~$ f\in\mathscr{C} $ and every~$ \omega\in \C $ with~$ |\omega|=1
      $ there exists an element~$ g\in G $ such that~$ f(l_u^n (l_gx))=\omega f(l_u^n
      x) $ for every~$ x\in M $ and every~$ n\in\Z $.
  \end{itemize}
\end{Th}

\begin{proof}
  The proof is by induction on the class of nilpotency~$ k $ of~$G$.

  If~$ k=1 $ then~$ M $ is the torus~$ \T^{(1)} $
  and~$ H_\chi=\C\chi $ for every character~$ \chi $ of~$ \T^{(1)} $.
  Let~$ \mathscr{C} $ be the set of nontrivial characters.  Then $\mathscr{C}$ is linearly dense in
  $L^2_0(M,\lambda)$.
  Moreover, for every pair of nontrivial characters~$ \chi_1 $, $ \chi_2 $ there
  is at most one pair of relatively prime natural numbers~$ r,s $ such
  that~$ \chi_1^r=\chi_2^s $.  In view of Proposition~%
  \ref{prop:aop_nil:1}, the hypotheses (i) of Theorem~%
\ref{thm:prop} are verified.  Moreover,
  if~$ \chi\in\mathscr {C} $ then for every~$ z\in \T^{(1)} $ we
  have~$ \chi(l_u^n(l_zx))=\chi (l_z(l_u^nx))=\chi(z)\chi(l_u^n x) $
  which verifying the hypotheses (ii)  of  Theorem~%
\ref{thm:prop}.

  Suppose that for every ergodic nil-translation on any compact
  connected nilmanifold of class~$ k-1 $ the required set of
  continuous functions does exist.  Let us consider an ergodic
  nil-translation~$ l_u $ on a compact connected nilmanifold~$ M=G/\Gamma
  $ of class~$ k $.  Then~$ M^{(k-1)} $ is a compact connected
  nilmanifold of class~$ k-1 $ and~$ \pi^{(k-1)}\colon M\to M^{(k-1)}
  $ intertwines~$ l_u $ with the ergodic rotation~$ l_{\pi^{(k-1)}u} $
  on~$ M^{(k-1)} $.  Denote by~$ \mathscr{C}^ {(k-1)} $ the subset of
  continuous functions on~$ M^{(k-1)} $ derived from the induction
  hypothesis.  Next we use
  \[
    L^2(M,\lambda)=\bigoplus_{\chi \in \widehat{ {\mathbb T}^{(k)}}} H_\chi.
  \]
  If~$ \chi=1 $ is the trivial character then~$ H_1 $ consists of~$ G^
  {(k)} $-periodic functions and $ H_1 $ can be naturally identified
  with~$ L^2(M^{(k-1)}) $.  Since the identification is~$ G $-equi\-va\-riant,
  for functions from~$ H_1 $ the dynamics given by the rotation~$ l_u $
  coincides with the dynamic of~$ l_{\pi^{(k-1)}u} $ on~$ M^{(k-1)} $.
  Let
  \[
    \mathscr{C}:=\mathscr{C}^{(k-1)}\cup\bigcup_{\chi \in \widehat{ {\mathbb
    T}^{(k)}}\setminus\{1\}}C_{\chi}.
  \]
 By
  Lemma~\ref{lem:densc} and by the induction hypothesis, the set
  $\mathscr{C}$ is linearly dense in the space $L_0^2(M)$.

  First note that for all functions from the subset~$ \mathscr{C}^{(k-1)}\subset
  \mathscr{C} $ both properties (i) and (ii) follows from the
  induction hypothesis.  To complete (i) we need to take~$ f_1\in H_{\chi_1}
  $ and~$ f_2\in H_{\chi_2} $ with at least one non-trivial character~$
  \chi_1 $, $ \chi_2 $ of~$ \T^{(k)} $.  As there is at most one pair
  of relatively prime natural numbers~$ r,s $ such that~$ \chi_1^{r^k}=\chi_2^
  {s^k} $, Proposition~%
  \ref{prop:aop_nil:1} implies (i).

  Let~$ f\in H_\chi $ for a non-trivial character~$ \chi $.  Since~$
  G^{(k)} $ is a subgroup of the center~$ Z(G) $, the action~$ z\in \T^{(k)}\mapsto l_z $ commutes with the nil-translation~$ l_u $. It
  follows that
  \[
    f(l_u^n(l_zx))=f(l_z(l_u^nx))=\chi(z)f(l_u^n x).
  \]
  The non-triviality of~$ \chi $ completes the proof of (ii).
\end{proof}

\section{AOP for affine diffeomorphisms of nilmanifolds}
\def\g{\mathfrak g}
\def\Aut{%
\operatorname{Aut}
}

% We recall that if~$M=G/\Gamma$ is a compact nilmanifold, with~$G$ a
% simply connected, connected Lie group and~$\Gamma$ a lattice in~$G$,
% an \emph{affine diffeomorphism of~$M$} is a mapping of the form
% $\phi(x) = u A(x)$, $x\in M$. Here~$u$ is an element of~$G$ and~$A$
% an automorphism of~$G$ such that~$A(\Gamma)=\Gamma$.  If we denote
% by~$l_g$ we shall then write~$\phi=l_u \circ A$ with the
% understanding that the latter conditions are satisfied
In this section we shall prove Theorem~\ref{ThA} and Corollaries~\ref{CorB}~and~\ref{CorC} for
zero entropy ergodic affine diffeomorphisms on compact connected
nilmanifolds.  Recall that each such affine diffeomorphism is
unipotent, see \cite{MR0260975} and \cite{MR1758456}.  In fact, we shall show that the hypotheses of Theorem~%
\ref{thm:prop} apply to such maps.

\subsection{On affine diffeomorphisms of nilmanifolds}%

For any  Lie group $G$, an
\emph {af\-fine diffeomorphism of~$ G $} is a mapping of the form~$ g
\mapsto u A(g) $, where~$ u\in G $ and~$ A $ is an automorphism of~$ G
$.

Let~$ l_u $ be the left translation on~$ G $ by an element~$ u\in G $.
Then the above affine map is the composition~$ l_u \circ A $. We shall
however use the shortened notation~$ u A $, whenever convenient.

Let~$ M=G/\Gamma $ be a compact nilmanifold, with~$ \Gamma $ a
lattice in~$ G $ and $G$ a connected, simply connected (nilpotent) Lie
group.  An affine diffeomorphism~$ uA $ of~$ G $ induces a
quotient diffeomorphism of~$ M $ if and only if~$ A(\Gamma)=\Gamma $;
an \emph{affine diffeomorphism of~$ M $} is a map of~$ M $ that arises
as such a quotient.  For simplicity, whenever the
condition~$ A(\Gamma)=\Gamma $ is satisfied, the symbol~$ \phi = uA $
(or~$ \phi = l_u\circ A $) will denote both an affine diffeomorphism
of~$ G $ and the induced quotient affine diffeomorphism of~$ M $.

We recall that the group~$ \Aut(G) $ of automorphisms of~$ G $ is
identified,  via the exponential map, with the
group~$ \Aut(\g) $ of automorphisms of the Lie algebra~$ \g $ of~$ G $;  thus, for~$ A\in \Aut(G)\approx\Aut (\g) $, we
have~$ \exp(A(X))= A(\exp X) $, for all $ X\in \g $.
%By abuse of notation
%we will use the same letter for the group automorphism~$ A:G\to G $
%and the corresponding Lie algebra automorphism~$ A:\g\to\g $.

An affine diffeomorphism~$ uA $ of~$ M $ (or of~$ G $) is \emph{unipotent},
if~$ A:\g\to\g $ is a unipotent automorphism. (Recall that an ergodic
affine diffeomorphism of~$ M $ has zero entropy if and only if it is
unipotent.)\enspace  In this case we can write  $
A=\exp B $, with~$ B:\g\to \g $ a nilpotent derivation of $\g$; by
definition of derivation we have
\begin{equation}
  \label{eq:deriv} B([X,Y])=[BX,Y]+[X,BY]\ \text{ for }\ X,Y\in\g.
\end{equation}

\subsubsection{Some rationality issues}
\label{sec:some-rati-issu}
Let $\Gamma$ be a lattice in $G$. Then $\Gamma$ determines a rational
structure on the Lie algebra $\g$. In fact, by Theorem~5.1.8
in~\cite{Co-Gr}, the vector space $ \g_\Gamma:=\Q $-$ \Span(\log
\Gamma) $ is a Lie algebra over $\Q$ such that~$
\g=\g_\Gamma\otimes_\Q\R $. Indeed, any strong Malcev basis strongly
based on the lattice~$ \Gamma $ is a $ \Q $-basis of~$ \g_\Gamma $.

Recall that a subalgebra~$ \mathfrak h\subset \g $ is \emph{rational (with
respect to $\Gamma$)} if $ (\mathfrak
h\cap \g_\Gamma)\otimes_\Q\R=\mathfrak h $, i.e.~if~$ \mathfrak h $ has a
basis contained in~$ \g_\Gamma $. For example, each~$ \g^ {(i)} $ is a
rational ideal with respect to any lattice of $G$.

By definition, a connected closed subgroup~$ H=\exp(\mathfrak h) $
is \emph{a rational subgroup of~$G$} if the subalgebra~$ \mathfrak h $
is rational.

In view of Theorem~5.1.11 in
\cite{Co-Gr}, a connected closed subgroup~$H$ is a rational subgroup of~$ G
$ if and only if  the intersection subgroup $
H \cap \Gamma $ is a lattice in~$ H $.  If furthermore $H$ is normal in~$G$, then~$ G/\Gamma H $ is a
compact nilmanifold; thus we obtain  a smooth~$ G $-equivariant factor map $
\pi_H:G/\Gamma \to G/\Gamma H $.

If $uA$ is an  affine unipotent diffeomorphism
of $M=G/\Gamma$, since~$ A(\Gamma)=\Gamma $, we have $ A(\g_\Gamma)=\g_\Gamma
$.  As the logarithm of a unipotent automorphims is a rational map, if
$A=\exp B$ we have
$ B(\g_\Gamma)\subset\g_\Gamma $. Hence $B$ is a rational endomorphism of $\g_\Gamma$.

\subsection{Suspensions}%
\label{sec:danis-construction} We consider the usual construction
turning an affine diffeomorphism into a translation~\cite{dani1977}.  Let~$ \phi=uA $ be a
unipotent affine diffeomorphism of the nilmanifold $ M=G/\Gamma $
such that~$ A\neq I $.  Let~$ B $ be the (nilpotent) derivation of~$ \g $ such that~$ A=\exp B $ and let~$ \mathcal A= \{A^t\}_
{t\in\R}$, with $ A^t=\exp t B $, be the one-parameter
subgroup of automorphisms of~$ G $ generated by $B$. (We have
a natural identification $A^t\in \mathcal A\mapsto t\in  \R$.)

The semi-direct product~$ \widetilde G=G\rtimes \mathcal A $ is a simply connected, connected
subgroup of the affine diffeomophisms of~$ G $ with the inherited product rule
defined by
\begin{equation}
  \label{eq:semi-direct}
    (g_1A^{t_1})\cdot (g_2A^{t_2})= g_1 A^{t_1}(g_2)\, A^{t_1+t_2}.
\end{equation}
We regard~$ G $ as a normal subgroup of~$ \widetilde G $ via the
inclusion map~$ l\colon g \in G \mapsto l_g\in \widetilde G $.  Thus we have a
split exact sequence
\begin{equation}
  \label{eq:NilpotentAOP-06-14:1}
   0 \to G \xhookrightarrow{l} \widetilde G \xrightarrow{\check p}  \R
   \to 0
\end{equation}
where $\check p (gA^{t}) = t$, for all $gA^{t}\in \widetilde G$.

The set~$ \widetilde \Gamma=\{\gamma A^n\mid \gamma\in \Gamma, n\in \Z\} $ is a
closed lattice subgroup of~$ \widetilde G $; thus  $ \widetilde M=\widetilde
G/\widetilde \Gamma $ is a compact
nilmanifold. By
formulas~\eqref{eq:NilpotentAOP-06-14:1} and \eqref{eq:semi-direct},
since $\check p(\widetilde \Gamma)=\Z$, the map $\check p$ induces a fibration
(in fact a $G$-bundle)
\begin{equation}
  \label{eq:NilpotentAOP-06-14:2}
p\colon  g A^t\widetilde \Gamma \in \widetilde G/\widetilde \Gamma \mapsto t\in \R/\Z
\end{equation}
with compact
fibers~$ p^{-1}\{t\}=\{gA^t\widetilde \Gamma\mid g\in G\}$
parametrized by $ t\in
\R/\Z $.
Since $ gA^t\widetilde \Gamma= g' A^t\widetilde \Gamma$, for  $t\in \R/\Z $ and
$g,g'\in G$, if and
only if $g  A^t(\Gamma) = g'  A^t(\Gamma)$ we identify $
p^{-1}\{t\}\approx G/A^t(\Gamma)$.  In  particular, $p^{-1}\{0\}\approx
G/\Gamma$, the identification being given by  the embedding~$ i\colon
G/\Gamma\hookrightarrow \widetilde G/\widetilde \Gamma $ defined
by~$ i(g\Gamma)=g \widetilde \Gamma $. Thus we shall consider $G/\Gamma$ as a
subset of  $\widehat G/\widetilde \Gamma $.

Let~$ \widetilde \g $ be the Lie algebra of the group~$ \widetilde G
$.
Since this group is generated by the normal subgroup~$G$ and
by~$\mathcal A=(\exp tB)_{t\in \R}$, the Lie algebra $ \widetilde \g $ is
the semi-direct product~$ \g\rtimes \R B $; in addition to the
commutation rules of elements of Lie algebra $ \g $, we have the rule\footnote{In fact, from~\eqref{eq:semi-direct}, $(\exp t B)\cdot \exp X\cdot (\exp -t B) = (\exp t B) (X)$  for every~$t\in
\R$ and every~$X\in \g$.}
\begin{equation}
  \label{eq:susplie}
  [ B, X] = B(X), \quad \text { for all }X \in \g.
\end{equation}
Let $\widetilde u= uA$. Since the derivation $B$ is nilpotent, the
group~$ \widetilde G $ is
nilpotent. Thus the exponential map of $\widetilde G$ is a
homeomorphism and we can write~$ \widetilde u= \exp (B+v) $
for some~$ v\in \g $. Indeed, since $\g$ is an ideal in
  $\widetilde \g$, by the Baker-Campbell-Hausdorff formula, setting
 $u=\exp   v_1$,  we have $u A= \exp v_1 \cdot \exp B = \exp (B +v)$ for some
  $v\in \g$.
On the compact
nilmanifold~$ \widetilde M=\widetilde G/\widetilde \Gamma $ we consider the
nil-translation~$ l_{\widetilde u} $.  Then
\begin{equation}
  \label{eq:elle_u}
  l_{\widetilde u}(gA^s\widetilde \Gamma)= \widetilde u \cdot (gA^s)\widetilde \Gamma = u A(g)A^{s+1}
  \widetilde \Gamma = u A(g)A^{s} \widetilde \Gamma,
\end{equation}
where in the last line we used the observation that~$ A^{-1} \in
\widetilde \Gamma $.  It follows that the nil-translation~$ l_{\widetilde u} $
preserves the fibres~$ p^{-1}\{t\} $ of the fibration~\eqref{eq:NilpotentAOP-06-14:2}.
By the formula~\eqref{eq:elle_u}, the
homeomorphism ~$ i\colon G/\Gamma\to p^{-1}\{0\}\subset \widetilde G/\widetilde \Gamma $
intertwines the affine diffeomor\-phism $ uA :G/\Gamma \to G/\Gamma $ with the
homeomorphism~$ (l_{\widetilde u})_{|p^{-1}\{0\}}\colon p^{-1}\{0\}\to p^{-1}\{0\} $.
 We shall therefore  identify these two maps.

% the Lie algebra of the group~$ \widetilde G $. It
% is standard that The
% group~$ \widetilde{G} $ has a natural representation in~$ GL(\g) $ given
% by~$ (gA^t)(X)=\Ad_g\circ A^t(X) $ for~$ X\in\g $.  The derivative of
% this representation yields a Lie algebra representation of~$ \widetilde \g $
% in $ \mathfrak{gl}(\g) $.  This representation is generated by the
% derivatives of two subrepresentations~$ g\mapsto \Ad_g\in GL(\g) $ and
% $ t\mapsto A^t\in GL(\g) $.  Therefore, the Lie algebra representation
% of~$ \widetilde \g $ is generated by~$ \ad_X\in \mathfrak{gl}(\g) $ for~$ X\in
% \g $ and~$ B\in \mathfrak{gl}(\g) $. Moreover, by \eqref{eq:deriv}, we
% have~$ [B,\ad_X]=\ad_{B(X)} $.  It follows that the Lie algebra~$ \widetilde
% \g $ can be identified with the semi-direct sum~$ \g\oplus_S \R B $
% endowed with the Lie product
% \[
%   [X_1+t_1B,X_2+t_2B]=[X_1,X_2]+t_1B(X_2)-t_2B(X_1), \quad X_1,X_2\in
%   \g, \,\, t_1,t_2 \in \R.
% \]

Let~$\{\widetilde u^t\}_{t\in \R}  < \widetilde G $ be the
one-parameter group defined by~$ \widetilde u^t=  \exp t (B+v)$.
Again, the Baker-Campbell-Hausdorff formula $\exp t (B+v)\exp (-t B)=u_t\in G$ for every $t\in\R$.
Therefore, $\widetilde u^t= u_t A^t$ for every $t\in\R$.

% The natural
% homomorphism $\widetilde G \to  \mathcal A $, ($gA^t \mapsto A^t$),
Denote by~$ ({\widetilde \phi}_t)_{t\in \R} $ the nilflow on~$ \widetilde
G/\widetilde \Gamma $ defined by the one-parameter group $\{\widetilde u^t\}_{t\in \R}$.  By definition,
$ {\widetilde \phi}_1=l_{\widetilde u} $ on~$ \widetilde G/\widetilde \Gamma $ and $ p\circ {\widetilde \phi}_t=p+t \mod \Z$. From the
invariance of the set $ p^{-1}\{0\}$ under the map $l_{\widetilde u} $ and
since the first return time of any point in $p^{-1}\{0\} $ for the
flow~$ ({\widetilde \phi}_t)_ {t\in \R} $ is equal to~$1$, we conclude that
$p^{-1}\{0\} $ is a Poincar\'e section of the flow~$ ({\widetilde \phi}_t)_
{t\in \R} $ (with constant return time~1). Moreover, $p^{-1}\{t\}= {\widetilde \phi}_t p^{-1}\{0\} $ for every $t\in\R/\Z$.
As the map~$ l_{\widetilde u} $
restricted to~$ p^{-1}\{0\} $ is identified with~$ \phi:M\to M $, we
conclude, by standard arguments, that any $\phi$ invariant measure
$\mu$ on $G/\Gamma$ defines a unique invariant measure $\widetilde \mu$ for
the flow $(\widetilde \phi^t)$ on $\widetilde G/\widetilde \Gamma$ given by
\[
\widetilde \mu(f) = \int_{0}^{1} dt\int_{G/\Gamma} d\mu (x)\, f (\widetilde\phi^t
(i (x)))= \int_{0}^{1} dt\int_{p^{-1}\{t\}} d\widetilde \mu_t (y)\, f (y) .
\]
where $\widetilde \mu_t$ is a probability measure on $p^{-1}\{t\}$ given by $\widetilde \mu_t=\widetilde\phi^t_*i_*\mu$. The above formula shows
that the family of measures $(\widetilde \mu_t)_{t\in[0,1)}$ form the conditional
measures of the measure $\widetilde \mu$ with respect to the projection
$p$.
%Conversely given this family of conditional
%measures for almost all $t\in \R/\Z$ we recover
%\[
%\mu= i^{-1}_*  \check \mu_0, \qquad
%\check \mu_0= \int_{0}^{1}  \phi^{-t}_* \widetilde \mu_t\,  dt.
%\]
\begin{Def}
  The measure preserving
  nilflow $(\widetilde G/\widetilde \Gamma, (\widetilde \phi^t), \widetilde \mu) $ is called
  the \emph{suspension} of the  measure preserving affine unipotent
  diffeomorphism $(G/\Gamma, \phi, \mu)$.
\end{Def}
% We summarize in the following proposition the elements of the
% construction of the suspension of $(G/\Gamma, \phi, \mu)$ which we'll need
% later on.
% \begin{Prop}
%   \label{lem:meas} Let $(\widetilde G/\widetilde \Gamma,i, ({\widetilde \phi}_t) ) $ be the suspension of the
%   affine unipotent diffeomorphism $(G/\Gamma, \phi)$.

% The image
%   $p^{-1}\{0\}=i(G/\Gamma)$ is an embedded sub-nilmanifold of
%   $\widetilde G/\widetilde \Gamma$ and a Poincaré section of the flow   $({\widetilde
%     \phi}_t)$ with constant return time $1$.

%  There is a bijective correspondence between the
%   measure If~$ \widetilde \mu $ is invariant under a
%   nil-translation~$ l_h:\widetilde G/\widetilde \Gamma\to \widetilde G/\widetilde \Gamma $ for
%   some~$ h\in G $ then~$ \mu $ is invariant under the nil-translation~$
%   l_h:G/\Gamma\to G/\Gamma $.
% \end{Prop}

A simple application of these ideas is the following lemma.

\begin{Lemma}
  \label{lem:meas} Let $\phi$ be a unipotent affine diffeomorphims of
  the  nilmanifold $G/\Gamma$ preserving a measure $\mu$.  Let
  $(\widetilde G/\widetilde \Gamma, (\widetilde \phi^t), \widetilde \mu)$ be the measure
  preserving nilflow suspension of $(G/\Gamma, \phi,\mu)$ and $Y$ the
  one-parameter group of $\widetilde G$ generating the flow.  Let
  $\Lambda(\widetilde\mu)< \widetilde G $ be the stabilizer
  of~$ \widetilde\mu $.  Then $\Lambda(\widetilde\mu)= H \rtimes \{\exp t Y\}$
   where $Y=B+v$ and $H$ is a closed subgroup of  $G$ satisfying $\Lambda(\mu)_0 < H
   <\Lambda(\mu)$, where $\Lambda(\mu)<G$  is the stabilizer
  of~$ \mu $ and $\Lambda (\mu)_0$ denotes the connected component
of the identity of $\Lambda (\mu)$.
\end{Lemma}

\begin{proof}
As usual, we identify the fiber $p^{-1}(\{0\})$ of the fiber bundle
$p\colon \widetilde G/\widetilde \Gamma \to \R/\Z$ with $G/\Gamma$ and the measure
$\mu$ with a measure $\widetilde \mu_0$ supported on  $p^{-1}(\{0\})\approx
G/\Gamma$.

\begin{sloppypar}The one-parameter group $ \{\exp t Y\}$ generating the
flow $\widetilde \phi^t$ is clearly contained in $ \Lambda(\widetilde \mu)$; hence
$H:= \Lambda(\widetilde \mu) \cap G$ is a normal closed subgroup of $ \Lambda(\widetilde \mu)$ and
$\Lambda(\widetilde \mu) = H \rtimes \{\exp t Y\}$.
\end{sloppypar}

For each $h\in H$, the left translation
$l_h: \widetilde G/\widetilde \Gamma \to \widetilde G/\widetilde \Gamma$ is a smooth
diffeomorphism fibering over the circle $\R/\Z$, i.e.\ $p\circ l_h=p$ and $l_h(p^{-1}(\{t\}))=p^{-1}(\{t\})$ for $t\in\R/\Z$. Thus, for all $h\in
H$, from
$(l_h)_*\widetilde\mu=\widetilde\mu$ we obtain that
$(l_h)_*\widetilde\mu_t=\widetilde\mu_t$ for almost all $t\in \R/\Z$.
%(Here, we denoted by
%$\widetilde \mu_t$ the conditional measure of $\widetilde \mu$ on the fiber
%$p^{-1}\{t\}$.)\enspace In fact, since we can chose a continuous family of
%conditional measures by setting $\widetilde\mu_t =  \phi^{t}_* \widetilde\mu_0
%=\phi^{t}_* \mu $ (where in the last equality we use the identification
%$G/\Gamma=p^{-1}\{0\}$), and
Since the mapping $(h,t) \in H\times
\R/\Z \to (l_h)_*\widetilde\mu_t = (l_h)_*\phi^{t}_*\widetilde\mu_0$ is continuous, we have $(l_h)_*\widetilde\mu_t=\widetilde\mu_t$ for all
$h\in H$ and \emph{all} $t\in \R/\Z$.
%Hence
In particular,
\[
(l_h)_* \mu =(l_h)_* \widetilde \mu_0 = \widetilde \mu_0 =\mu .
\]
% \[
% (l_h)_* \mu = \int_{0}^{1}  \widetilde \mu_t \circ \widetilde \phi^{-t}  \circ
% l_h  \, d t= \int_{0}^{1}    \widetilde \mu_t \circ
% l_{h_{t}} \circ \widetilde \phi^{-t}   \, d t =  \int_{0}^{1}  \widetilde \mu_t
% \circ \widetilde \phi^{-t} \, d t= \mu.
% \]
This shows that $H$ is contained in the stabilizer $ \Lambda (\mu)$ of
$\mu$.

The measure $\widetilde \mu_0$ on $\widetilde G/\widetilde \Gamma$ is preserved by the time
one map $\widetilde \phi^1$, since this map restricted to $ p^{-1}\{0\} $
coincides with the affine map $\phi$. Hence, for all $h\in  \Lambda
(\mu)$ and all $n\in \Z$, $( \widetilde \phi^n \circ l_h \circ \widetilde
\phi^{-n} )_*\widetilde \mu_0=\widetilde \mu_0$, that is $\exp (n Y) \Lambda (\mu) \exp (-n
Y) =\Lambda (\mu)$.  Since the adjoint action of $\widetilde G$ on $\g$
is algebraic, we obtain  the identity 
\[
\exp (t Y) \Lambda (\mu)_0 \exp (-t
Y) =\Lambda (\mu)_0,\] for all $t\in \R$. It follows that if $h\in \Lambda (\mu)_0$ then $\widetilde \mu_0$ is also
$\widetilde \phi^{-t}\circ l_h\circ\widetilde \phi^{t}$-invariant for every $t\in\R$. Therefore, $(l_h)_*\widetilde\mu_t=\widetilde\mu_t$ for every $t\in\R$ and
$(l_h)_*\widetilde\mu=\widetilde\mu$. Consequently, $\Lambda(\mu)_0 < H$.
\end{proof}

% \subsubsection{More on rationality issues}
% Let
% \[
%   \widetilde{\g} =\widetilde \g^{(1)}\supset\widetilde \g^{(2)}\supset\dots \supset \widetilde
%   \g^{(k)}\supset \widetilde \g^{(k+1)}=\{0\}
% \]
% be the descending central series of~$ \widetilde \g $ (with~$ \widetilde \g^{(k)}\neq\{0\}
% $).  Then
% \begin{equation}
%   \label{eq:comu} \widetilde \g^{(2)}= \g^{(2)}+B\g\text{ and }\widetilde \g^{(i)}=
%   [\g,\widetilde \g^{(i-1)}]+B\widetilde \g^{(i-1)}\text{ for }3\leq i\leq k.
% \end{equation}
% Hence~$ \widetilde \g^{(i)}\subset \g $ for all~$ 2\leq i\leq k $. Since
% each such~$ \widetilde \g^{(i)} $ is an ideal in~$ \widetilde\g $, it follows
% that~$ \widetilde \g^{(i)} $ is a~$ B $-invariant ideal in~$ \g $ and~$ \widetilde
% \g^{(k)}\subset\ker B $.

% Suppose that~$ \mathfrak h $ is a~$ B $-invariant rational ideal in $
% \g $.  Then~$ \mathfrak h'=[\g,\mathfrak h]+B\mathfrak h $ is also a
% rational space.  Indeed, if~$ \{X_i\mid i\in I\} $ is a basis of~$ \g $
% and~$ \{Y_j\mid j\in J\} $ is a basis of~$ \mathfrak h $ and all basis
% elements belong to~$ \g_\Q $, then~$ [X_i,Y_j] $ for~$ i\in I $, $ j\in
% J $ together with~$ B(Y_j) $ for~$ j\in J $ generates~$ \mathfrak h' $
% and all belong to~$ \g_\Q $. Proceeding by induction and using \eqref{eq:comu},
% it follows that~$ \widetilde \g^{(i)} $ is a~$ B $-invariant rational ideal
% in $ \g $ for all~$ 2\leq i\leq k $.

\subsubsection{A bit of categorical thinking}

Nilmanifolds are the objets of a category $\NilMan$ which we could
formalise in the following way: the objects of this category are
pairs~$(G,\Gamma)$ with~$G$ a connected, simply connected nilpotent
Lie group and~$\Gamma$ a lattice in $G$; a morphism~$f$ from~$(G,\Gamma)$
to~$(G',\Gamma')$ is a  smooth homomorphism $f\colon G\to G'$ such that
$f(\Gamma) \subset\Gamma'$. Thus a morphism $f$  determines a  smooth map
$\bar f\colon G/\Gamma\to G'/\Gamma'$.

This category can enriched by adding new structures to an
object~$(G,\Gamma)$; for example, we may add an element $X\in \g$,
(or, equivalently, the one-parameter group $\{\exp t X\}$, or the flow
determined by $\{\exp t X\}$ on $G/\Gamma$). Such category will be
called the \emph{category of nilflows}.

Morphisms for these enriched categories are morphisms $\NilMan$
respecting the additional structures. As another example, we may consider the
\emph{category of measure preserving nilflows}, whose objects are
quadruples consisting of a
nilmanifold $(G,\Gamma)$, an element $X\in \g$ and a probability
measure~$\mu$ on~$G/\Gamma$ invariant by the flow determined by
$\{\exp t X\}$ on $G/\Gamma$. A morphism~$f$ from the measure preserving
nilflow $(G,\Gamma, X, \mu)$
to $(G',\Gamma', X', \mu')$ is a morphism of nilmanifolds
$f\colon (G,\Gamma) \to (G',\Gamma')$ such that $f_* X=X'$ and
$\bar f_* \mu=\mu'$, where, as before, $\bar f$ denotes the quotient map
$\bar f\colon G/\Gamma\to G'/\Gamma'$.

 The reader will be easily define in an analogous way the category
$\Uaff$ of
\emph{measure preserving unipotent affine diffeomorphisms of
  nilmanifolds}.

The interest of these categories lies in the fact that we have already
seen some functorial constructions.

The \emph{abelianization functor $\Ab$} from the category of measure
preserving nilflows to itself associates to each measure preserving
nilflow $F=(G, \Gamma, X, \mu)$ the toral flow
$\Ab(F) = (G/G^{(2)}, \Gamma/\Gamma^{(2)} , \bar X, \bar \mu)$ where
$ \bar X = X + \g^{(2)}$ and $\bar \mu =\pi^{(2)}_* \mu$ is image of
$\mu$ by the projection $\pi^{(2)}\colon G/\Gamma \to G/G^{(2)}\Gamma$
(see \eqref{eq:fibration}).  We define, for any homomorphism
$f\colon G\to G_1$, the abelianized homomorphism $\Ab(f)$ as the
induced homomorphism $G/G^{(2)}\to G_1/G_1^{(2)}$. It is routine to
check that $\Ab$ is a well defined functor from the category of
measure preserving nilflows to itself. The Auslander- Green-Hahn
criterion states that the object $\Ab( F)$ is ergodic if and only if
$F$ is ergodic.

%\begin{sploppypar}
  The suspension construction we discussed above is in fact an
isomorphism of categories.  Let us define the category $\Nilover$ of
\emph{measure preserving nilflows fibering over the circle flow}. The
objects of this category are quadruples $(\widetilde G,\widetilde \Gamma, p, Y, \widetilde \mu)$ where $p$ is a
morphism from nilmanifolds $(\widetilde G,\widetilde \Gamma)$ to the nilmanifold $(\R,
\Z)$, $Y$ is an element in $\widetilde \g$ is such that $p_* Y=d/dt$,
and $\widetilde \mu$ is a  $Y$-invariant measure on $\widetilde G/\widetilde \Gamma$. (For short, we
write such an object as $(p, Y,\widetilde \mu)$ since $(\widetilde G,\widetilde \Gamma)$ is implied by $p$).\enspace Morphisms for this category are
defined in the obvious way (this is in fact a slice category).
%\end{sploppypar}

\begin{sloppypar}
The suspension construction associates to each measure preserving
unipotent affine diffeomorphisms of nilmanifolds
$(G, \Gamma, uA , \mu)$ a nilmanifold
$(\widetilde G, \widetilde \Gamma)$ and a morphism
$\check p\colon (\widetilde G, \widetilde \Gamma) \to (\R, \Z)$
inducing a fiber bundle
$p \colon\widetilde G/\widetilde \Gamma \to \R/\Z$; it further defines
a vector $Y=B+v = \log (uA)\in\widetilde \g $, satisfying
$\pi_*Y = d/dt$ and a $Y$-invariant measure $\widetilde \mu $ on
$\widetilde G/\widetilde \Gamma$. Thus to each object
$(G, \Gamma, uA , \mu)$ in the category $\Uaff$ of measure preserving
unipotent affine diffeomorphisms of nilmanifolds we have associated
the object $(\pi,Y,\widetilde \mu) = \Susp (G, \Gamma, uA , \mu)$ in
the category $\Nilover$ of measure preserving nilflows fibering over
the circle flow.
\end{sloppypar}

% We leave to the reader the
% care of defining the correspondence of morphisms and to check that we
% have actually defined a functor $\Susp$ between these
% categories.

\begin{sloppypar}
  Suppose $f \colon  (G, \Gamma, uA , \mu) \to (G', \Gamma', u'A' ,
\mu')$ is a morphism of  $\Uaff$. By definition $f$ is
a homomorphism $f \colon  G \to G'$ such that $f(\Gamma)\subset
\Gamma'$ and such that the induced map $\bar f  \colon  G/ \Gamma \to
G'/ \Gamma'$ is a morphism of the measure preserving dynamical
systems $(G/\Gamma, uA , \mu) $ and  $(G'/\Gamma', u'A' ,
\mu')$.
Let $(\widetilde
G,\widetilde \Gamma, \check p, Y, \widetilde \mu)$ and $(\widetilde G',\widetilde \Gamma', \check p', Y', \widetilde
\mu')$ be the  suspensions of the systems $(G, \Gamma, uA , \mu)$ and $(G', \Gamma', u'A' ,
\mu')$. By definition $\widetilde G = G \rtimes \{A^t\}$, $Y=\log
uA$ etc. Define $\widetilde f = \Susp (f)$ by setting
\end{sloppypar}
\[
\widetilde f
\colon   \widetilde
G \to \widetilde G',\quad  \widetilde f (g A^t)= f(g) (A')^t \qquad
\text{for all }g\in G.
\]
Since $A^t g A^{-t} = A^t(g)$ and $f\circ A=A'\circ f$, we have
\[
\begin{split}\widetilde f (A^t g A^{-t}) &= \widetilde f ( A^t(g) ) =  f ( A^t(g) ) = (A')^t f(g)\\& = (A')^t f(g)
(A')^{-t} = \widetilde f (A^t) \widetilde f(g) \widetilde f (A^{-t});
\end{split}
\] thus $\widetilde f$ is
a homomorphism. We leave to the reader the care of checking that
$\widetilde f$ is a morphism of $\Nilover$, that is that  $\widetilde f(\widetilde \Gamma) \subset
\widetilde \Gamma'$; that $p'\circ \widetilde f = p$; that $\widetilde f_* Y = Y'$; and that
$\overline {\widetilde f}_* \widetilde \mu = \widetilde \mu '$.

We have showed that $\Susp$ is a functor. The functor $\Susp$ is in
fact an
isomorphism of categories. In fact given $(\check p ,Y,\widetilde \mu)$,
where $\check p\colon (\widetilde G, \widetilde \Gamma) \to (\R, \Z)$ induces a
fiber bundle $p\colon \widetilde G/\widetilde \Gamma \to \R/\Z$, we recover $G$ as
$\ker \check p$ and $\Gamma$ as $\widetilde \Gamma \cap G$; the affine map $uA$
is obtained as the first return map of the flow generated by $Y$ to
the fiber $p^{-1}(\{0\})=G/\Gamma $; finally the measure measure $\mu$
is obtained as the conditional measure of $\mu$ on the fiber
$p^{-1}(\{0\})= G/\Gamma$. If $F\colon (\check p ,Y,\widetilde \mu)  \to
(\check p' ,Y',\widetilde \mu')$ is a morphism of $\Nilover$ from $ \check
p' \circ F = p$ we obtain that $ F(\ker p) \subset \ker p'$; thus $F$
restricted to $G:=\ker p$ is a homomorphism of $G$ into $G':=\ker
p'$. We let to the reader the verification that $F_{|G}$ is a morphism
of $\Uaff$ from $  \Susp^{-1}(\check p ,Y,\widetilde \mu)$   to
$\Susp ^{-1}(\check p' ,Y',\widetilde \mu')$,

Just as for the abelianization functor, the functor $\Susp$ has the
property that the object $(G, \Gamma, uA , \mu)$ is ergodic if and
only if the image $\Susp (G, \Gamma, uA , \mu)$ is ergodic.

For nil-translations the proof of the AOP property was based on the
functorial properties of the abelianization. For affine unipotent maps
it is based on the interplay of the two functors $\Ab$ and $\Susp$;
more precisely on study of the functor
$\Susp^{-1}\circ \Ab \circ \Susp$.

\subsection{Invariant measures for unipotent affine diffeomorphisms}

\begin{sloppypar}
  The construction above gave a correspondence, which preserves
ergodicity, between affine unipotent diffeomorphisms of
nilmanifolds and nilflows. Criteria of ergodicity for these
dynamical systems were given by Parry~\cite{MR0260975} and by Hahn
\cite{MR0155956, MR0164001}, for affine diffeomorphims and, as
previouly mentioned, by Auslander, Green and Hahn for the
nilflows. We shall exploit this correspondence and generalize it to
non-ergodic measures. To simplify matters and to clarify a main point
in the proof let us start examining the simplest case.
\end{sloppypar}

Assume $G=\R^d$ and let $\phi=uA$ be an affine unipotent
diffeomorphism of the a torus $ G/\Gamma $. The
suspension of $ (G/\Gamma, \phi) $ yields a nilmanifold
$ \widetilde G/\widetilde \Gamma $ and a flow $\widetilde \phi^t$. If~$ A=\exp B $, the Lie algebra of~$\widetilde G$ is $\widetilde \g = \R^n \oplus \R B$ with the only commutation
relation~$[B,X]=B(X)$ for any~$X\in \R^n$ (here we identified~$\g$
with~$\R^n$). In particular the class  of nilpotency of~$\widetilde G$
is equal to the class of nilpotency of the endomorphism~$B$.

The first derived subalgebra $[\widetilde \g, \widetilde \g] $ is
therefore the subspace of~$\g$ image of~$B$, denoted~$\mathfrak v$,
and the group~$[\widetilde G, \widetilde G] $ is the subgroup $V$ of~$G$
generated by~$\mathfrak v$. The subspace $\mathfrak v$ is rational
because the automorphism of the torus $A$ maps the lattice $\Gamma$ to
itself (hence it is an element of $\SL_{n}(\Z)$ is a suitable integral
basis for $\Gamma$) and because the~$\log$ of a unipotent automorphism
is a rational map. Hence the orbits of the subgroup $V$ are closed and
the space of orbits~$G/V\Gamma$ is a lower dimensional quotient torus
of the torus~$G/\Gamma$.

The affine map~$\phi=uA$ passes to the quotient torus~$G/V\Gamma$,
(since $\mathfrak v = B(\R^n)$ and
$A(\mathfrak v) = (\exp B) (\mathfrak v) \subset \mathfrak v$),
yielding a quotient affine map $\bar \phi$ of the torus~$G/V\Gamma$ onto
itself. However, since $B$ is nilpotent, the automorphism induced by
$A$ on the quotient torus~$G/V\Gamma$ is the identity automorphism. It
follows that the map $\bar \phi$ is a pure translation by the element
$\bar u$, projection of $u$ in $G/V$.

Coming back to the suspended
nilflow~$(\widetilde G/\widetilde \Gamma, \widetilde \phi^t)$, we know by the
Auslander-Green-Hahn criterion that this flow is ergodic if and only
if the abelianized flow on the torus
$\widetilde G/[\widetilde G, \widetilde G]\widetilde \Gamma$ is. Since
$\widetilde G/[\widetilde G, \widetilde G]\widetilde \Gamma=\widetilde G/V\widetilde \Gamma$, the
abelianized flow is the suspension of the translation $l_{\bar u}$ on
the torus~$G/V\Gamma$. (The latter assertion can be easily verified
going through the steps of the construction of the suspension, but it
is, in fact, a consequence of the functoriality of the
construction.)\enspace

The reader will have no difficulty in generalizing the above discussion
to the case where $G/\Gamma$ is a general connected nilmanifold. The
conclusion is the following lemma.

\begin{Lemma}
  \label{lem:susp} Let $\phi$ be a unipotent affine diffeomorphism of
  the  nilmanifold $G/\Gamma$ preserving a probability measure
  $\mu$.  Let $(\widetilde G/\widetilde \Gamma, (\widetilde \phi^t), \widetilde \mu)$ be the
  measure preserving nilflow suspension of $(G/\Gamma, \phi,\mu)$.

Let
  $F$ be a rational  subgroup of $G$ normal in $ \widetilde G$ (hence
  normal in $G$). Then
  \begin{enumerate}
  \item The affine map $\phi$ project to an affine map $\phi_F$ on
    $G/F\Gamma$, preserving the measure~$\mu_F$ image of $\mu$ by
  the  quotient map $G/\Gamma  \to G/F\Gamma$.

\item The suspension of the unipotent affine diffeomorphim  $(G/F\Gamma, \phi_F,\mu_F)$ is the flow
  $(\widetilde G/F\widetilde \Gamma, (\widetilde{\phi_F}^t),\widetilde \mu_F)$, where $(\widetilde{\phi_F}^t)$,
  $\widetilde \mu_F$ are the images of $(\widetilde \phi^t)$,
  $\widetilde \mu$ by the quotient map  $\widetilde
  G/\widetilde \Gamma \to \widetilde
  G/F\widetilde \Gamma$.
  \end{enumerate}
% Let
%   $\nu$ and $\widetilde \mu$ the measure image of $\mu$ and $\widetilde \mu$ bay
%   the projections $G/\Gamma  \to G/H\Gamma$ and $\widetilde
%   G/H\widetilde \Gamma$. Then
\end{Lemma}

We remark that the above lemma applies to the case
$F=[\widetilde G, \widetilde G^{(i)}]$. In fact we have:

\begin{Lemma}
  \label{lem:rationcomm}
  %Assume the notations of Lemma~\ref{lem:susp}.
  The commutator $[\widetilde G, \widetilde G]$ is a rational subgroup of
  $G$ normal in $\widetilde G$. In fact
  $[\widetilde \g, \widetilde \g]= [\g, \g]+ B\g$ is a rational sub-algebra
  of~$G$. Furthemore the descending central series
  $[\widetilde G, \widetilde G^{(i)}]$ and $[\widetilde \g, \widetilde \g^{(i)}]$ of
  $\widetilde G$ and of $\widetilde \g$ form a descending series of rational
  normal subgroups of $G$ % normal in $\widetilde G$
  and of rational normal ideals of $\g$% normal in $\widetilde g$
  .
\end{Lemma}

\begin{proof}
  The equality $[\widetilde \g, \widetilde \g]= [\g, \g]+ B\g$ and the inclusion
  $[\widetilde G, \widetilde G]\triangleleft G$ are immediate consequences of the
  definition of suspended flow $(\widetilde G/\widetilde \Gamma, \widetilde \phi^t)$ and
  of the commutation relations~\eqref{eq:susplie}. Since
  $[\widetilde \g, \widetilde \g]$ is a rational sub-algebra of $\widetilde \g$ the
  orbits of $[\widetilde G, \widetilde G]$ are closed in
  $\widetilde G/\widetilde \Gamma$. Thus the intersections of the
  $[\widetilde G, \widetilde G]$ orbits with
  $G/\Gamma\subset \widetilde G/\widetilde \Gamma$ are also closed. It follows that
  $[\widetilde G, \widetilde G]$ is a rational subgroup of
  $G$.

Alternatively we may argue as in
  paragraph~\ref{sec:some-rati-issu} that both $[\widetilde \g, \widetilde g]$ and
  $B(\g)$ are rational sub-algebras of $\widetilde \g$ and use the fact that
  the sum of  rational sub-algebras is rational.

\emph{Mutatis mutandis}, the same arguments apply to $[\widetilde \g, \widetilde
\g^{(i)}]$ and to  $[\widetilde G, \widetilde G^{(i)}]$.
\end{proof}

\begin{Cor}
  \label{cor:susp} Let $\Sigma= ( G, \Gamma, \phi, \mu)$ be a measure
  preserving unipotent affine diffeomorphism with $\phi =uA$. Let
  $(\widetilde G/\widetilde \Gamma, (\widetilde \phi^t), \widetilde \mu)$ be the measure
  preserving nilflow corresponding to the suspension
  $\widetilde \Sigma= \Susp (\Sigma)$.

  \begin{enumerate}
  \item The affine map $\phi$ projects to a translation $\phi_T$ on the torus
    $G/[\widetilde G, \widetilde G]\Gamma$ 
    by the element $\bar u = u[\widetilde G, \widetilde G]$. The
    translation  $\phi_T$   preserves the
    measure~$\mu_T$ image of $\mu$ by the quotient map
    $G/\Gamma \to G/[\widetilde G, \widetilde G]\Gamma$.
\item The suspension of the toral translation   $\Sigma_T:=( G/[\widetilde G,
  \widetilde G]\Gamma, \phi_T,\mu_T)$ is the linear flow
  $(\widetilde G/[\widetilde G,
  \widetilde G]\widetilde \Gamma, (\widetilde{\phi_T}^t),\widetilde \mu_T)$, where $ (\widetilde{\phi_T}^t)$,
  $\widetilde \mu_T$ are the images of $(\widetilde \phi^t)$,
  $\widetilde \mu$ by the quotient map  $\widetilde
  G/\widetilde \Gamma \to \widetilde
  G/[\widetilde G, \widetilde G]\widetilde \Gamma$.
  \end{enumerate}
% Let
%   $\nu$ and $\widetilde \mu$ the measure image of $\mu$ and $\widetilde \mu$ bay
%   the projections $G/\Gamma  \to G/H\Gamma$ and $\widetilde
%   G/H\widetilde \Gamma$. Then
\end{Cor}
\begin{proof}
  The nilmanifold $\widetilde G/[\widetilde G, \widetilde G]\widetilde \Gamma$ is clearly a
  torus. The
  nilmanifold   $G/[\widetilde G, \widetilde G]\Gamma$ is a torus because it
  is a quotient of the torus  $G/[\widetilde G,\widetilde G]\Gamma$. Thus the  nilflow   $(\widetilde{\phi_T}^t)$ is a linear flow.
  The affine unipotent
  map $\phi_T$ is indeed a translation on $G/[G,G]\Gamma$: in fact,
  since $B(\g) \subset [\widetilde\g,\widetilde\g]$ the derivation $B$ is trivial
  on $\g/ [\widetilde\g,\widetilde\g]$, which implies that the automorphims $A$
  projects to the identity automorphism of $G/[\widetilde G, \widetilde
  G]$.

  The other statements of the corollary are proved in
  Lemmata~\ref{lem:susp} and~\ref{lem:rationcomm}.
\end{proof}

\begin{Remark}
  Observe that the nilflow
  $\widetilde \Sigma_T=(\widetilde G/[\widetilde G, \widetilde G]\widetilde \Gamma, (\widetilde{ \phi_T}^t),\widetilde \mu_T)$
  is just the abe\-lia\-nization $\Ab(\widetilde \Sigma)$ of
  $\widetilde\Sigma=\Susp(\Sigma)$ and that
  \[\Sigma_T = \Susp^{-1}(\widetilde \Sigma_T)= \Susp^{-1}(\Ab(
  \Susp(\Sigma))).\]

Thus the main content of the above corollary is the toral rotation
  $\Sigma_T$ is obtained from $\Sigma$ via the quotient morphism
  $G/\Gamma \to G/[\widetilde G, \widetilde G]\Gamma$.
\end{Remark}

\subsection{Joinings of unipotent affine diffeomorphisms of
  nilmanifolds}

Let~$ \phi=uA $, with $ u\in G $
and~$ A\in \Aut(G) $, be an ergodic affine unipotent diffeomorphism of
the connected nilmanifold~$ G/\Gamma $ such that~$ A\neq I $.  Denote
by~$ \lambda $ the uniquely $ \phi $-invariant probability measure
on~$ G/\Gamma $, i.e. the Haar measure on $ G/\Gamma $.

Let~$ r,s\in \N $ be relatively prime positive integers.  Let us
consider the product diffeomorphism~$ \phi^r\times \phi^s $
on~$ G^2/\Gamma^2=(G\times G)/(\Gamma\times \Gamma) $.  The map
$ \phi^r\times \phi^s$ is an affine unipotent diffeomorphism of
$ G/\Gamma\times G/\Gamma$.  Indeed,
$ \phi^r\times \phi^s=u_{r,s}A_{r,s} $
with~$ u_{r,s}=(u_r,u_s)\in G^2 $
and~$ A_{r,s}=A^r\times A^s\in \Aut(G^2) $.  Note
that~$ B_{r,s}:=\log A_ {r,s}=(rB,sB) $, with $B:=\log A$. Thus

\begin{sloppypar}
  Let~$ \rho $ be an ergodic joining of the measure theoretical systems
$ \Sigma_r= (G/\Gamma,\phi^r, \lambda) $
and~$ \Sigma_s=(G/\Gamma,\phi^s, \lambda) $.  By definition $ \rho $
is a measure on~$ G/\Gamma\times G/\Gamma $ with marginals $ \lambda $
on each factor, invariant and ergodic for the
transformation~$ \phi^r\times \phi^s $.  Thus we have an measure
preserving unipotent affine diffeomorphism
$\Sigma_{r,s} =(G^2, \Gamma^2, \phi^r\times \phi^s, \rho)$ and
morphisms $q_i \colon \Sigma_{r,s} \to\Sigma_i$, $(i=r,s)$.
\end{sloppypar}

Let
$\widetilde \Sigma_{r,s}=( \widetilde G_{r,s} , \widetilde \Gamma_{r,s} , p_{r,s},
Y_{r,s}, \widetilde \rho_{r,s})=\Susp ( \Sigma_{r,s})$.
Here $ \widetilde G_{r,s} = G^2 \rtimes \{\exp t B_{r,s}\}$ and
$Y_{r,s} = \log ( \phi^r\times \phi^s) = (rY, sY)$, with
$Y=\log \phi=\log uA= B + v$, for some $v\in \g$. In the sequel we
shall write
$\widetilde \Sigma_{r,s}=( \widetilde G_{r,s}/\widetilde \Gamma_{r,s} , Y_{r,s}, \widetilde
\rho_{r,s})$,
for short.  Let $ \widetilde\g_{r,s} $ denote the Lie algebra
of~$ \widetilde G_{r,s} $. Then $ \widetilde\g_{r,s} = \g^2\oplus\R B_{r,s}$.

For $i=r,s$, let
$\widetilde \Sigma_{i}=(\widetilde G_{i}, \widetilde \Gamma_{i}, p_i, Y_i , \lambda_i)$ be the
suspensions of $\Sigma_i$ and denote by~$\g_i$ the Lie algebra of
$G_i$. Observe that $\lambda_r$ and $\lambda_s$ are the Haar measures
on the corresponding nilmanifolds and that $Y_r=rY$ and $Y_s=sY$.

Note that, since $\widetilde\g_{r,s}$ is the semi-direct product of
$\g^2 =\g\oplus \g$ and the span of the derivation $B_{r,s}= (rB,sB)$
of $\g^2$ we have
\[
  \widetilde\g_{r,s}= (\g\oplus \g)\rtimes \R (rB,sB) = (\g \rtimes \R r B)
  \oplus (\g \rtimes \R s B)= \widetilde \g_r \oplus\widetilde \g_s = \widetilde \g
  \oplus \widetilde \g =\widetilde \g^2,
\]
where $\widetilde\g= \g \rtimes \R B$.
Hence~$ \widetilde G_{r,s}=\widetilde G_r\times\widetilde G_s =\widetilde G\times\widetilde G$,
with $\widetilde G = G \rtimes \{\exp tB\}$.  Let $k$ be the class of
nilpotency of the group $\widetilde G$. Then, for all $i=2, \dots, k$,
we have:
\begin{gather*}
    \widetilde
  \g^{(i)}_{r,s}=\widetilde\g_r^{(i)}\oplus\widetilde\g_s^{(i)}=\widetilde\g^{(i)}\oplus\widetilde\g^{(i)}\subset\g\oplus\g\intertext{and} \widetilde G^{(i)}_{r,s}=\widetilde G_r^{(i)}\times\widetilde
  G_s^{(i)}=\widetilde G^{(i)}\times\widetilde G^{(i)}< G\times G.
\end{gather*}
We have\footnote{Remark that $\widetilde \Sigma_{r,s}$ is not a joining
  of $\widetilde \Sigma_{r}$ and $\widetilde \Sigma_{s}$, since
  $\widetilde G_{r,s}/\widetilde \Gamma_{r,s} \not= \widetilde G_{r}/\widetilde \Gamma_{r} \times \widetilde
  G_{s}/\widetilde \Gamma_{s}$.
  In fact $\dim \widetilde G_{r,s} = 2 \dim G +1 $ and
  $\dim \widetilde G_{r}=\dim \widetilde G_{s}= \dim G +1$; thus
  $\dim \widetilde G_{r,s} \not= \dim \widetilde G_{r}+\dim \widetilde G_{s}$.}:
\[
  \begin{tikzcd}[column sep=small]
    & \Sigma_{r,s}=( (G/\Gamma)^2, \phi^r\times \phi^s , \rho)
    \arrow{dl}[description]{q_r} \arrow{dr}[description]{q_s}
    \arrow[dashed]{dd}[description]{\Susp}
    & \\
    \Sigma_{r}=(G/\Gamma, \phi^r ,
    \lambda)\arrow[dashed]{dd}[description]{\Susp} &
    & \Sigma_{s}=(G/\Gamma, \phi^s , \lambda)\arrow[dashed]{dd}[description]{\Susp}\\
    &\widetilde \Sigma_{r,s}=( \widetilde G_{r,s}/\widetilde \Gamma_{r,s} , Y_{r,s} ,
    \widetilde \rho_{r,s}) \arrow{dl}[description]{\widetilde q_r=\Susp( q_r)}
    \arrow{dr}[description]{\widetilde q_s=\Susp( q_s)}\\
    \widetilde \Sigma_{r}=(\widetilde G_{r}/ \widetilde \Gamma_{r}, rY , \lambda_r) & &
    \widetilde \Sigma_{s}=(\widetilde G_{s}/\widetilde \Gamma_{s}, sY , \lambda_s)
  \end{tikzcd}
\]
By Lemma~\ref{lem:meas} we have:
\begin{Lemma}
  \label{lem:stabil}
  Denote by~\( \widetilde {\mathfrak h}_{r,s}\subset \widetilde\g_{r,s} \)
  and by \( {\mathfrak h}_{r,s} \)
  the Lie algebras of the stabilisers \(\Lambda (\widetilde\rho_{r,s})<\widetilde G_{r,s} \)
  and~\(\Lambda (\rho_{r,s})< G^2 \)
  of the measures~\(\widetilde\rho_{r,s} \) and~\(\rho\). Then
  \begin{equation}
    \label{eq:NilpotentAOP-06-14:3}
    \widetilde {\mathfrak h}_{r,s} =  {\mathfrak h}_{r,s} \rtimes \R Y_{r,s}.
  \end{equation}
\end{Lemma}

By Corollary~\ref{cor:susp}, applying the functor
$\Ab'=\Susp^{-1}\circ \Ab\circ \Susp$ to the triangle $\Sigma_{r,s}$,
$\Sigma_{r}$, $\Sigma_{s}$ we obtain a diagram
\[
  \begin{tikzcd}[column sep=small]
    & \Ab'(\Sigma_{r,s})=( T^2, l_{\bar u}^r\times l_{\bar u}^s ,
    \rho') \arrow{dl}[description]{\Ab'(q_r)}
    \arrow{dr}[description]{\Ab'(q_s)}
    & \\
    \Ab'(\Sigma_{r})=( T, l_{\bar u}^r , \lambda_{\mathbb T}) & &
    \Ab'(\Sigma_{s})=( T, l_{\bar u}^s, \lambda_{\mathbb T})
  \end{tikzcd}
\]
where $ T$ is the torus $G/[\widetilde G,\widetilde G]\Gamma$,
$\lambda_{ T}$ is the Lebesgue/Haar measure on
$G/[\widetilde G,\widetilde G]\Gamma$, $\Ab'(q_i)$ are the projections of
$ T^2= T\times T$ on the corresponding factors, $l_{\bar u}$ is the
left translation by $\bar u=u [\widetilde G,\widetilde G]$ and $\rho'$ is
the projection of the ergodic joining $\rho$ via the map
$(G/\Gamma)^2\to (G/[\widetilde G,\widetilde G]\Gamma)^2$. In particular
the measure $ \rho'$ is a joining of the ergodic and minimal
toral rotations $ \Ab'(\Sigma_{r})$ and $\Ab'(\Sigma_{s})$.

Since the measure $\rho$ is ergodic, the measure $\rho'$, image of
$\rho$ by the morphism $\Sigma_{r,s} \to \Ab'(\Sigma_{r,s})$ is an
ergodic measure for the for the rotation
~$ l^r_{\bar u}\times l^s_{\bar u} $ of~$\mathbb T^2$.  By Lemma~%
\ref{lem:torus}, the stabilizer $\Lambda(\rho')$ of the measure
$\rho'$ is the connected group
\[
  \Lambda(\rho')= \big\{ \exp(rX,sX)\mid X\in \operatorname
  {Lie}(T)=\g/[\widetilde \g,\widetilde\g] \big\}.
\]

Now we focus on the quadrangle of morphisms given by
Corollary~\ref{cor:susp}.
\[
  \begin{tikzcd}[column sep=small]
    \Sigma_{r,s}=( (G/\Gamma)^2, \phi^r\times \phi^s , \rho) \arrow{r}
    \arrow[dashed]{d}[description]{\Susp} & \Ab'(\Sigma_{r,s})=( T^2,
    l_{\bar u}^r\times l_{\bar u}^s , \rho')
    \\
    \widetilde \Sigma_{r,s}=( \widetilde G_{r,s}/\widetilde \Gamma_{r,s} , Y_{r,s} , \widetilde
    \rho_{r,s}) \arrow{r} & \Ab(\widetilde \Sigma_{r,s})=(\widetilde
    G_{r,s}/[\widetilde G_{r,s}, \widetilde G_{r,s}]\widetilde \Gamma_{r,s},
    Y'_{r,s}, \widetilde \rho_{r,s}')
    \arrow[dashed]{u}[description]{\Susp^{-1}}
  \end{tikzcd}
\]
\begin{Lemma}
  \label{lem:aop_nil:12} The stabilizer $\Lambda( \widetilde \rho_{r,s}')$
  of the measure $ \widetilde \rho_{r,s}'$, image of $\widetilde \rho_{r,s}$ by
  the morphism $\widetilde \Sigma_{r,s}\to \Ab(\widetilde \Sigma_{r,s})$,
  is the connected group
  \[
    \Lambda( \widetilde \rho_{r,s}')= \big\{ \exp(rX,sX)\mid X\in \widetilde
    \g/[\widetilde \g,\widetilde\g] \big\}.
  \]
\end{Lemma}

\begin{proof}
  Since the toral flow $ \Ab( \Sigma_{r,s})$ is a suspension of the
  toral rotation $\Ab'(\Sigma_{r,s})$, it follows by elementary
  reasons of by Lemma~\ref{lem:meas} that we have
  \[
    \Lambda( \widetilde \rho_{r,s}')= \Lambda(\rho')\times \{\exp t
    Y'_{r,s}\}= \big\{ \exp(rX,sX)\mid X\in \g/[\widetilde \g,\widetilde\g]
    \big\}\times \{\exp t Y'_{r,s}\},
  \]
  where $ Y'_{r,s} = Y_{r,s}\mod [\widetilde \g,\widetilde \g]^2 $. Thus, $ Y'_{r,s}= (r Y', s Y')$, with
  $Y'= Y \mod [\widetilde \g,\widetilde \g]$. Since the span of $\g$ and of $Y$ is
  the Lie algebra $\widetilde \g$, the statement follows.
\end{proof}

By Lesigne-Starkov-Ratner Theorem~\ref{thm:ratner}, since the measure
$\widetilde\rho_{r,s}$ is ergodic, the stabilizer
$ \Lambda (\widetilde\rho_{r,s})<\widetilde G_{r,s} $ of $\widetilde\rho_{r,s}$ is a
connected closed subgroup of $\widetilde G_{r,s} $ such that the topological
support of~$ \widetilde\rho_{r,s} $ is a~$ \Lambda (\widetilde\rho_{r,s} )$-orbit
in~$ \widetilde G_{r,s}/\widetilde \Gamma_ {r,s} $.  Clearly the same properties hold
for the stabilizer $\Lambda( \widetilde \rho_{r,s}')$ of the
measure~$ \widetilde \rho_{r,s}'$. Since the action of
$\widetilde G_{r,s}/[\widetilde G_{r,s}, \widetilde G_{r,s}]$ on the torus
$\widetilde G_{r,s}/[\widetilde G_{r,s}, \widetilde G_{r,s}]\widetilde \Gamma_{r,s}$ is
locally free, we obtain:

\begin{Lemma}
  \label{lem:aop_nil:120} The projection of the stabilizer
  $ \Lambda (\widetilde\rho_{r,s})<\widetilde G_{r,s} $ onto
  $\widetilde G_{r,s}/[\widetilde G_{r,s}, \widetilde G_{r,s}]$ is the
  stabilizer $\Lambda( \widetilde \rho_{r,s}')$ of the
  measure~$ \widetilde \rho_{r,s}'$.
\end{Lemma}

\begin{Lemma}
  \label{lem:aop_nil:121} The Lie
  algebra~$ \widetilde{\mathfrak h}_{r,s}\subset\widetilde\g\oplus\widetilde\g $
  satisfies
  \[
    \{ (r^k Z, s^k Z )\mid Z\in \widetilde\g^{(k)}\} \subset \widetilde{\mathfrak
      h}_{r,s} \cap (\widetilde\g^{(k)}\oplus\widetilde\g^{(k)} ).
  \]
  It follows that the stabilizer~$ \Lambda(\rho)<G\times G $ of the
  joining~$ \rho $ contains the subgroup
  \[
    L\defin \{ \exp(r^k Z, s^k Z)\mid Z\in\widetilde\g^{(k)}\} <\widetilde
    G^{(k)}\times\widetilde G^{(k)}=\widetilde G^{(k)}_{r,s}.
  \]
\end{Lemma}
\begin{proof}
  The first statement follows from the previous two lemmata and from
  Lem\-ma~\ref{lem:aop_nil:4} as in the proof of Lemma~\ref{lem:aop_nil:1}.
 The second is an application of Lemma~\ref{lem:stabil}.
\end{proof}

Recall that~$ \widetilde G^{(k)}<G $ is a closed normal subgroup in the
center of~$ G $ such that its Lie algebra~$ \widetilde\g^{(k)} $ is a
rational ideal in $ \g $ include in the kernel of~$ B:\g\to\g $. It
follows that $ M^{(k-1)}:=G/\Gamma \widetilde G^{(k)} $ is a nilmanifold
and $ \T^{(k)}:=\widetilde G^{(k)}/(\widetilde G^{(k)}\cap\Gamma) $ is a torus.
The group~$ \T^{(k)} $ acts on~$ G/\Gamma $ by left translation.
Thus, we have an orthogonal decomposition
\[
  L^2(M,\lambda) = \bigoplus_{\chi \in \widehat{\T^{(k)}}} H_\chi
\]
where, for each character~$ \chi $ of the torus~$ \mathbb T^{(k)} $ we
set
\[
  H_\chi= \big\{ f \in L^2(M,\lambda) \mid f(zx)= \chi (z)f(x),\, \forall
  z\in \mathbb T^{(k)}, \,\forall x \in M\big\}.
\]

%Therefore, $ l_{\pi^{(2)}_{r,s}(u_{r,s})} $ is the product rotation
%$ l^r_{\pi^{(2)}(u)}\times l^s_{\pi^{(2)}(u)} $
%on~$ G/\Gamma\widetilde G^{(2)}\times G/\Gamma\widetilde G^{(2)} $.

\begin{Prop}
  \label{prop:aop_nil:2} Let~$ \chi_1 $, $ \chi_2 $ be characters of
  the torus~$ \mathbb T^{(k)} $ so that at least one is non-trivial.
  Then, if~$ \chi_1^{r^k}\not= \chi_2^{s^k} $, for
  any~$ f_i\in H_{\chi_i} $, $ i=1,2 $, and any ergodic
  joining~$ \rho $ of~$ (M,\phi^r, \lambda) $
  and~$ (M,\phi^s, \lambda) $, we have
  \begin{equation}
    \label{eq:zero} \rho(f_1\otimes \bar{f}_2)=0.
  \end{equation}
  Moreover, if~$ f\in H_\chi $ for nontrivial~$ \chi $ then for every
  $ \omega\in \C $ with~$ |\omega|=1 $ there exists~$ g\in G $ such
  that
  \[
    f(\phi^n(l_gx))=\omega f(\phi^n x)\text{ for all }x\in M,\ %
    n\in\Z.
  \]
\end{Prop}
\begin{proof}
  Suppose that~$ \chi_1^{r^k}\neq \chi_2^{s^k} $
  and~$ f_1\in H_{\chi_1} $ and~$ f_2\in H_{\chi_2} $.  Let~$ \rho $
  be any ergodic joining of~$ (M,\phi^r,\lambda) $
  and~$ (M,\phi^s,\lambda) $.  By Lemma~%
  \ref{lem:aop_nil:121}, every element of the
  group~$ L<\widetilde G^{(k)}\times \widetilde G^{(k)} $ stabilizes the
  measure~$ \rho $.  It follows that the rotation
  by $ (z^{r^k},z^{s^k})\in \T^{(k)}\times \T^{(k)} $ on~$ M\times M $
  preserves~$ \rho $ for every~$ z\in \T^{(k)} $. Therefore
  \[
    \rho(f_1\otimes
    \bar{f}_2)=(\chi_1^{r^k}\chi_2^{-r^k})(z)\rho(f_1\otimes
    \bar{f}_2),
  \]
  which gives \eqref{eq:zero}.

  Since~$ \widetilde\g^{(k)}\subset \ker B $, for
  every~$ h\in \widetilde G^{(k)} $ we have~$ A(h)=h $, so
  \begin{equation*}
    \phi(l_{h}g\Gamma)=uA(hg)\Gamma=l_{h}\phi(g\Gamma)\text{ for every
    }g\Gamma\in M.
  \end{equation*}
  It follows that, if~$ f\in H_\chi $ for nontrivial~$ \chi $ then for
  every~$ \omega\in \C $ with~$ |\omega|=1 $ there
  exists~$ g\in \widetilde G^{(k)} $ such that
  \[
    f(\phi^n(l_gx))=f(l_g\phi^n(x))=\chi(g) f(\phi^n x)=\omega
    f(\phi^n x)\text{ for all }x\in M,\ n\in\Z,
  \]
  and the proof is complete.
\end{proof}
Finally, proceeding by induction on~$ k $ the clas of
nilpotency of the group $\widetilde G$, we obtain the following
result.
\begin{Th}
  \label{thm:mainaff} Let~$ \phi $ be an ergodic affine unipotent
  diffeomorphism of a compact connected nilmanifold~$ M=G/\Gamma $.
  Then there exists a set~$ \mathscr{C}\subset C(M) \cap L^2_0(M,\lambda) $ whose span is
  dense in~$ L^2_0(M,\lambda) $ such that:
  \begin{itemize}
  \item[(i)] for any pair~$ f_1,f_2\in\mathscr{C} $, for all but at
    most one pair of relatively prime natural numbers~$ (r,s) $ we
    have~$ \rho(f_1\otimes \bar{f}_2)=0 $ for all ergodic
    joinings~$ \rho $ of $ \phi^r $ and~$ \phi^s $;
  \item[(ii)] for every~$ f\in\mathscr{C} $ and any~$ \omega \in \C $
    with~$ |\omega|=1 $ there exists~$ g\in G $ such
    that we have $ f(\phi^n (l_gx))=\omega f(\phi^n x) $ for every~$ x\in M $
    and~$ n\in\Z $.
  \end{itemize}
\end{Th}

\begin{proof}
  First note that if~$ \chi$ is the trivial character, then every
  element of~$ H_\chi $ is invariant under the action of~$ \widetilde G^{%
    (k)} $,
  hence it can be treated as an element of $ L^2(M^{(k-1)},\lambda) $.
  Furthermore, we can define the affine unipotent
  diffeomorphism~$ \phi^ {(k-1)}:M^{(k-1)}\to M^{(k-1)} $ by
  \[
    \phi^{(k-1)}(g\Gamma\widetilde G^{(k)})=uA(g)\Gamma\widetilde G^{(k)}.
  \]
  As~$ A(\Gamma)=\Gamma $ and~$ A(\widetilde G^{(k)})=\widetilde G^{(k)} $, this
  map is well defined and is a factor of~$ \phi $ via the
  natural~$ G $-equivariant
  projection~$ \pi^{(k)}:G/\Gamma\to G/\Gamma \widetilde G^{(k)} $.  It
  allows us to proceed by induction.

  The inductive step is given by Proposition~%
  \ref{prop:aop_nil:2}.  Note that Proposition~%
  \ref{prop:aop_nil:2} requires the non-triviality of the
  automorphism~$ A $.  Therefore induction must start from an affine
  diffeomorphism with trivial automorphism, so from a nil-transla\-tion.
  Thus the basis case of induction is derived from Theorem~%
  \ref{thm:mainrot}.  Finally, the construction of the
  set~$ \mathscr {C} $ runs as in the proof of Theorem~%
  \ref{thm:mainrot}.
\end{proof}

\begin{proof}[Proof of Theorem~\ref{ThA} and Corollaries~\ref{CorB}~and~\ref{CorC}]
  Theorem~\ref{ThA} as well as the  Corollaries~\ref{CorB}~and~\ref{CorC} are direct consequences of
  Theorem~%
  \ref{thm:mainaff} and Theorem~%
  \ref{thm:prop}.
\end{proof}

\subsection{Nil-translations on non-connected nilmanifolds} Suppose
that~$ G $ is a nil\-potent Lie group and~$ \Gamma< G $ its lattice.
Then~$ M=G/\Gamma $ is a compact nilmanifold which we assume to be
non-connected.  Let~$ l_u $ for~$ u\in G $ be a uniquely ergodic
translation on~$ M $.  Denote by~$ G_0<G $ the identity component of~$ G $.
The group~$ G_0 $ is closed (and open) normal, and we may assume that
it is simply connected.  Moreover, $ M_0=(G_0\Gamma)/\Gamma\approx G_0/
(G_0\cap\Gamma) $ is a connected component of~$ M $.  Since~$ l_u $ is
minimal and~$ M $ is compact, there is a minimal integer~$ m>1 $ such that~$ M $ is the
disjoint union of clopen sets~$ l^k_uM_0 $ for~$ 0\leq k<m $ and $ l^m_uM_0=M_0
$.  Moreover, $ l^m_u:M_0\to M_0 $ is a uniquely ergodic
diffeomorphism which is a unipotent affine diffeomorphism of the
compact connected nilmanifold~$ M_0= G_0/(G_0\cap\Gamma) $.  Indeed, as $ l^m_u
((G_0\Gamma)/\Gamma)=(G_0\Gamma)/\Gamma $, we have~$ u^m\in G_0\Gamma $,
and hence there are~$ u_0\in G_0 $ and~$ \gamma\in \Gamma $ such that~$
u^m=u_0\gamma $.  It follows that~$ l_u^m=l_{u_0}\circ \Ad_{\gamma} $
on~$ M_0 $, where~$ l_{u_0} $ is a nil-translation on the
nilmanifold~$ G_0/(G_0\cap\Gamma) $ and~$ \Ad_{\gamma}:G_0\to G_0 $
is a unipotent automorphism of~$ G_0 $ with $
\Ad_{\gamma}(G_0\cap\Gamma)=
G_0\cap\Gamma
$.  Therefore~$ l^m_u:M_0\to M_0 $ is an ergodic unipotent affine
diffeomorphism which will be denoted by~$ \phi:M_0\to M_0 $

Let us consider the factor map~$ p:M\to\Z/ m\Z  $ defined by \[
 p(l_u^kx)=k\mod
m\Z, \quad \text{ if }~ k\in\Z \text{ and~}  x\in M_0 .\]  Then the map~$ p $ intertwines
the nil-translation~$ l_u $ on~$ M $ with the rotation on $ \Z/ m\Z  $
by~$ 1 $.  Denote by~$ H^+ $ the space of~$ L^2 (M,\lambda) $ of
functions constant on fibers of~$ p $ and by~$ H^- $ its
ortho\-com\-ple\-ment in~$ L^2(M,\lambda) $.  Then~$ H^- $ coincides with
the space of $ L^2 $-functions which have zero mean on each fiber~$ M_k:=l^k_uM_0
$ for~$ 0\leq k<m $.

\begin{Prop}
  \label{prop:nonconnect} Let~$ l_u $ be an ergodic nil-translation of
  a compact nilmanifold $ M=G/\Gamma $ with $m$ connected components.  Then there exists a set~$
  \mathscr{C}\subset C(M)\cap H^- $ whose span is dense in~$ H^- $
  such that:
  \begin{itemize}
    \item[(i)]
      for any $ f_1,f_2\in\mathscr{C} $, for all but a finite number of pairs of distinct
      prime numbers $(r, s)$, we have~$
      \rho(f_1\otimes \bar{f}_2)=0 $ for all ergodic joinings~$ \rho $
      of~$ l_u^r $ and~$ l_u^s $;
    \item[(ii)]
      for every~$ f\in\mathscr{C} $ and any~$ \omega\in \C $ with~$ |\omega|=1
      $ there exists a homeomorphism~$ S:M\to M $ such that~$ f(l_u^n
      (Sx))=\omega f(l_u^n x) $ for every~$ x\in M $ and~$ n\in\Z $.
  \end{itemize}
\end{Prop}

\begin{proof}

  Let us consider the ergodic unipotent affine
  diffeomorphism~$ \phi:=l_ {u_0}\Ad_\gamma $ on~$ M_0 $. By Theorem~%
  \ref{thm:mainaff} applied to~$ \phi $, there exists a
  set~$ \mathscr{C}_0\subset C(M_0)\cap L^2_0(M_0) $ whose span is
  dense in~$ L^2_0(M_0,\lambda) $ and which satisfies the conditions
  (i) and (ii) of that theorem.  We consider~$ \mathscr{C}_0$ as a set
  of continuous functions on $M$ by extending its elements as $0$ on
  $M\setminus M_0$. For every~$ k\in \Z/ m\Z  $ the
  map~$ l_u^k: M\to M $ establishes a homeomorphism between~$ M_0 $
  and~$ M_k $.
  Set~$ \mathscr{C}_k= \{f\circ l_u^{-k}\mid f\in \mathscr{C}_0\}$.
  Elements of~$ \mathscr{C}_k$ are continuous functions on $M$
  supported on~$ M_k $.
  Then~$ \mathscr{C}:=\bigcup_ {k\in \Z/ m\Z }\mathscr{C}_k $ is a
  linearly dense set in~$ H^- $.

   Fix~$ f\in \mathscr{C}_k $ for some~$ 0\leq k<m $ and~$ \omega\in\C
  $ with~$ |\omega|=1 $.  Then~$ f $ is supported on~$ M_k $ and
  there exists~$ h\in \mathscr{C}_0 $ such that the restriction of~$
  f $ to~$ M_k $ is equal to~$ h\circ l_u^{-k} $.  By (ii) in
  Theorem~%
  \ref{thm:mainaff}, there exists~$ g\in G_0 $ such that~$ h(\phi^n (l_gx))=\omega
  h(\phi^n x) $ for every~$ x\in M_0 $ and~$ n\in\Z $.   Let~$ S:M\to M
  $ be the homeomorphism defined by~$ Sx=l_u^{k'}\circ l_g\circ l_u^{-k'}x
  $ if~$ x\in M_{k'} $. Then
  \begin{equation}
    \label{eq:dicon} f(l_u^n(Sx))=\omega f(l_u^nx)\text{ for all }n\in\Z\text
    { and }x\in M.
  \end{equation}
  Indeed, let~$ x\in M_{k'} $:  if~$ n\neq k-k'\mod m $ then both~$
  l_u^nx $ and~$ l_u^n(Sx) $ do not belong to~$ M_k $, and hence
  both sides of \eqref{eq:dicon} vanish.  If~$ n=k-k'+n' m $ for some
  $ n'\in\Z $ then~$ x=l^{k'}_uy $ for some~$ y\in M_0 $ and
  \[
    f(l_u^n(Sx))=h(l_u^{n+k'-k}l_{g}y)=h(\phi^{n'} l_{g}y)=\omega h(\phi^
    {n'} y) =\omega h(l_u^{n+k'-k}y)=\omega f(l_u^nx).
  \]

  Note that there exits only finitely many pairs of distinct
  prime numbers $(r, s)$ such that $m, r, s$ are not pairwise coprime.
  Therefore to prove we can assume that $(r,s)$ is a pair such that $m, r, s$ are pairwise coprime. Then both the
   nil-translations $
  l_u^r $ and~$ l_u^s $ are also uniquely ergodic.

  Let~$ \rho $ be an ergodic joining of~$ l_u^r $ and~$ l_u^s $.
  Choose any pair~$ f_1,f_2 $ of functions in~$ \mathscr{C} $. Suppose
  that~$ f_1\in \mathscr{C}_{k_1} $ and~$ f_2\in \mathscr {C}_{k_2}
  $ for some~$ 0\leq k_1,k_2<m $.  Then, for~$ i=1,2 $, $ f_i $ is
  supported on~$ M_{k_i} $ and~$ f_i= h_i\circ l_u^{-k_i} $ for some~$ h_i\in
  \mathscr{C}_0 $.  Denote by~$ \rho_{k_1,k_2} $ the restriction of~$
  \rho $ to the set~$ M_k\times M_{k'} $. Then the measure~$ (l_u^
  {-k_1}\times l_u^{-k_2})_*\rho_{k_1,k_2} $ is supported on~$ M_0\times
  M_0 $ and coincides with the restriction of~$ (l_u^ {-k_1}\times l_u^{-k_2})_*\rho
  $ to~$ M_0\times M_0 $. Then we have
  \[
    \rho(f_1\otimes \bar{f}_2)=(l_u^{-k_1}\times l_u^{-k_2})_*\rho_{k_1,k_2}
    (h_1\otimes \bar{h}_2).
  \]
  If the measure~$ \rho_{k_1,k_2} $ vanishes then ~$
  \rho(f_1\otimes \bar{f}_2)=0 $.  Thus we may assume that $ \rho_{k_1,k_2} $
  is not zero.  As~$ l_u^r $ and~$ l_u^s $ are uniquely ergodic and~$
  l_u^{-k_1}\times l_u^{-k_2} $ commutes with~$ l_u^{r}\times l_u^{s} $,
  the measure~$ (l_u^{-k_1}\times l_u^{-k_2})_*\rho $ is also an
  ergodic joining of~$ l_u^r $ and~$ l_u^s $.

  Since~$ r,s,m $ are pairwise coprime and~$ (l_u^{-k_1}\times l_u^ {-k_2})_*\rho
  $ does not vanish on~$ M_0\times M_0 $, the measure~$ (l_u^{-k_1}\times
  l_u^{-k_2})_*\rho $ is supported on~$ \bigcup_{k=0}^{m-1}(l_u^r\times
  l_u^s)^k(M_0\times M_0) $, where the elements of the union are
  pairwise disjoint and~$ (l_u^r\times l_u^s)^m(M_0\times M_0)=M_0\times
  M_0 $.  It follows that the restriction~$ (l_u^{-k_1}\times l_u^{-k_2})_*\rho_
  {k_1,k_2} $ is an~$ l_u^{rm}\times l_u^{sm} $-invariant ergodic
  measure on~$ M_0\times M_0 $.  By (i) in Theorem~%
  \ref{thm:mainaff}, it follows that
  \[
    \rho(f_1\otimes \bar{f}_2)=(l_u^{-k_1}\times l_u^{-k_2})_*\rho_{k_1,k_2}
    (h_1\otimes \bar{h}_2)=0
  \]
  for all but one pair~$ (r,s) $, which completes the proof.
\end{proof}

\begin{Th}
  \label{thm:nonconnect} Let~$ l_u $ be a nil-translation of
  the compact nilmanifold~$ M=G/\Gamma $.  Then
  \begin{equation}
    \label{short1} \frac1M\sum_{M\leq m<2M}\left|\frac1H\sum_{m\leq
    n<m+H}f(l_u^nx)\mob(n)\right|\to 0
  \end{equation}
  when~$ H\to\infty $, $ H/M\to0 $ for each~$ x\in M $ and each~$ f\in
  C(M) $. If additionally $l_u$ is ergodic then \eqref{short1} hold for all $f\in D(M)$.
\end{Th}

\begin{proof}
  Assume that the nil-translation is ergodic.  Denote by~$ \mathcal{B}_{rat}\subset\mathcal{B} $ the maximal factor
  of~$ l_u $ with the rational discrete spectrum.  Then~$ H^+=L^2 (M,\mathcal
  {B}_{rat},\lambda) $.  In view of (i) in Proposition~%
  \ref{prop:nonconnect}, the nil-translation satisfies the AOP
  property (i.e. \eqref{cond:aop}) for all functions~$ f,g\in H^-=L^2
  (M,\mathcal{B}_{rat},\lambda)^\perp $.  Since the space~$ H^- $ does
  not depend on the choice of the topological model of the
  nil-translation, similar arguments to those used in the proof of
  Theorem~3 in
  \cite{Ab-Le-Ru} show that for each zero mean~$ f\in C(M)\cap H^- $,
  for any sequence~$ (x_k)_{k\geq 1} $ in~$ M $ and any~$ (b_k)_ {k\geq
  1} $ with~$ b_{k+1}-b_k\to +\infty $ we have
  \[
    \frac1{b_{K+1}}\sum_{k\leq K}\Big(\sum_{b_k\leq n< b_{k+1}}f(l_u^nx_k)\mob
    (n)\Big)\to0\text{ when }K\to\infty.
  \]
  Together with (ii) in Proposition~%
  \ref{prop:nonconnect} this gives \eqref{short1} for every~$ f\in C(M)\cap
  H^- $.  If~$ f\in H^+ $ then the sequence $(f(l_u^nx))$ is periodic and the property~\eqref{short1} follows
  from Theorem~1.7 in
  \cite{Ma-Ra-Ta} applied to the Möbius function (cf.\ Remark~\ref{rem:mrt}).  If follows that \eqref{short1}
  holds for every~$ f\in C(M) $, and finally, by Remark~%
  \ref{rem:riemint}, also for every~$ f\in D(M) $.

  We need now to consider the case where the  nil-translation~$ l_u $
  is not ergodic. Then, for any~$ x\in M $ denote by~$ M_x $ its orbit closure.  By
  \cite{Rat91}, the restriction of~$ l_u $ to~$ M_x $ is topologically
  isomorphic to an ergodic nil-translation on a compact (not necessary
  connected) nilmanifold.  Therefore, our claim is reduced to the
  ergodic case.
\end{proof}

\subsection{Polynomial type sequences}
\begin{Th}
  \label{thm:gammaaffine} Let~$ \phi $ be an ergodic affine unipotent
  diffeomorphism of a compact connected nilmanifold~$ M=G/\Gamma $.
  Then for all~$ \gamma\in \R\setminus\{0\} $ and $\varrho\in\R$, we have
  \begin{equation}
    \label{short2} \frac1M\sum_{M\leq m<2M}\left|\frac1H\sum_{m\leq
    n<m+H}f(\phi^{[\gamma n+\varrho]}x)\mob(n)\right|\to 0
  \end{equation}
  when~$ H\to\infty $, $ H/M\to0 $ for each~$ x\in M $ and each~$ f\in
  C(M) $.
\end{Th}

\begin{proof}
  Let~$ \widetilde \Sigma=(\widetilde G, \widetilde \Gamma, p , ({\widetilde \phi}_t)_{t\in \R}, \widetilde \lambda) $
  be the suspension of the ergodic affine unipotent
  diffeomorphism $\Sigma=(G,\Gamma, \phi,\lambda)$, with $\lambda$ and $\widetilde
  \lambda$ the Haar measures on the corresponding nilmanifolds.
  The nilflow $\widetilde \Sigma$ is uniquely ergodic. As usual we
  identify $G/\Gamma$ with $p^{-1}\{0\}$. For each continuous function
  $f$ on $G/\Gamma$ let $\widetilde f$ be the unique function on $\widetilde
  G/\widetilde \Gamma$ defined by the condition
\[
\widetilde f({\widetilde \phi}_t x ) = f(x), \quad \forall x \in G/\Gamma\approx p^{-1}\{0\},~
\forall t\in [0,1)\,,
\]
or, equivalently, by the condition
\begin{equation}
  \label{eq:fractional}
\widetilde f({\widetilde \phi}_t x ) = f(\phi^{\lfloor t\rfloor}x), \quad \forall
x \in G/\Gamma\approx p^{-1}\{0\} ,~
\forall t\in \R\,.
\end{equation}

Since the set of discontinuities of~$ \widetilde{f} $ is contained
in~$p^{-1}\{0\}$, the function $\widetilde f$ belong so the class
$D(\widetilde{M}) $ defined in  Remark~\ref{rem:riemint}.

  The map $ {\widetilde \phi}_\gamma$ is a nil-translation $ l_{\widetilde {u}}
  $ on
  $ \widetilde G/\widetilde \Gamma $ by an element $\widetilde u\in \widetilde G$; thus we
  shall consider two cases.

 \noindent \textbf{Ergodic case:} Suppose that the nil-translation~$ {\widetilde
    \phi}_\gamma =l_{\widetilde {u}}$ on~$ \widetilde G/\widetilde \Gamma $ is
  ergodic.  Then ${\widetilde
    \phi}_{\gamma n+\varrho}={\widetilde
    \phi}_{\varrho}\circ l_{\widetilde {u}}^n $ and, by
  Theorem~\ref{thm:nonconnect} and formula~\eqref{eq:fractional}, we have, for all $x\in G/\Gamma$,
  \begin{align*}
    \frac1M&\sum_{M\leq m<2M}\left|\frac1H\sum_{m\leq n<m+H}f(\phi^{[\gamma
    n+\varrho]}x)\mob(n)\right|\\
    &=\frac1M\sum_{M\leq m<2M}\left|\frac1H\sum_{m\leq n<m+H}\widetilde{f}\circ {\widetilde
    \phi}_{\varrho}
    (l_{\widetilde {u}}^n(x))\mob(n)\right|\to 0.
  \end{align*}
  \textbf{Non-ergodic case:} Assume that the nil-translation~$ {\widetilde
  \phi}_\gamma = l_{\widetilde {u}}$ on~$ \widetilde G/\widetilde \Gamma $ is not ergodic.  Then, by
  \cite{Rat91}, for every~$ x=g\Gamma\in G/\Gamma $ there exists a closed
  subgroup~$ H<\widetilde{G} $ so that $ \widetilde{u}\in H $, the~$ l_{\widetilde u}
  $-orbit closure of~$ \widetilde{x}:=g\widetilde \Gamma $
  coincides with the orbit~$ H\widetilde{x} $ and the restriction of~$
  l_{\widetilde{u}} $
  to the sub-nilmanifold~$ \widetilde W:=H\widetilde{x}\approx  H/
  (H\cap g\widetilde \Gamma g^{-1}) $ is uniquely ergodic. The
  sub-nilmanifold~$  \widetilde W$ projects via $p$ to a sub-nilmanifold of
  $\R/\Z$, that is to  a closed subgroup $W$ of $\R/\Z$. There are two
  possibilities: either $W$ is finite or $W=\R/\Z$. In the first case
  the nilmanifold $ \widetilde W $ has a finite number of components
  and the restriction $ {\widetilde{f}\circ {\widetilde
    \phi}_{\varrho}}_{|  \widetilde W}$ of~$ \widetilde{f} \circ {\widetilde
    \phi}_{\varrho}$ to  $\widetilde W$ is continuous. In the second case, the
  discontinuities of $ {\widetilde{f}\circ {\widetilde
    \phi}_{\varrho}}_{|  \widetilde W}$ belong to the
  lower dimensional sub-nilmanifold $\widetilde W \cap p^{-1} \{-\varrho\}$.

  % Recall that~$ G $ is a normal closed
  % subgroup of~$ \widetilde{G} $ and its orbit~$ G\widetilde{x} $ is equal
  % to~$ \Psi(G\times\{0\})=\mathcal{L}_0 $.  Therefore, the set of
  % discontinuities of~$ \widetilde{f}:H\widetilde{x}\to\C $ is contained
  % in~$ H\widetilde{x}\cap G\widetilde{x}=(H\cap G)\widetilde{x} $. Of
  % course, $ H\cap G $ is a Lie subgroup of~$ \widetilde{G} $.

  % If~$ \dim(H\cap G)<\dim H $ then~$ (H\cap G)\widetilde{x} $ is a
  % compact submanifold of~$ H\widetilde{x} $ with a drop of dimension. It
  % follows that~$ \widetilde{f}\in D(H\widetilde{x}) $.

  % Assume that~$ \dim(H\cap G)=\dim H $.  Denote by~$ H_0<H $ the
  % connected identity component of~$ H $.  Then~$ H_0 $ is a clopen
  % normal subgroup of~$ G $ and~$ H_0<H\cap G $.  Moreover, there
  % exists~$ k\in\N $ such that~$ H\widetilde{x} $ is the disjoint union
  % of the compact sets $ l^j_{\widetilde{u}}H_0\widetilde{x} $, $ 0\leq j<k
  % $.  Since~$ \widetilde{x}\in H $ and~$ H_0 $ is a normal subgroup of~$
  % H $, each s~$ l^j_{\widetilde{u}}H_0\widetilde{x} $ is either a subset
  % of~$ (H\cap G)\widetilde{x} $ or is disjoint from this set.  As~$
  % \widetilde{f} $ restricted to~$ G\widetilde{x} $ is continuous, it
  % follows that~$ \widetilde{f}:H\widetilde{x}\to\C $ is continuous.

  In summary, in both cases we have~$ {\widetilde{f}\circ {\widetilde
    \phi}_{\varrho}}_{|  \widetilde W}\in D(\widetilde{W}) $
  and it can be treated as a function on the compact nilmanifold $\widetilde{W}= H/
  (H\cap g\widetilde \Gamma g^{-1}) $.  Then, by Theorem~%
  \ref{thm:nonconnect}, we have
  \[
    \frac1M\sum_{M\leq m<2M}\left|\frac1H\sum_{m\leq n<m+H}\widetilde{f}\circ {\widetilde
    \phi}_{\varrho}
    (l_{\widetilde {u}}^n(x)(\widetilde{x}))\mob(n)\right|\to 0,
  \]
  which completes the proof in the non-ergodic case.
\end{proof}

\begin{proof}[Proof of Proposition~\ref{PropD}]
  Let~$ P(x)=a_dx^d+\ldots+a_1x+a_0\in\R[x] $ have the leading
  coefficient $ a_d $ irrational.  Let~$ \alpha=a_d\cdot d! $ and let~$
  \phi\colon\T^d\to\T^d $ be given by
  \[
    \phi(x_1,x_2,\ldots, x_d)=(x_1+\alpha, x_1+x_2,\ldots, x_{d-1}+x_d).
  \]
  Then~$ \phi $ is an ergodic affine unipotent diffeomorphism of~$ \T^d
  $.  Following
  \cite{Fu} (see also
  \cite{Ei-Wa}), we can now find~$ x_1,\ldots,x_d\in\T $ so that
  \begin{equation*}
    P(n)=\binom{n}{d} \alpha+ \binom{n}{d-1}x_1+\ldots+nx_{d-1}+x_d\bmod
    1\text{ for each }n\in\Z.
  \end{equation*}
  Moreover, notice that (mod~1) we have
  \begin{equation*}
  {n\choose d}
    \alpha+{n\choose d-1}x_1+\ldots+nx_{d-1}+x_d=f(\phi^n(x_1,\ldots,x_d)),
  \end{equation*}
  where~$ f\colon \T^d\to\T $ is the projection on the last
  coordinate.  Finally, the assertion of the theorem follows directly
  from Theorem~%
  \ref{thm:gammaaffine} applied to the function $\exp(\imath f)$.

\end{proof}
% \bibliographystyle{siam} \bibliography{aop_nil.bib}

% \begin{thebibliography}{99}
% \bibitem{Fr-Le}K.\ Fr\k{a}czek, M.\ Lema\'nczyk, \emph{Smooth
%   singular flows in dimension 2 with the minimal self-joining
%   property}, Monatsh. Math. 156 (2009), 11-45.
% \end{thebibliography}
% \end{document}
%
% \begin{Remark} By taking~$\bfu(n)=e^{it\log n}$ ($0\neq t\in\R$), we
%   obtain
%$$\frac1M\sum_{M\leq m<2M}\frac1H\left|\sum_{m\leq n< m+H}e^{2\pi i(P([\gamma n])+t\log n)}\right|\to 0$$
%when~$M,H\to\infty$ and~$H/M\to 0$. In particular, the sequence
% $(P([\gamma n])+t\log n)_{n\geq2}$ is uniformly distributed mod~1.
% \end{Remark}

\subsection{Nilsequences and polynomial multiple correlations}
\label{pmc}

 Following \cite{MR2138068} and \cite{MR2643713}, recall that a bounded sequence
$(c_n)\in \C^\N$ is called a \emph{basic nilsequence} if there exist
a compact but not necessarily connected nilmanifold $G/\Gamma$, a
continuous function $f\in C(G/\Gamma)$, an element $u\in G$ and a
point $x\in G/\Gamma$ such that
\[
c_n=f(l_{u}^n x), \qquad \forall n\in\N,
\]
where as usual $l_u$ denotes the nil-translation $G/\Gamma\ni x
\mapsto ux\in G/\Gamma$. A sequence $(d_n)\in  \C^\N$ is called a
\emph{nilsequence}, if it is a uniform limit of basic nilsequences,
i.e.\ for
every $\varepsilon>0$, there exists a basic nil-sequence $(c_n)$ such that
\[
|c_n-d_n|<\varepsilon\text{ for all }n\in\N.
\]

\begin{proof}[Proof of Theorem~\ref{TheoremE}] Theorem~\ref{TheoremE} follows directly from Theorem~\ref{thm:nonconnect}.
\end{proof}

\begin{Th*}[Leibman, \cite{MR2643713}]
  Let $T$ be an automorphism of a probability standard Borel space
  $(X,{\mathcal B},\mu)$. Given, for all $i=1,\ldots,k$, functions
  $g_i\in L^\infty(X,\mu)$ and polynomials $p_i\in\Z[x]$, there exists a nilsequence $(d_n)$ such that
  \begin{equation}
    \label{eq:leibman}
  \limsup_{{N-M}\to\infty} \frac1{N-M}\sum_{n=M}^{N-1} \left|d_n-\int_X
g_1\circ T^{p_1(n)}\cdot\ldots\cdot g_k\circ T^{p_k(n)}\,d\mu\right|=0.
\end{equation}
\end{Th*}

%\begin{Th}
%Let $T$ be an automorphism of the standard Borel probability space
%$(X,{\mathcal B},\mu)$,   $k$ be a  positive integer,
%$g_i\in L^\infty(X,\mu)$ and $p_i\in\Z[X]$, for
%all $i=1,\ldots,k$. Then
%\[
%\frac1M\sum_{M\leq m<2M}\left| \frac1H\sum_{m\leq n<m+H} %\mob(n)\int_X
%g_1\circ T^{p_1(n)}\cdot\ldots\cdot g_k\circ
%T^{p_k(n)}\,d\mu\right|\to 0
%\]
%when $H\to\infty$ and $H/M\to0$.
%\end{Th}

\begin{proof}[Proof of Corollary~\ref{CorollaryF}]
  By Theorem~\ref{TheoremE}, we have
\[
\frac1M\sum_{M\leq m<2M}\left| \frac1H\sum_{m\leq n<m+H} \mob(n) d_n\right|\to 0
\]
for a
nilsequence  $(d_n)$ satisfying the assertion~\eqref{eq:leibman} of
Leibman's theorem.
%Thus
% \[
%   \begin{split}
%    \frac1H&\sum_{m\leq n<m+H} \mob(n)\int_X
%   g_1\circ T^{p_1(n)}\cdot\ldots\cdot g_k\circ
%   T^{p_k(n)}\,d\mu\\
% &=\frac1H\sum_{m\leq n<m+H} \mob(n) \left(d_n - \int_X
%   g_1\circ T^{p_1(n)}\cdot\ldots\cdot g_k\circ
%   T^{p_k(n)}\,d\mu\right)
% \\&\qquad\qquad
% + \frac1H\sum_{m\leq n<m+H}  \mob(n) d_n
%   \end{split}
% \]
The statement follows now by an  immediate application of the triangle inequality.
\end{proof}

\section{Lifting AOP to induced action}
\subsection{Induced actions}%
\label{ap-d} In this section, we follow
\cite{Ma},
\cite{Zi}.  Assume that $ G $ is a locally compact second countable (lcsc)
group.  Assume that $ H\subset G $ is a closed subgroup of~$ G $.
% We will constantly assume that
% \begin{equation}\label{in1}\mbox{$H$ has no non-trivial compact
%     subgroups.}\footnote{Equivalently, we assume that the dual group
%     %   $\widetilde{H}$ is connected. It may happen that~$G$ has non-trivial
%     compact subgroups.}
% \end{equation}
The quotient space~$ G/H $ is lcsc for the quotient topology.  Let $
\pi:G\to G/H $ be the canonical, continuous quotient map.  Let $ \tau=
(\tau_g)_{g\in G} $ denote the~$ G $-action on~$ G/H $  by left
translations:~$ \tau_g(xH)= gxH $.  The action~$ \tau $ is continuous,
transitive, hence ergodic for any~$ G $-quasi-invariant Borel
probability measure~$ m $ on~$ G/H $ (such measures always exist).  We
say that~$ G/H $ is a finite volume space if~$ G/H $ supports a~$ G $-invariant
probability measure~$ m_{G/H} $,

Let~$ \cs=(S_h)_{h\in H} $ be an almost free Borel left action of~$ H $
on a probability standard Borel space~$ \ycn $.  The group~$ H $ acts
on the right on the product $ Y\times G $, by
\[
  (y,g)h = (S_{h^{-1}} y, gh)
\]
Let~$ (Y\times G)/H $ be the orbit space endowed with the quotient
measurable structure.  Since~$ G/H $ is Hausdorff, since~$ \ycn $ is
a standard Borel space, and since the projection~$ (Y\times G)/H \to
G/H $ measurable, the Borel space~$ (Y\times G)/H $ is countably
separated (i.e. the~$ H $-action on~$ Y\times G $ is smooth, in the
sense of
\cite[Def.\ 2.1.9]{Zi}).

% It is also countably generated and complete, i.e. a standard Borel
% space.

The~$ G $-action $ ( (\tau_{\cs})_g)_{g\in G} $ on~$ (Y\times G)/H $
is just given by left translation on the second factor:
\begin{equation}
  \label{eq:NilpotentAOP-14-06:1} (\tau_{\cs})_g:  (Y\times G)/H\to (Y\times
  G)/H, \quad (\tau_{\cs})_g(y,g_1)H=(y, gg_1)H
\end{equation}
is called the \emph{$ G $-action induced from~$ \cs $}.

If~$ m $ is any~$ G $-quasi-invariant Borel probability measure on~$
G/H $, then we may define a Borel probability measure~$ \nu \otimes_{\mathcal
S} m $ on~$ (Y\times G)/H $, by setting, for any positive measurable
function~$ F $ on~$ (Y\times G)/H $,
\[
  \nu \otimes_{\mathcal S} m (F) = \int_{G/H} \left(\int_Y \widetilde
  F(y,g) \, \D \nu(y)\right)\, \D m (gH)
\]
where~$ \widetilde F $ is the~$ H $-invariant lift of~$ F $ to~$ Y\times
G $.  The measure~$ \nu \otimes_{\mathcal S} m $ is quasi-invariant
for the~$ G $-action induced from~$ \cs $ and it is ergodic if the
action~$ \cs $ is ergodic on $ \ycn $. Furthermore, if~$ G/H $ admits
a~$ G $-invariant probability measure~$ m_{G/H} $, then~$ \nu
\otimes_{\mathcal S} m_{G/H} $ is a~$ G $-invariant Borel probability
measure on~$ (Y\times G)/H $.

Let
\begin{equation}
  \label{in2} s:G/H\to G,\text{ a Borel map, }\pi\circ s=
  \operatorname{Id}
  _{G/H}
\end{equation}
be a (measurable) selector for~$ \pi $.
% Since~$H$ is cocompact, there is a relatively compact fundamental
% domain, that is, a subset~$A\subset G$ such that
% \begin{equation}\label{in2a}
%   \mbox{$\ov{A}$ is compact, $hA\cap h'A=\emptyset$ whenever~$h\neq h'$ and~$\bigcup_{h\in H}hA=G$.}\end{equation} We can always assume that
% \begin{equation}\label{in2b}
%   s:H\to A\text{ and }s(H)=e.\end{equation}
%   % and let
% \begin{equation}\label{in2a}
%   \mbox{$\ov{A}$ is compact, $hA\cap h'A=\emptyset$ whenever~$h\neq h'$ and~$\bigcup_{h\in H}hA=G$.}\end{equation}
%   % We can always assume that
% \begin{equation}\label{in2b}
%  s(H)=e.
% \end{equation}
The map~$ \theta $ defined on~$ G\times G/H $ by setting
\begin{equation}
  \label{in3} \theta(g,xH)=s\big(gxH\big)^{-1}\, g \,s\big(xH\big).
\end{equation}
takes its values in~$ H $, since %$$
%\big(gs(gxH)^{-1}s(xH)\big)H=gH\cdot (gxH)^{-1}\cdot xH=H,
%$$
$$
  s\big(gxH\big) H= gxH, \quad \text{and}\quad g \,s\big(xH\big)= gxH.
$$
% whence
% \begin{equation}\label{in3a} \theta:G\times G/H\to H.
% \end{equation}

\begin{Lemma}
  \label{l:in1} The map~$ \theta:G\times G/H\to H $ is a (left)
  cocycle for the~$ G $-action~$ \tau $.  % Moreover,
  % \begin{equation}\label{in4}
  %   \theta(h,xH)=h\end{equation} for each~$h\in H$ and~$xH\in G/H$.
\end{Lemma}
\begin{proof}
  We have
  \[
    \begin{split}
      \theta(g_1g_2,xH)&=s(g_1g_2xH)^{-1} g_1g_2 s(xH)\\
      &= s(g_1g_2xH)^{-1} g_1\,s(g_2 xH)^{-1}s(g_2 xH)\,g_2\,s(xH)\\
      &= \theta(g_1,g_2xH)\,\theta(g_2,xH).\qedhere
    \end{split}
  \]
\end{proof}

% Now, \eqref{in4} follows by the definition~\eqref{in3} of~$\theta$.

The skew product~$ G $-action $ \tau_{\theta,\cs}=((\tau_{\theta,\cs})_g)_
{g\in G} $, defined by
\begin{equation}
  \label{in5} (\tau_{\theta,\cs})_g:  Y\times G/H\to Y\times G/H,\qquad\quad
  (\tau_{\theta,\cs})_g(y, xH)=(S_{\theta(g,xH)}y,gxH)
\end{equation}
is isomorphic to the~$ G $-action~$ \tau_{\cs} $ on~$ (Y\times G)/H $
induced from~$ \cs $, via the Borel isomorphism~$ \Phi\colon Y\times
G/H\to (Y\times G)/H $ defined by \[ \Phi( y, gH) = \big(S_{g^{-1}s(gH)}y, g\big)H = \big(y, s(gH)\big)H.\]

% \footnote{Note that~$\theta$ is so called {\em transient}
% cocycle. Hence we will not have any non-trivial essential value of
% it, that is, $E(\theta)=\{e,\infty\}$ \cite{Sch}. Recall that if
% there is a non-trivial essential value of a cocycle then it has to
% be recurrent \cite{Sch}.}

% By Remark~\ref{r:in1} and the ergodicity of~$\theta$ (and of~$\cs$),
% it follows immediately that
% \begin{equation*}
%   \tau_{\theta,\cs}\text{ is ergodic.}
% \end{equation*}

% It is quite often (see the classical construction of the suspension
% of an automorphism) that we reverse the order: that is once
% $\cs=(S_h)_{h\in H}$ is given and~$G$ extending~$H$ is known then
% we write~$\widetilde{\cs}=(\widetilde{S}_g)_{g\in G}$ acting
% on~$(Y\times G/H,\nu\ot m_{G/H})$, where
% \begin{equation}\label{in7}
%   \widetilde{S}_g(y,xH)=(S_{\theta(g,xH)}(y),\tau_g(xH))=(S_{\theta(g,xH)}(y),gxH).
% \end{equation}

% It follows that~$\widetilde{\cs}$ is an ergodic~$G$-action.

\begin{Remark}
  \label{r:in1} In the vocabulary of
  \cite{Le-Le},
  \cite{Le-Pa}, the induced~$ G $-action is a Rokhlin cocycle
  extension of the~$ G $-action~$ \tau $ through the~$ H $-valued
  cocycle~$ \theta $ and an ergodic~$ H $-action~$ \cs $.  Both
  descriptions \eqref{eq:NilpotentAOP-14-06:1} and \eqref{in5} for
  the~$ G $-action induced from $ \cs $ have their advantages and
  disadvantages:  the first is natural and intrinsic; the second,
  albeit depending of a arbitrary selector~$ \theta $, yields an easier
  description of the measure~$ \nu\otimes_{\cs} m $ which is just~$
  \nu\otimes m $ on the space~$ Y\times G/H $; the second definition
  also make apparent that~$ (Y\times G)/H $ endowed with the measure~$
  \nu\otimes m $ is a standard Borel probability space.
\end{Remark}

\subsection{The AOP property}\label{liftingAOPia}
\label{ap-e} Assume now that~$ G $ is an abelian group without torsion
elements, that $ H<G $ is closed and co-compact and let~$ m_{G/H} $ be
the~$ G $-invariant probability measure on $ G/H $.  Assume also that
we have an ergodic~$ H $-action $ \cs=(S_h)_{h\in H} $ on~$ \ycn $ and
let~$ p\in\Z $.  Then~$ p $ determines another~$ H $-action~$ \cs^{(p)}=
(S^{(p)}_h)_{h\in H} $, where
\begin{equation}
  \label{in8} S^{(p)}_h(y)=(S_h)^p(y)=S_{h^p}(y);
\end{equation}
indeed the map~$ h\mapsto h^p $ is a group homomorphism.

Following
\cite{Ab-Le-Ru}, we now define:
\begin{Def}
  An ergodic~$ H $ action~$ \cs $ is said to have the AOP property if
  for each~$ f,g\in L^2_0\ycn $
  \begin{equation}
    \label{in9} \lim_{p\neq q, p,q\in\mathscr{P},p,q\to\infty}
    \sup_{\kappa\in J^e(\cs^{(p)},\cs^{(q)})}\left|\int_{Y\times Y}f\ot
    g\,d\kappa\right|=0.
  \end{equation}
\end{Def}
If \eqref{in9} is satisfied for~$ f,g $, then we say that~$ f,g $
satisfy the AOP property.

\begin{Remark}
  \label{r:in00} (i) In order to talk about the AOP property, we need
  to know that the~$ H $-actions are ``totally'' ergodic, that is
  that~$ \cs^{(p)} $ is ergodic for each prime~$ p $.  To see its
  meaning let us pass to the character group~$ \widehat{H} $ and
  consider the endomorphism~$ E_p:  \chi\mapsto \chi^p $.  We need to
  assume that the kernel of~$ E_p $ is a set of measure zero for the
  maximal spectral type of~$ \cs $ on~$ L^2_0\ycn $.  If~$ H $
  equals~$ \Z $ then it simply means that~$ T $ has no roots of unity
  as eigenvalues.  If~$ \widehat{H} $ is torsion free, as for example
  it holds for~$ H=\R $, then each ergodic action is totally ergodic.

  (ii) If \eqref{in9} holds then it also holds if~$ f\in L^2\ycn $
  and~$ g\in L^2_0\ycn $ (or vice versa); indeed, we
  write~$ f=f_0+c $, with~$ c=\int_Yf\,d\nu $
  then~$ \int_{Y\times Y}c\ot g\,d\kappa=\int_Y g\,d\nu=0 $.

  (iii) If \eqref{in9} holds, then, in fact, we can replace the set of
  ergodic joinings $ J^e(\cs^
  {(p)},\cs^{(q)}) $ by $ J(\cs^{(p)},\cs^{(q)}) $; indeed,
  $$
    \sup_{\kappa\in J^e(\cs^{(p)},\cs^{(q)})}\left|\int_{Y\times Y}f\ot
    g\,d\kappa\right|=\sup_{\kappa\in J(\cs^{(p)},\cs^{(q)})}\left|\int_
    {Y\times Y}f\ot g\,d\kappa\right|
  $$
  for each~$ f,g\in L^2\ycn $.

  (iv) If~$ L $ is a closed subgroup of~$ H $ and the~$ L $-subaction
  of~$ \cs $ has the AOP property, then the original action has the
  AOP property.  Indeed, any ergodic invariant measure for~$ S_{h^p}\times
  S_{h^q} $, $ h\in H $ is an invariant measure for the subaction~$ S_
  {\ell^p}\times S_{\ell^q} $, $ \ell\in L $.  Then use~(iii).
\end{Remark}

We will constantly assume that
\begin{equation}\label{in1}\mbox{$H$ has no non-trivial compact subgroups.}\footnote{Equivalently, we assume that the dual group $\widehat{H}$ is connected. It may happen that $G$ has non-trivial compact subgroups.}
\end{equation}
Since $H$ is cocompact, there is a relatively compact fundamental domain, that is, a subset $A\subset G$ such that
\begin{equation}\label{in2a}
\mbox{$\ov{A}$ is compact, $hA\cap h'A=\emptyset$ whenever $h\neq h'$ and $\bigcup_{h\in H}hA=G$.}\end{equation} We can always assume that (cf.~\eqref{in2})
\begin{equation}\label{in2b}
s:H\to A\text{ and }s(H)=e.\end{equation}

%Our aim will be to prove the following.

%\begin{Prop}
%\label{p:in1} Assume that an~$ H $-action~$ \cs $ has the AOP
%property.  The for each~$ F,G\in L^2(Y\times G/H)\ominus L^2(G/H) $
%we have~\eqref{in9} for the induced~$ G $-action~$ \widetilde{\cs} $
%(i.e. the induced action has the ``relative AOP'' property with
%respect to the factor~$ \tau $).
%\end{Prop}

\subsection{Ergodic components for the \texorpdfstring{$ G $}{G}-action
\texorpdfstring{$ \tau^{(p)}\times\tau^{(q)} $}{tau\^{}(p) times tau\^
{}(q) } and regularity of cocycles} Recall~$ \tau^{(p)}_g(xH)=g^pxH $.
Assume that~$ (p,q)=1 $ and let $ a,b\in\Z $ be so that
\[
  ap+bq=1.
\]
The following result is a more general version of Lemma~2.2.1 in \cite{Ku-Le} and Lem\-ma~\ref{lem:torus}.
Since the proof runs along similar lines, we omit it.
\begin{Lemma}
  \label{l:in2}The ergodic components of~$ \tau^{(p)}\times\tau^{(q)}=
  (\tau^{(p)}_g\times\tau^{(q)}_g)_{g\in G} $ are of the form
  \[
    \ch_c:=\big\{(x_1H,x_2cH):\:x_1^qH=x_2^pH\big\},\,c\in G.%
    \footnote{Note that since $G$ is abelian, $x^mH=(xH)^m$ in $G/H$ }
  \]
  Moreover, the action of~$ \tau^{(p)}\times\tau^{(q)} $ on~$ \ch_c $
  is (topologically)%
  \footnote{The action of~$ \tau $ on~$ G/H $ is uniquely ergodic and
  so is the action of~$ \tau^{(p)}\times\tau^{(q)} $ on~$ \ch_c $.}
  isomorphic to~$ \tau $.
\end{Lemma}
%\begin{proof}
%  Clearly, either~$ \ch_c=\ch_{c'} $ or they are disjoint. Moreover, $
%  \ch_c $ is closed and~$ \tau^{(p)}_g\times\tau^{(q)}_g $-invariant, $
%  g\in G $.  Define
%  \begin{equation}
%    \label{in10} R:\ch_c\to G/H,\;R(x_1H,x_2cH):=x_1^ax_2^bH.
%  \end{equation}
%  Since~$ g=g^{ap+bq} $, we have
%  $$
%    R\circ (\tau^{(p)}_g\times\tau^{(q)}_g)=\tau_g\circ R.
%  $$
%  Moreover, $ R $ is continuous, invertible and its inverse is given
%  by the map $ xH\mapsto (x^pH,x^qcH) $.
%\end{proof}

We are now interested in the (infinite) skew product $ G $-action $
(\tau_\theta)^{(p)}\times(\tau_\theta)^{(q)} $,
\[
\begin{split}
    &\big((\tau_\theta)^{(p)}\times(\tau_\theta)^{(q)}
    \big)_g(x_1H,h_1,x_2H,h_2)=\\ &\qquad\qquad
  \big(g^px_1H,\theta(g^p,x_1H)h_1,g^qx_2H, \theta(g^q,x_2H)h_2\big)
\end{split}
\]
restricted to~$ \ch_c\times H\times H $ (up to a natural interchange
of coordinates).

\begin{Lemma}
  \label{l:in3}For each~$ c\in G $, the map~$ \widetilde{R}(x_1H,h_1,x_2cH,h_2):=
  (x_1^ax_2^bH,h_1,h_2) $ establishes an isomorphism of~$ (\tau_\theta)^
  {(p)}\times(\tau_\theta)^{(q)}|_{\ch_c} $ with~$ \tau_{\theta^{(p,q)}}
  $, where the cocycle
  \[
    \theta^{(p,q)}:G\times G/H\to H\times H,
  \]
  is given by the formula
  \begin{equation}
    \label{in11} \theta^{(p,q)}(g,yH):=\big(\theta(g^p,y^pH),\theta(g^q,y^qcH)\big).
  \end{equation}
\end{Lemma}
\begin{proof}
  First notice that indeed~$ \theta^{(p,q)} $ is a cocycle, that is,
  we have
  \begin{align*}
    \theta^{(p,q)}(g_1g_2,yH)&=\big( \theta((g_1g_2)^p,y^pH),(\theta((g_1g_2)^q,y^qcH)\big)\\
    &=
    \big(\theta(g_1^p,y^pH),\theta(g_1^qy^qcH)\big)\cdot \big(\theta(g_2^p,g_1^py^pH),\theta
    (g_2^q,g_1^qy^qcH)\big)\\
    & =
    \theta^{(p,q)}(g_1,yH)\cdot\theta^{(p,q)}(g_2,\tau_{g_1}(yH)).
  \end{align*}
  Since
  \[
    \theta^{(p,q)}(g,x_1^ax_2^bH)=\big(\theta(g^p,x_1^{pa}x_2^{pb}H),
    \theta(g^q,x_1^{qa}x_2^{qb}cH)\big)
  \]
  and
  \[
    x_1^{pa}x_2^{pb}H=x_1^{pa}x_1^{qb}H=x_1H
  \]
  with a similar observation concerning the second coordinate, the
  equivariance easily follows.
\end{proof}

\begin{Lemma}
  \label{l:in4} The cocycle~$ \theta^{(p,q)} $ is regular.
\end{Lemma}
\begin{proof}
  Let~$ J $ be an algebraic automorphism of~$ H\times H $ given by the
  matrix~$ \begin{bmatrix}
    q & -p\\
    a& b
  \end{bmatrix}
  $%
  \footnote{$ J(h_1,h_2)=(h_1^qh_2^{-p},h_1^ah_2^b) $.  The inverse of
  it is given by~$ \left[
  \begin{array}{cc}
    b & p\\
    -a& q
  \end{array}
  \right] $.} and consider the cocycle~$ J\circ \theta^{(p,q)}=:(\Psi_1,\Psi_2)
  $, where (in view of~\eqref{in3})
  \begin{align*}
    \Psi_1(g,yH)&=\left(\theta(g^p,y^pH)\right)^q\left(\theta(g^q,y^qcH)
    \right)^{-p}\\
    &=
    \left(g^ps(g^py^pH)^{-1}s(y^pH)\right)^q \left( g^qs(g^qy^qcH)^{-1}s
    (y^qcH)\right)^{-p}\\
    &=
    \left(s(g^py^pH)^{-1}s(y^pH)\right)^q \left(s(g^qy^qcH)^{-1}s(y^qcH)\right)^
    {-p}.
  \end{align*}
  In view of~\eqref{in2b} the values of the cocycle~$ \Psi_1 $ belong
  to the set $ (A^{-1}A)^{-p+q} $ which is relatively compact.  It
  follows by~\eqref{in1} and Theorem~5.2 in
  \cite{Mo-Sch} that the cocycle $ \Psi_1 $ is a coboundary.  Now,
  \begin{align*}
    \Psi_2(g,yH)&= \left(\theta(g^p,y^pH)\right)^a\left(\theta(g^q,y^qcH)
    \right)^{b}\\
    &=
    g^{pa}\left(s(g^py^pH)^{-1}s(y^pH)\right)^a g^{qb}\left(s(g^qy^qcH)^
    {-1}s(y^qcH)\right)^{b}\\
    &=
    g\left(s(g^py^pH)^{-1}s(y^pH)\right)^a \left(s(g^qy^qcH)^{-1}s(y^qcH)\right)^
    {b}.
  \end{align*}
  Note that for each~$ h\in H $, we have
  \[
    \Psi_2(h,yH)=h.
  \]
  It follows immediately that~$ \tau_{\Psi_2} $ is transitive,%
  \footnote{Indeed, we have
  \[
    \left(\tau_{\Psi_2}\right)_{hg}(H,e)=\left(\tau_{\Psi_2}\right)_h
    \left(\tau_{\Psi_2}\right)_g(H,e)=
    \left(\tau_{\Psi_2}\right)_h(gH,\Psi_2(g,H))=(gH,h\Psi_2(g,H)).
  \]
  } whence~$ \Psi_2 $ is an ergodic cocycle.

  It follows that our cocycle~$ J\circ \theta^{(p,q)} $ is
  cohomologous to a cocycle taking values in the group~$ \{e\}\times H
  $ and the cocycle is ergodic (the corresponding skew product is
  transitive).  If we write
  \[
    J\circ \theta^{(p,q)}(g,yH)= \big(\eta(gyH)^{-1}\eta(yH),\Psi_2(g,yH)\big)
  \]
  then
  \[
    \theta^{(p,q)}(g,yH)=
    \big(\Psi_2(x,yH)^p \big(\eta(gyH)^{-1}\eta(yH)\big)^b, \Psi_2
    (x,yH)^q\big(\eta(gyH)^{-1}\eta(yH)\big)^{-a}\big).
  \]
  It follows that~$ \theta^{(p,q)} $ is cohomologous to the cocycle~
  \[
    (g,yH)\mapsto\left(\left(\Psi_2(g,yH)\right)^p,\left(\Psi_2(x,yH)\right)^q\right)
  \]
  taking values in the subgroup
  \[
    \cf^{(p,q)}:=\big\{(h_1,h_2)\in H\times H:\:  h_1^q=h_2^p\big\}=\big\{(h^p,h^q):\:h\in
    H\big\}.
  \]
  The latter cocycle is ergodic,%
  \footnote{The map~$ (xH,(h_1,h_2))\mapsto (xH,h_1^ah_2^b) $ settles
  an isomorphism of~$ \tau_{(\Psi_2^p,\Psi_2^q)} $ with~$ \tau_{\Psi_2}
  $.} which completes the proof.
\end{proof}

\subsection{Ergodic joinings of the \texorpdfstring{$ G $}{G}-action
\texorpdfstring{$ (\tau_{\theta,\cs})^{(p)} $}{} with \texorpdfstring{$
(\tau_{\theta,\cs})^{(q)} $}{}} We are interested in a description of~$
\tilde{\la}\in J^e((\tau_{\theta,\cs})^{(p)},(\tau_ {\theta,\cs})^{(q)})
$.  Recall that
\begin{align*}
  \left((\tau_{\theta,\cs})^{(p)}\times (\tau_{\theta,\cs})^{(q)}\right)_g
  &\big((x_1H,y_1),(x_2H,y_2)\big)\\=&
  \big(g^px_1H,S_{\theta(g^p,x_1H)}(y_1),g^qx_2H,S_{\theta(g^q,x_2H)}(y_2)\big).
\end{align*}
But~$ \widetilde{\la}|_{G/H\times G/H}=:\la $ is an ergodic joining of
$ \tau^{(p)} $ with~$ \tau^{(q)} $, and by unique ergodicity, it is an
ergodic component of the~$ G $-action~$ \tau^{(p)}\times \tau^{(q)} $.
Hence we can pass to~$ \ch_c $ replacing~$ (G/H\times G/H,\la) $ by~$
(G/H,m_{G/H}) $, with~$ c\in G/H $ in this notation implicit.  It
follows that we now consider the~$ G $-action $ \tau_{\theta^{(p,q)},\cs\ot\cs}
$ with an ergodic measure $ \widetilde{\la} $ (whose projection on the
first and the second, and the first and the third coordinates are
equal to $ m_{G/H}\ot\nu $). Here, $ \cs\ot\cs $ denotes the product~$
H\times H $-action~$ (S_h\times S_{h'})_{(h,h')}\in H\times H $.
Moreover,
\[
  \left(\tau_{\theta^{(p,q)},\cs\ot\cs}\right)_g(xH,y_1,y_2)= \big(\tau_g(xH),S_
  {\theta(g^p,x^pH)}(y_1),S_{\theta(g^p,x^pcH)}(y_2)\big).
\]
But, by Lemma~%
\ref{l:in4}, $ \theta^{(p,q)} $ is cohomologous to the cocycle~$ (\Psi_2^p,\Psi_2^q)
$ taking values in~$ \cf^{(p,q)}\subset H\times H $, and the latter
cocycle is ergodic.  As a matter of fact:
\begin{equation*}
   \theta^{(p,q)}\cdot\big(\eta(g,\cdot)^b, \eta(g,\cdot)^
  {-a}\big)=(\eta^b,\eta^{-a}) \cdot\big(\Psi_2^p,\Psi_2^q\big).
\end{equation*}
This yields an isomorphism (an equivariant map)
  \[
  Id_{(S_{\eta^b},S_{\eta^{-a}})}\big(xH,y_1,y_2\big)= \big(xH,S_{\eta
  (xH)^b}(y_1),S_{\eta(xH)^{-a}}(y_2)\big)
  \]
between the~$ G $-actions~$ \tau_{\theta^{(p,q)},\cs\ot\cs} $ and~$
\tau_{(\Psi_2^p,\Psi_2^q),\cs^{(p)}\times\cs^{(q)}} $.

Now, by Theorem~3 in
\cite{Le-Me-Na}, since~$ (\Psi^p_2,\Psi_2^q) $ is ergodic, we obtain
the following.

\begin{Lemma}
  \label{l:in5} We have
 \begin{equation}
  \label{in21} 
    \Big(\big(Id_{(S_{\eta^b},S_{\eta^{-a}})}\big)^{-1}\Big)_\ast (\widetilde
    {\la})=m_{G/H}\ot \kappa,
  \end{equation}
  where~$ \kappa\in J^e(\cs^{(p)},\cs^{(q)}) $.
\end{Lemma}

\subsection{Proof of Proposition~\ref{PropositionG}} Before we start proving Proposition~\ref{PropositionG}, let us make the following observation.  Assume that~$
\crr=(R_g)_{g\in G} $ is a $ G $-action on~$ (Z,\mu) $ and let~$ (W,\xi)
$ be a probability standard Borel space.  Assume that we have two
Rokhlin cocycles:
\[
  \Theta,\Psi:G\times Z\to {\rm Aut}(W,\xi)
\]
and a measurable map $ \Sigma:Z\to {\rm Aut}(W,\xi) $.  Assume that~$
Id_{\Sigma} $ establishes an isomorphism of~$ (\crr_{\Psi},Z\times
W,\mu\ot\xi) $ with~$ (\crr_{\Theta},Z\times W,\rho) $, where~$ \rho|_Z=\mu
$.  Assume that~$ f\in L^1(Z,\mu) $, $ g\in L^1(Y,\xi) $, then
\begin{align*}
  \int_{Z\times W}f(z)g(w)\,d\rho(z,w)&= \int_{Z\times W}(f\ot
  g)\circ Id_{\Sigma} (z,w)\,d\mu(z)d\xi(w)
  \\&=
\int_{Z}f(z)\left(\int_W g(\Sigma_z(w))\,d\xi(w)\right)d\mu(z)
\\&=
 \int_{Z}f(z)\left(\int_W g(w)\,d\left((\Sigma_z)_\ast(\xi)\right)(w)\right)d\mu
  (z).
\end{align*}
It follows that
\begin{equation}
  \label{in22} \left|\int_{Z\times W}f(z) g(z)\,d\rho(z,w)\right|\leq \|f\|_
  {L^1(\mu)}\sup_{z\in Z}\left|\int_W g\,d\left(\left(\Sigma_z\right)_\ast
  (\xi)\right)\right|.
\end{equation}

\begin{proof}[Proof of Proposition~\ref{PropositionG}]
Consider now the situation in which~$ \crr=\tau $, $ Z=G/H $, $\mu=m_{G/H}$, $ W=Y\times
Y $, $\xi=\kappa$:  here~$ \crr_{\Theta}=\tau_{\theta^{(p,q)},\cs\otimes\cs} $ acts  on~$ G/H\times
Y\times Y $ and preserves~$ \rho=\tilde{\la} $ (the parameter~$ c $ of
the ergodic component is implicit), and the other $ G $-action $\crr_{\Psi}=
\tau_{(\Psi_2^p,\Psi_2^q),\cs^{(p)}\times\cs^ {(q)}} $ preserves
$ m_{G/H}\ot\kappa $.  In view of~\eqref{in21}, we consider the
isomorphism $ Id_{(S_{\eta^b},S_{\eta^{-a}})} $ between $ \tau_{\theta^
{(p,q)},\cs\otimes\cs} $ and $ \tau_{(\Psi_2^p,\Psi_2^q),\cs^ {(p)}\times\cs^
{(q)}} $.  Hence, in our notation,
\[
  \Sigma_{xH}=S_{\eta(xH)^b}\times S_{\eta(xH)^{-a}}.
\]
It follows that the fiber automorphisms are of the form~$ S_{h^{b}}\times
S_{h^{-a}} $.  Therefore, each fiber automorphism commute with the $ G
$-action~$ \cs^{(p)}\times \cs^{(q)} $.  It easily follows that
\[
  \left(\Sigma_{xH}\right)_\ast \kappa\in J^e(\cs^{(p)},\cs^{(q)})
\]
and from~\eqref{in22}, we obtain that for each~$ f\in L^2(G/H,m_{G/H})
$ and~$ F_1,F_2\in L^2(Y,\nu) $, we have
\begin{equation}
  \label{in23} \left|\int_{G/H\times Y\times Y}f\ot F_1\ot F_2\,d\widetilde
  {\la}\right|\leq \|f\|_1\sup_{\kappa'\in J^e(\cs^{(p)},\cs^{(q)})}\left|
  \int_{Y\times Y}F_1\ot F_2\,d\kappa'\right|.
\end{equation}
To deduce the assertion, it is now enough to notice that, although on
the LHS of~\eqref{in23}, the constant~$ c $ is implicit, it has no
influence since if we consider (as we should)~$ f_1\ot F_1\ot f_2\ot F_2
$ as a member of~$ L^2(G/H\times Y\times G/H\times Y,\widetilde{\la}) $
(with~$ \widetilde{\la} $ an arbitrary ergodic joining of~$ (\tau_ {\theta,\cs})^
{(p)} $ and $ (\tau_{\theta,\cs})^{(q)} $) then by the Schwarz
inequality, the $ L^1(\widetilde{\la}) $-norm of~$ f_1\ot f_2 $ will
be bounded by $ \|f_1\|_{L^2(m_{G/H})}\|f_2\|_{L^2(m_{G/H})} $, hence
does not depend on~$ \widetilde{\la} $.  The result follows.
\end{proof}

\subsection{Examples and applications} Assume that we have an ergodic
$ G $-action $ \cs=(S_g)_{g\in G} $ on~$ \ycn $.  If~$ H\subset
G $ is cocompact then, by
\cite{Zi}, the induced~$ G $-action from the subaction~$ (S_h)_{h\in H}
$ is isomorphic to~$ \cs\times \tau $ (by an isomorphism being the
identity on the $ \tau $-coordinates).

It follows now from Proposition~\ref{PropositionG} that:%
\footnote{This should be compared with Remark~%
\ref{r:in00} (iv).}

\begin{Cor}
  If~$ \cs=(S_g)_{g\in G} $ is as above and the subaction~$ (S_h)_ {h\in
  H} $ has the AOP property, then the~$ G $-action~$ \cs\times\tau= (S_g\times\tau_g)_
  {g\in G} $ on~$ (Y\times G/H,\nu\otimes m_{G/H}) $ has the relative
  (with respect to~$ \tau $) AOP property.
\end{Cor}

\begin{Example}
  If we consider~$ \Z\subset\R $, then a natural choice of the
  selector~$ s:\R/Z\to\R $ satisfying~\eqref{in2},~\eqref{in2a} and~\eqref{in2b} is to
  set~$ s(r):=\{r\} $ being the fractional part of~$ r\in\R $. Then
  \[
    \theta(t,\{r\})=-\{t+\{r\}\}+t+\{r\}=[t+\{r\}]\in\Z.
  \]
  Assume now that~$ S $ is an ergodic automorphism of~$ \ycn $. It
  follows that the induced $ \R $-action~$
  \widetilde{S} $ is given by
  \[
    \widetilde{S}_t(y,\{r\})=(S^{\lfloor t+\{r\}\rfloor}y,\{t+r\}),
  \]
  hence, it is the standard {\em suspension} construction over~$ S $.
\end{Example}

\begin{Example}
  Consider~$ n\Z\subset\Z $.  Then~$ s:\Z/n\Z\to\Z $ is given by~$ s(k+n\Z)=k\text
  { mod }n $.  Then, by writing~$ m+k=tn+r $ with~$ 0\leq r<n $, we
  have
  \[
    \theta(m,k\text{ mod }n)=-(m+(k\text{ mod }n))+m+(k\text{ mod }n)=-r+m+k=tn\in
    n\Z.
  \]
  If we are now given an ergodic $ n\Z $-action%
  \footnote{This is of course a~$ \Z $-action.} we can induce. Instead
  of doing this directly, first write the obvious automorphism~$ n\ell=\ell
  $ between~$ n\Z $ and~$ \Z $ and then rewrite the cocycle~$ \theta $,
  which now becomes
  \[
    \theta(m,x)=\left\lfloor\frac{x+m}n\right\rfloor.
  \]
  Since this is a~$ \Z $-cocycle, it is entirely determined by the
  function~$ \theta=\theta(1,\cdot) $ given by:~$ \theta(x)=0 $ if~$
  x=0,\ldots,n-2 $ and~$ \theta(n-1)=1 $. If now $ S $ is an ergodic
  automorphism of~$ \ycn $ then its ($ n $-discrete) induced is given
  by~$ \widetilde{S} $ acting on~$ Y\times\{0,1,\ldots,n-1\} $ by the
  formula:~$ \widetilde{S}(y,j)=(y,j+1) $ if~$ j=0\ldots,n-2 $, and $
  \widetilde{S}(y,n-1)=(Sy,0) $.  Hence, we obtain the classical $ n $-discrete
  suspension.
\end{Example}

\begin{proof}[Proof of Corollary~\ref{CorollaryH}]
  The~$ k $-discrete suspension~$ \widetilde{T} $ of~$ T $ is a
  uniquely ergodic homeomorphism of the
  space~$ X\times\{0,1,\ldots,k-1\} $.  Define~$ F(z,i)=0 $
  if~$ i\neq j $ and~$ z\in X $, and~$ F(z,j)=f (z) $.
  Then~$ F\perp L^2(\{0,1,\ldots,k-1\}) $.  Now, by the relative AOP
  property, we have
  \[
    \frac1N\sum_{n\leq N}F(\widetilde{T}^n(x,0))\bfu(n)\to0.
  \]
  It is however not hard to see that by the definition of~$ F $,
  \[
    \frac1N\sum_{n\leq N}F(\widetilde{T}^n(x,0))\bfu(n)=\frac1N\sum_{n\leq
    N,n=j\text{ mod }k}f(T^nx)\bfu(n)
  \]
  and the result follows.
\end{proof}

\begin{Lemma}\label{l:podciag}
  Let $ T $ be a uniquely ergodic homeomorphism
  of~$ X $, with the unique invariant measure~$ \mu $.  Assume
  that~$ (X,\mu,T) $ has the AOP property. Assume additionally that there exists a set $ \mathscr{C}\subset
  C(X)\cap L^2_0(X,\mu) $ whose linear span is dense in~$ L^2_0   (X,\mu) $
  such that for all~$ f\in\mathscr{C} $ and all~$
      \omega\in \C $ with~$ |\omega|=1 $ there exists a homeomorphism
      $ S:X\to X $ such that~$ f(T^n(Sx))=\omega f(T^n x) $ for every
      $ x\in X $ and~$ n\in\Z $.

      Let $\widetilde{T}:\widetilde{X}\to\widetilde{X}$ be a discrete suspension of $T$.
  Then for every $F\in C(\widetilde{X})$  and every $\widetilde{x}\in \widetilde{X}$ we have
   \begin{equation} \label{eq:krotper:0}
   \frac 1 M\sum_{M\leq m<2M}\left|\frac1H\sum_{m\leq
        n<m+H}F(\widetilde{T}^n\widetilde{x})\mob(n)\right| \longrightarrow 0
  \end{equation}
      when~$ H\to\infty $ and~$ H/M\to 0 $.
\end{Lemma}

\begin{proof}
 Suppose that $\widetilde{T}$ the $k$-discrete suspension of $T$. Let us consider the subspace $H:=L^2(\widetilde{X})\ominus L^2(\{0,\ldots,k-1\})$.
 By Proposition~\ref{PropositionG}, the subspace $H$ satisfies the AOP property.

 For every $f\in \mathscr{C}$ and $0\leq j<k$ denote by $f_j:\widetilde{X}\to\C$ the continues function which vanishes on all level sets except $X\times\{j\}$, where
 is given by $f_j(x,j)=f(x)$. Functions of such form  establish a linearly dense subset of $H$.

 Suppose that  $ S:X\to X $ is a homeomorphism such that~$ f(T^n(Sx))=\omega f(T^n x) $ for every
      $ x\in X $ and~$ n\in\Z $. Then for the homeomorphism $\widetilde{S}:\widetilde{X}\to\widetilde{X}$ given by $\widetilde{S}(x,l)=(Sx,l)$ we have
      $ f_j(\widetilde{T}^n(\widetilde{S}\widetilde{x}))=\omega f_j(\widetilde{T}^n \widetilde{x}) $ for every
      $ \widetilde{x}\in \widetilde{X} $ and~$ n\in\Z $.
 Now the arguments used in the proof of Theorem~\ref{thm:prop} applied to the family of functions $\{f_j\}$ show that \eqref{eq:krotper:0} holds
 for every continuous function $F\in H$. If $F\in L^2(\{0,\ldots,k-1\})$ then the sequence $\big(F(\widetilde{T}^n\widetilde{x})\big)$ is periodic
 and \eqref{eq:krotper:0} follows from Remark~\ref{rem:mrt}. It follows that \eqref{eq:krotper:0} holds
 for every $F\in C(\widetilde{X})$.
\end{proof}

\begin{Lemma}
 Let $\widetilde{T}:\widetilde{X}\to\widetilde{X}$ be a homomorphism coming from Lemma~\ref{l:podciag}. Then for every $k\geq 1$, $0\leq j\leq k$,
 $F\in C(\widetilde{X})$ and $\widetilde{x}\in\widetilde{X}$ we have
 \begin{equation} \label{eq:krotper:01}
   \frac 1 M\sum_{M\leq m<2M}\left|\frac1H\sum_{m\leq
        n<m+H}F(\widetilde{T}^n\widetilde{x})\mob(kn+j)\right| \longrightarrow 0
  \end{equation}
      when~$ H\to\infty $ and~$ H/M\to 0 $.
\end{Lemma}

\begin{proof}
  Following the proof of Corollary~\ref{CorollaryH} we consider the $ k $-discrete suspension~$ \widetilde{\widetilde{T}} $ of $\widetilde{T}$.
  By Proposition~2.4 in~\cite{Zi}, $ \widetilde{\widetilde{T}} $ is also a discrete suspension of ${T}$.

  For every $F\in C(\widetilde{X})$ define
  $\widetilde{F}(\widetilde{z},j)=F(\widetilde{z})$ and $\widetilde{F}(\widetilde{z},l)=0$ for $l\neq j$.
  Since $\widetilde{F}\in C(\widetilde{\widetilde{X}})$, by Lemma~\ref{l:podciag} applied to $ \widetilde{\widetilde{T}} $ as a discrete suspension of ${T}$,
  we have
  \[
  \frac 1 {kM}\sum_{kM\leq m<2kM}\left|\frac1{kH}\sum_{m\leq
        n<m+kH}\widetilde{F}(\widetilde{\widetilde{T}}^n(\widetilde{x},0))\mob(n)\right|\longrightarrow 0.
  \]
  Moreover,
  \begin{align*}
  \frac 1 M\sum_{M\leq m<2M}&\left|\frac1H\sum_{m\leq
        n<m+H}F(\widetilde{T}^n\widetilde{x})\mob(kn+j)\right|\\&= k^2\frac 1 {kM}\sum_{kM\leq km<2kM}\left|\frac1{kH}\sum_{km\leq
        n<km+kH}\widetilde{F}(\widetilde{\widetilde{T}}^n(\widetilde{x},0))\mob(n)\right|\\&
        \leq k^2\frac 1 {kM}\sum_{kM\leq m<2kM}\left|\frac1{kH}\sum_{m\leq
        n<m+kH}\widetilde{F}(\widetilde{\widetilde{T}}^n(\widetilde{x},0))\mob(n)\right|\longrightarrow 0,
  \end{align*}
  which completes the proof.
\end{proof}

\begin{proof}[Proof of Corollary~\ref{CorollaryI}]
  First observe the it suffices to focus only on basic nil\-se\-quences,
  since all considered sequences are or are appropriately approximated
  by basic nil-sequences.  Moreover, every nilsequence is of the form
  $\big(F(\widetilde{T}^n\widetilde{x})\big)$, where $T$ is a
  homeomorphism satisfying the assumption of Lemma~\ref{l:podciag} and
  $F\in C(\widetilde{X})$.  Indeed, $T$ is an ergodic affine unipotent
  diffeomorphism on a compact connected nilmanifold. Therefore, by
  Lemma~\ref{l:podciag}, the sequence
  $\big(F(\widetilde{T}^n\widetilde{x})\big)$ meets \eqref{eq:corI}.
\end{proof}

All examples of AOP actions that appeared so far in the paper have
either purely discrete or, they have mixed spectrum:  discrete and
Lebesgue.  We will now give examples of AOP flows that have mixed
spectrum:  discrete and continuous singular, making use of recent
results from
\cite{Ku-Le1}.

Below, we use Remark~%
\ref{r:in1} and \eqref{in9}.

\begin{Lemma}
  \label{l:sre1} Let~$ T $ be an ergodic automorphism of probability a
  standard Borel space~$ \xbm $ and~$ \va:X\to K $ a cocycle with
  values in an abelian lcsc group.  Then the formula
  \[
    \widetilde{\va}(t,(x,s)):=\va^{(\lfloor t+s\rfloor )}(x)
  \]
  defines a cocycle for the $ \R $-action of the suspension~$
  \widetilde{T} $ of~$ T $.
\end{Lemma}
\begin{proof}
  First notice that
  \[
    \widetilde{T}_t(x,s)=(T^{\lfloor t+s\rfloor }x,\{t+s\}),
  \]
  whence $ \widetilde{T}_{t_1+t_2}(x,s)=\widetilde{T}_{t_2}(\widetilde
  {T}_{t_1}(x,s)) $ implies
  \begin{equation}
    \label{su1} \lfloor t_1+t_2+s\rfloor =\lfloor t_1+s\rfloor +\lfloor t_2+\{t_1+s\}\rfloor .
  \end{equation}
  Now, we have
  \begin{align*}
    \widetilde{\va}(t_1+t_2,(x,s))&=\va^{(\lfloor t_1+t_2+s\rfloor )}(x)\stackrel{\eqref
    {su1}}{=}\va^{(\lfloor t_1+s\rfloor +\lfloor t_2+\{t_1+s\}\rfloor )}(x)\\
    &    =
    \va^{(\lfloor t_1+s\rfloor )}(x)+\va^{(\lfloor
      t_2+\{t_1+s\}\rfloor )}(T^{\lfloor t_1+s\rfloor }x)
\\&=
    \widetilde{\va}(t_1,(x,s))+\widetilde{\va}(t_2,\widetilde{T}_{t_1}
    (x,s)).
  \end{align*}
\end{proof}

Assume now that~$ K=\R $, so~$ \va:X\to\R $.  We will now consider a
special class of extensions of~$ T $.  Namely, let~$ \cs=(S_t) $ be an
ergodic flow on a probability standard Borel space~$ \ycn $. Then
consider the (probability) space~$ (X\times Y,\cb\ot\cc,\mu\ot\nu) $
on which we consider the (measure-preserving) automorphism:
\[
  \tfs(x,y):=(Tx,S_{\va(x)}(y)).
\]
Note that if~$ \cs $ is a continuous flow acting on~$ Y $ compact and
if~$ \va $ is continuous then~$ \tfs $ is a homeomorphism of~$ X\times
Y $.  Our special interest in this class of actions come from the
following result.

\begin{Th}[\cite{Ku-Le1}] \label{t:kule}
  For each~$ \va:\T\to\R $ of class~$ C^2 $, $ \va $
  different from a tri\-go\-polynomial, there is~$ \alpha $
  irrational that if~$ Tx=x+\alpha $ and~$ \cs $ is an arbitrary
  uniquely ergodic flow then~$ \tfs $ has the AOP property.  If~$ \va
  (x)=x-\frac12 $ then for each~$ \alpha $ with bounded partial
  quotients and each uniquely ergodic~$ \cs $ which has no non-trivial
  rational eigenvalues, $ \tfs $ has the AOP property.
\end{Th}

It follows now from Theorem~\ref{PropositionG} that the suspension of the~$ \tfs $
with the AOP property will also enjoy the same property (for flows and relatively).
We will now try to say a little bit more on the structure of these
suspensions (including a relationship with nilflows, see Proposition~%
\ref{p:negK}).

\begin{Lemma}
  \label{l:sre2} We have (up to natural isomorphism)~$ \widetilde{\tfs}=\widetilde
  {T}_{\widetilde{\va},\cs} $.
\end{Lemma}
\begin{proof} We have
  \begin{align*}
    \left(\widetilde{\tfs}\right)_t((x,y),s)&=\big(\left(\tfs\right)^{\lfloor t+s\rfloor }
    (x,y),\{t+s\}\big)\\& = \big(T^{\lfloor t+s\rfloor }x, S_{\va^{(\lfloor t+s\rfloor )}(x)}(y),\{t+s\}\big)
    \\&=
    \big(T^{\lfloor t+s\rfloor }x, \{t+s\}, S_{\va^{(\lfloor t+s\rfloor )}(x)}(y)\big)=\big(\widetilde{T}_t(x,s),S_
    {\widetilde{\va}(t,(x,s))}(y)\big)
    \\&= \left(\widetilde{T}_{\widetilde{\va},\cs}\right)_t
    ((x,s),y).
  \end{align*}
\end{proof}
By its definition, each suspension flow has the linear flow~$ {\mathcal L}
$:~$ L_tx=x+t $ on~$ \T $, as its factor, %
\footnote{The factor map is given by~$ (x,s)\mapsto s $.} so if we
think about good ergodic properties of the suspension, it should be
considered relatively to this factor~$ {\mathcal L} $.  For example, if~$
T $ is weakly mixing then $ \widetilde{T} $ is relatively weakly
mixing over~$ {\mathcal L} $.  Indeed,
\[
\begin{split}
    \left(\widetilde T\times_{\mathcal L}\widetilde T\right)_t(x_1,x_2,s)&=
  \big(T^{\lfloor t+s\rfloor }x_1,T^{\lfloor t+s\rfloor }x_2,\{s+t\}\big)\\&=
  \big((T\times T)^{\lfloor t+s\rfloor }(x_1,x_2),\{s+t\}\big),
\end{split}
\]
so the relative product is just the suspension over~$ T\times T $.
Similar calculation can be done if we want to compute some relative
properties of~$ \widetilde\tfs $ over~$ \widetilde{T} $.  Indeed, for
the relative product we have:
\begin{align*}
  &(\widetilde\tfs)\times_{\widetilde T}(\widetilde\tfs)
  ((x,y_1),(x,y_2),s)\\&\qquad=
  \big((\tfs)^{\lfloor t+s\rfloor }(x,y_1),(\tfs)^{\lfloor t+s\rfloor }(x,y_2)
  ,\{s+t\}\big)\\&\qquad
  =
  \big(\left(\tfs\times_{T}\tfs\right)^{\lfloor t+s\rfloor }(x,y_1,y_2),\{s+t\}\big),
\end{align*}
so again the relative product of the suspension is the suspension of
the relative product.  Now, the examples of~$ \tfs $ in Theorem~%
\ref{t:kule} are relatively weakly mixing over~$ T $ whenever $ \cs $
is weakly mixing.  It follows that, we have the following.

\begin{Cor}
  \label{c:kule} Assume that~$ T $ and~$ \va $ are as in Theorem~%
  \ref{t:kule}.  Assume that~$ \cs $ is weakly mixing.  Then~$
  \widetilde\tfs $ have the relative AOP property and they are relatively
  weakly mixing extensions of~$ \widetilde{T} $.~%
  \footnote{Note that this is completely different case than nilflows
  which are distal extensions of of linear flows.}
\end{Cor}

The situation changes if on the fibers, instead of~$ \cs $ weakly
mixing, we put a distal flow.  For example, if we apply the above to $
\va(x)=x-\frac12 $ and~$ S_t=x+t $, we obtain a nilflow (the
Heisenberg case), but because of restrictions on~$ \cs $ (which in
this case has rational eigenvalues), we cannot apply
\cite{Ku-Le1} directly.  Consider~$ S_tx=x+\beta t $ with~$ \beta $
irrational.  Here we obtain that~$ \tfs $ has the AOP property.  The
natural question arises whether in case of~$ \beta $ irrational, we
also obtain a nilflow.  The answer is negative as the following result
shows.

\begin{Prop}
  \label{p:negK} If~$ \va(x)=x-\frac12 $ and~$ S^{(\beta)}_t(x):=x+t\beta
  $ with~$ \beta\notin\Q $ then~$ \widetilde\tfs $ has singular
  spectrum.
\end{Prop}
\begin{proof}
  We have~$ \widetilde{T}_{\widetilde{\va},\cs^{(\beta)}}=\widetilde
  {T}_{\beta\widetilde{\va},\cs^{(1)}} $.  Then
  $$
    \widetilde{T}_{\beta\widetilde{\va},\cs^{(1)}}=\widetilde{T_{\beta\va,\cs^
    {(1)}}}.
  $$
  We have to now argue that the maximal spectral type of $ T_{\beta\va,\cs^
  {(1)}} $ is singular.  In view of
  \cite{Le-Pa3}, to compute the maximal spectral type of~$ T_{\beta\va,\cs^
  {(1)}} $, we need to calculate the maximal spectral types of
  weighted unitary operators given by the multiples~$ r\beta\va $, $ r\in\R
  $, and then integrate this against the maximal spectral type of~$
  \cs^{(1)} $.  The latter is simply a purely atomic measure whose
  atoms are in~$ \Z $.  This means that we are interested only in the
  weighted operators given by~$ m\beta\va $ and they are all with
  singular spectrum by
  \cite{Iw-Le-Ma}.
\end{proof}

\section*{Appendix}
\appendix
\section{Cocycles and group extensions} Assume that~$ T $ is an
ergodic automorphism of a probability standard Borel space $ \xbm $.
Let~$ K $ be a compact (metric) abelian group.  Each Borel function~$
\va:X\to K $ is called a {\em cocycle}.  Actually, $ \va $ determines~$
\va^{(\,\cdot\,)}(\,\cdot\,):\Z\times X\to K $:
\begin{equation}
  \label{defcoc} \va^{(n)}(x)=\left\{%
  \begin{array}{ccc}
    \va(x)+\va(Tx)+\ldots+ \va(T^{n-1}x)&\mbox{if}&n\geq1\\
    0&\mbox{if}&n=0\\
    -(\va(T^{n}x)+\ldots+\va(T^{-1}x))& \mbox{if}&n<0,
  \end{array}
  \right.
\end{equation}
satisfying the cocycle identity~$ \va^{(m+n)}(x)=\va^{(m)}(x)+\va^{(n)}
(T^mx) $ for all~$ m,n\in\Z $ and $ \mu $-a.e.~$ x\in X $. Having~$ T $
and~$ \va $, we can define the corresponding {\em group extension}~$ T_\va
$ of~$ T $ by setting:
\[
  T_\va:X\times K\to X\times K,\,T_\va(x,k)=(Tx,\va(x)+k).
\]
Clearly $ T_\va $ is an automorphism of~$ (X\times K,\cb\ot\cb(K),\mu\ot\la_K)
$, where~$ \la_K $ stands for Haar measure on~$ K $.  Then~$ T_\va $
is ergodic if and only if the only measurable solutions~$ \xi:X\to\bs^1
$ of the equations
\[
  \chi\circ \va=\xi\circ T/\xi,\,\chi\in\widehat{K}
\]
(i.e.\ $ \chi\circ\va $ is a {\em coboundary}) exist when~$ \xi $ is a
constant function and~$ \chi=\raz $
\cite{An}.

Let~$ \sigma_k:X\times K\to X\times K $ be defined by $ \sigma_k(x,k')=
(x,k'+k) $.  Then~$ \sigma_k $ ($ k\in K $) is an automorphism of~$ (X\times
K,\cb\ot\cb(K),\mu\ot\la_K) $ and $ \sigma_k$ is an element of the centralizer $ C(\tf) $.  Then, by a
slight abuse of notation, we have $ K\subset C(\tf) $ and~$ K $ is a
compact (abelian) subgroup in the weak topology.  The reciprocal is
also true.

\begin{Prop}[e.g.\   \cite{Ju-Ru}]  \label{p:grext}
  Assume that~$ \ov{T} $ is an ergodic automorphism of a
  probability standard Borel space~$ (\ov{X},\ov{\cb},\ov{\mu}) $.
  Assume that~$ K\subset C(\ov{T}) $ is a compact abelian subgroup (we
  assume that~$ K $ acts freely on~$ X $) of the centalizer.  Let
  \[
    \ov\ca:=\{\ov{A}\in\ov{\cb}:\:k\ov A=\ov A\text{ for each } k\in K\}.
  \]
  Then~$ \ov{\ca} $ is an~$ \ov{T} $-invariant~$ \sigma $-algebra and
  \[
  K=\{\ov R\in C(\ov T):\:\ov R\ov A=\ov A\text{ for each }\ov{A}\in\ov
  {\ca}\}.
  \]
  If~$ T=\ov{T}|_{\ov\ca} $ is the factor automorphism
  acting on the factor space \[ \xbm := (\ov{X}/\ov{\ca},\ov{\ca},\ov{\mu}|_
  {\ov\ca}) \] then there is a cocycle~$ \va:X\to K $ such that~$ \ov{T}
  $ is isomorphic to~$ \tf $ (with an isomorphism being the identity
  on~$ \ov\ca $).
\end{Prop}

\begin{Remark}
  We recall how to define~$ \va $.  Let~$ \pi:\ov X\to X $ be the
  factor map.  Then the fibers~$ \pi^{-1}(x) $ are copies of~$ K $.
  Let~$ \xi:X\to \ov X $ be a measurable selector (for~$ \pi $). Then
  for each~$ \ov x\in\pi^{-1}(x) $ there is a unique~$ k_{\ov x}\in K
  $ such that~$ k_{\ov x}\ov x=\xi(x) $.  Note that then
  $$
    k_{\ov x}(\ov T \ov x)=\ov T(k_{\ov x}\ov x)=\ov T\xi(x),
  $$
  whence if we define~$ \va(x)\in K $ as the only element of~$ K $
  such that $ \va(x)(\ov T\xi(x))=\xi(Tx) $ then the map~$ \ov x\mapsto
  (\pi(\ov x),k_{\ov x}) $ establishes an isomorphism between~$ \ov T $
  and $ T_\va $.
\end{Remark}

\section{Invariant measures for group extensions} If~$ T_\va $ is
ergodic, then~$ \mu\ot\la_K $ is the only~$ \tf $-invariant
probability measure~$ \kappa $ whose projection $ \left(\pi_X\right)_\ast
(\kappa) $ is~$ \mu $
\cite{Fu}.  This result has a refinement as follows (see, e.g.\ %
\cite{Ke-Ne},
\cite{Le-Me}).  If~$ \va:X\to K $ is not ergodic then there is a (unique)
closed subgroup~$ K'\subset K $, a cocycle~$ \va':X\to K' $ and a
Borel map $ j:X\to K $ such that
\begin{equation}
  \label{essgp1} \mbox{$ T_{\va'} $ is ergodic as an automorphism of~$
  \big(X\times K',\cb\ot\cb(K'),\mu\ot\la_{K'}\big) $}
\end{equation}
and
\begin{equation}
  \label{essgp2} \va(x)=\va'(x)+j(Tx)-j(x)
\end{equation}
for a measurable~$ j:X\to K $ (i.e.~$ \va $ is {\em cohomologous} to a
cocycle~$ \va' $ which is taking values in a smaller closed subgroup $
K' $ and~$ T_{\va'} $ is ergodic).  The group~$ K' $ is called the {\em
group of essential values} of~$ \va $ and is denoted by $ E(\va) $.
Moreover,
\begin{equation}
  \label{essgp3} E(\va')=E(\va)=\Lambda(\va)^\perp,
\end{equation}
where
\[
  \Lambda(\va):=\big\{\chi\in \widehat{K}:\:  \chi\circ\va\text{ is a
  coboundary}\big\}.
\]
In particular,
\begin{equation}
  \label{essgp4} \mbox{$ \tf $ is ergodic if and only if~$ E(\va)=K $.}
\end{equation}

We intend to describe the set~$ \cm(X\times K,\tf;\mu) $ of all $ \tf $-invariant
probability measures~$ \kappa $ on~$ X\times K $ that $ \left(\pi_X\right)_\ast
(\kappa)=\mu $.  It is again a simplex (with its natural affine
structure).  The set of extremal points is equal to $ \cm^e(X\times K,\tf;\mu)
$ of ergodic members of~$ \cm(X\times K,\tf;\mu) $.  We have the
following.

\begin{Prop}[\cite{Le-Me}]
  \label{p:transl1}
  \begin{sloppypar}
    If~$ \kappa\in \cm^e(X\times K,\tf;\mu) $ then~$
  \left(\sigma_{k}\right)_\ast(\kappa)\in \cm^e(X\times K,\tf;\mu) $
  for each~$ k\in K $.  Moreover, if~$ \kappa,\kappa'\in \cm^e(X\times
  K,\tf;\mu) $ then~$ \left(\sigma_{k_0}\right)_\ast(\kappa)=\kappa' $
  for some~$ k_0\in K $.
  \end{sloppypar}
\end{Prop}

\begin{sloppypar}
  It follows from Proposition~%
\ref{p:transl1} that to describe all ergodic members of~$ \cm(X\times
K,\tf;\mu) $, we need to describe just one.  Let~$ \kappa\in \cm^e(X\times
K,\tf;\mu) $.  Set
\end{sloppypar}

\[
  \operatorname{stab}(\kappa):=\{k\in K:\:\left(\sigma_k\right)_\ast(\kappa)=\kappa\}.
\]

\begin{Lemma}[\cite{Le-Me}]
  \label{l:leme} For each~$ \kappa\in \cm^e(X\times K,\tf;\mu) $, we
  have~$ \operatorname{stab}(\kappa)=E(\va) $.
\end{Lemma}

By~\eqref{essgp1},
\begin{equation}
  \label{essgp33}\cm\big(X\times E(\va),T_{\va'};\mu\big)=\cm^{e}\big(X\times E(\va),T_
  {\va'};\mu\big)=\{\mu\ot\la_{E(\va)}\}.
\end{equation}
Moreover, by~\eqref{essgp2}, the map~$ \Theta:X\times E(\va)\to X\times
K $ given by
\[
  \Theta(x,k')=(x,k'+j(x))
\]
is equivariant, i.e.~$ \Theta\circ T_{\va'}=\tf\circ \Theta $.  It
follows that $ \kappa:=\Theta_\ast(\mu\ot\la_{E(\va)})\in \cm^{e}(X\times
K,\tf;\mu) $.  Finally, notice that if~$ j:X\to K $ satisfies~\eqref {essgp2}
then, for each~$ k\in K $, also $ j_k(x):=j(x)+k $ satisfies~\eqref{essgp2}
(moreover, each Borel $ \ov{j}:X\to K $ satisfying~\eqref{essgp2} is
of the form $ j_k $).  Taking into account this and Proposition~%
\ref{p:transl1}, we obtain the following.

\begin{Prop}[\cite{Le-Me}]
  \label{p:leme} If~$ \kappa\in \cm^e(X\times K,\tf;\mu) $ then there
  exists a Borel~$ j:X\to K $ satisfying~\eqref{essgp2} for which
  \[
    \kappa=\int_X \left(\sigma_{j(x)}\right)_\ast(\delta_x\ot\la_{E(\va)})\,d\mu
    (x).
  \]
\end{Prop}

If we take~$ f\in L^2\xbm $ and a character~$ \chi\in\widehat{K} $
then, once~$ \chi $ is non-trivial, the tensor~$ f\ot\chi $ has the
zero mean for~$ \mu\ot\la_K $.  We will now show that the zero mean
phenomenon holds for many characters and ergodic members of~$ \cm(X\times
K,\tf;\mu) $.

\begin{Cor}
  \label{c:leme} Assume that~$ f\in L^2\xbm $.  If~$ \chi\notin\Lambda
  (\va) $ then~$ \int_{X\times K}f\ot\chi\,d\kappa=0 $ for each~$
  \kappa\in \cm^e(X\times K,\tf;\mu) $.
\end{Cor}
\begin{proof}
  By \eqref{essgp3}, $ \chi\notin\Lambda(\va) $ if and only if~$ \chi\notin
  E(\va)^\perp $ and the latter is equivalent to saying that~$ \chi|_{E
  (\va)}\neq\raz $.  By Proposition~%
  \ref{p:leme},
  \begin{equation}
    \label{fub}\kappa=\int_X \left(\sigma_{j(x)}\right)_\ast(\delta_x\ot\la_
    {E(\va)})\,d\mu(x).
  \end{equation}
  But if~$ \chi|_{E(\va)}\neq\raz $ then~$ \int_{K}\chi\,d\la_{E(\va)}=0
  $ and the same integral vanishes if we integrate against an
  arbitrary shift of~$ \la_{E(\va)} $.  The result now follows from~\eqref
  {fub}.
\end{proof}

Assume now that~$ S $ is an ergodic automorphism of a probability
standard Borel space~$ \ycn $ and let~$ \psi:Y\to L $ be a cocycle,
where~$ L $ is a compact (metric) abelian group.  Assume that~$ \tf $
and~$ S_\psi $ are ergodic and let~$ \widetilde{\rho}\in J^e(\tf,S_\psi)
$.  Denote $ \rho=\widetilde{\rho}|_{X\times Y} $. Then~$ \rho\in J^e(T,S)
$.  We also have
\[
  \widetilde{\rho}\in \cm^e\big((X\times Y)\times (K\times L), (T\times S)_
  {\va\times\psi};\rho\big),
\]
where~$ \va\times\psi:(X\times Y,\rho)\to K\times L $:
\[
  (\va\times\psi)(x,y)=(\va(x),\psi(y))\text{ for }\rho-\text{a.e. }(x,y)\in
  X\times Y.
\]
In other words,
\begin{equation}
  \label{jge} J^e(\tf,S_\psi)\subset \cm^e\big((X\times Y)\times (K\times
  L), (T\times S)_{\va\times\psi};\rho\big).
\end{equation}

We can now take up a more general problem and study the form of \[ \ov
{\rho}\in \cm^e\big(((X\times Y)\times (K\times L), (T\times S)_{\va\times\psi};\rho\big)
\] (with~$ \rho=\ov\rho|_{X\times Y} $). This problem is equivalent to
investigating the invariant measures for the group extensions $ (T\times S,\rho)_
{\va\times\psi} $ with~$ \rho\in J^e(T,S) $.  By denoting $ E_\rho (\va\times\psi)
$ the corresponding group of essential values, the problem is reduced
to study the form of such groups.

\begin{Lemma}
  \label{l:rzut1} If~$ \rho\in J^e(T,S) $ and the cocycles~$ \va $
  and~$ \psi $ considered over~$ (T\times S,\rho) $ are ergodic then~$
  \pi_{K}(E_\rho(\va\times\psi))=K $ and~$ \pi_{L}(E_\rho(\va\times\psi))=L
  $.
\end{Lemma}
\begin{proof}
  It is not hard to see that~$ E(\pi_K\circ (\va\times\psi))=\pi_K(E_\rho
  (\va\times\psi)) $ and since~$ \pi_K\circ (\va\times\psi)=\va $ is
  ergodic over~$ (T\times S,\rho) $ by assumption, the result follows
  from~\eqref{essgp4}.
\end{proof}

\begin{Remark}
  \label{r:uwdonil} Although the problem of describing all elements of
  the set~$ \big\{E_\rho(\va\times \psi):\:\rho\in J^e(T,S)\big\} $ looks as a
  problem more general than the problem of description of the members
  of~$ J^e(\tf,S_\psi) $, in fact, quite often they are the same.  We
  will see the equality of the two sets in~\eqref{jge} when we study
  nil-cocycles in forthcoming sections.
\end{Remark}

\section{A criterion to lift AOP to a group extension} Assume again
that~$ T $ is a totally ergodic automorphism of a probability standard
Borel space~$ \xbm $.  Let~$ K $ be a compact (metric) abelian group
and let~$ \va:X\to K $ be a cocycle.  We assume also that~$ \tf $ is
also totally ergodic.~%
\footnote{This means that~$ \tf $ has no root of unity in its
spectrum.  This is equivalent to saying that for no~$ n\geq2 $, no
character~$ \chi\in\widehat{K} $, we can solve the functional
equation~$ \chi\circ \va=e^{2\pi i/n}\xi\circ T/\xi $ for a
measurable~$ \xi:X\to\bs^1 $
\cite{An}.} We want to study the AOP property for~$ \tf $ (assuming
that~$ T $ enjoys it).  To study this property, we will have to
consider the set $ J^e\left((\tf)^r,(\tf)^s\right) $ for~$
r,s $ coprime.  Note that~$ (\tf)^r=(T^r)_{\va^{(r)}} $, so
we are in the framework of the previous subsection in which~$ T $, $
\va $, $ S $, $ \psi $ are~$ T^r $, $ \va^{(r)} $, $ T^s $ and~$ \va^{%
(s)} $, respectively.

Assume that~$ K=\T^d $ and identify characters of~$ K $ with~$ \Z^d $:
given $ \ov m\in\Z^d $, we set~$ \chi_{\ov m}(\ov x)=e^{2\pi i\langle
\ov m,\ov x\rangle} $.  Given~$ k\geq1 $ and~$ r,s $ coprime, let
\[
  A_{k,r,s}:=\big\{(\ov m,\ov n)\in\Z^d\times\Z^d:\:s^k\ov m=r^k\ov n\big\}\subset
  \widehat{K}\times\widehat{K}.
\]

\begin{Prop}
  \label{p:critAOP} Let~$ T $ be a totally ergodic automorphism of a
  probability standard Borel space~$ \xbm $.  Let~$ \va:X\to \T^d $ be
  a cocycle so that~$ \tf $ is totally ergodic.  Assume that for
  some~$ k\geq1 $ and for all~$ r,s $ coprime, we have
  \begin{equation}
    \label{zawier} \Lambda_{(\va^{(r)}\times\va^{(s)},\rho)}\subset A_
    {k,r,s}\text{ for each }\rho\in J^e(T^r,T^s).
  \end{equation}
  If~$ T $ enjoys the AOP property, so does~$ \tf $.
\end{Prop}
\begin{proof}
  Fix~$ f,g\in L^2\xbm $ and~$ \ov m_0,\ov n_0\in\Z^d $ such that at least one is no-zero.  It is enough
  to show that for all~$ r,s $ coprime and sufficiently large, we have
  \begin{equation}
    \label{pao1} \int f\ot\chi_{\ov m_0}\ot g\ot \chi_{\ov n_0}\,d\widetilde
    {\rho}=0\text{ for each }\widetilde{\rho}\in J^e\left(\tf)^r,(\tf)^s\right),
  \end{equation}
  cf.\ Remark~%
  \ref{r:weaktop}.  By~\eqref{jge}, it is sufficient that~\eqref{pao1}
  holds for each~$ \widetilde{\rho}\in\cm^e\big((X\times X)\times(\T^d\times\T^d),
  (T^r\times T^s)_{\va^{(r)}\times\va^{(s)}};\rho\big) $.  By Corollary~%
  \ref{c:leme}, we only need to show that for all~$ r,s $ coprime and
  sufficiently large, $ (\ov m_0,\ov n_0)\notin\Lambda_{\va^{(r)}\times\va^
  {(s)}} $ for all~$ \rho\in J^e(T^r,T^s) $.  Hence, by assumption, it
  is sufficient to show that~$ (\ov m_0,\ov n_0)\notin A_{k,r,s} $
  for all~$ r,s $ coprime and sufficiently large.  This is however
  clear, if~$ (\ov m_0,\ov n_0)\in A_{k,r,s} $ then~$ s^k\ov m_0=r^k\ov
  n_0 $, whence the coordinates of~$ \ov m_0 $ must be multiples of~$
  r^k $ and the coordinates of~$ \ov n_0 $ must be multiples of~$ s^k $
  and this can hold only for finitely many~$ r,s $.  The result
  follows.
\end{proof}

Since~$ (r,s)=1 $, we have~$ A_{k,r,s}=\{(r^k\ov j,s^k \ov j):\:\ov j\in\Z^d\}
$.  Then, it follows that
\[
  A_{k,r,s}^\perp=\big\{(\ov x,\ov y)\in\T^d\times \T^d:  r^k\ov x+s^k\ov
  y=\ov 0\big\}.
\]

Using Proposition~%
\ref{p:critAOP}, Lemma~%
\ref{l:leme} and~\eqref{essgp3}, we obtain the following.

\begin{Cor}
  \label{c:critAOP} Let~$ T $ be a totally ergodic automorphism of a
  probability standard Borel space~$ \xbm $.  Let~$ \va:X\to \T^d $ be
  a cocycle so that~$ \tf $ is totally ergodic.  Assume that for
  some~$ k\geq1 $ and for all~$ r,s $ coprime, we have
  \begin{equation}
    \label{zawier10} \operatorname{stab}(\widetilde{\rho})\supset A^\perp_{k,r,s}
  \end{equation}
  \begin{sloppypar}
      for some~$ \widetilde{\rho}\in\cm^e\big((X\times X)\times(\T^d\times\T^d),
  (T^r\times T^s)_{\va^{(r)}\times\va^{(s)}};\rho\big) $ and for every
  $ \rho\in
  J^e(T^r,T^s) $.  If~$ T $ enjoys the AOP property, so does~$ \tf $.
  \end{sloppypar}
\end{Cor}

\section{Some lemmata on nilpotent Lie algebras}

In the sequel~$ \g $ is a~$ k $-step nilpotent Lie algebra and
\[
  \mathfrak g=\mathfrak g^{(1)}\supset\mathfrak g^{(2)}\supset\dots
  \supset \mathfrak g^{(k)}
\]
its descending central series.  By Lemma~1.1.1 in
\cite{Co-Gr},
\begin{equation*}
  [\g^{(i)},\g^{(j)}]\subset \g^{(i+j)}\quad\text{for all}\ i,j\geq 1.
\end{equation*}

\begin{Lemma}
  \label{lem:aop_nil:2} A set~$ S=\{X_1, \dots , X_j\}\subset \g $ is
  a minimal set of generators for a nilpotent Lie algebra~$ \mathfrak
  g $ if and only if~$ S[\mathfrak g,\mathfrak g] $ is a basis of the
  vector space~$ \mathfrak g/[\mathfrak g,\mathfrak g] $.
\end{Lemma}
\begin{proof}
  The proof is by induction on~$ k $.  The statement is true for~$ \g $
  abelian ($ k=1 $).  Suppose the statement true for all nilpotent Lie
  algebras of class of nilpotency~$ \ell \le k $ and assume~$ \g $
  nilpotent of class of nilpotency~$ k +1 $.

  Let~$ S\subset \g $ be a set such that~$ S[\mathfrak g,\mathfrak g] $
  is a basis of the vector space~$ \mathfrak g/[\mathfrak g,\mathfrak
  g] $.  Then~$ \g=\< S\> + [\g,\g] $, where~$ \< S\> $ denotes the
  linear span of~$ S $.  It follows that $ \g^{(k+1)}= [\g, \g^ {(k)}]=
  [\<S\>, \g^{(k)}] $ as~$ [[\g,\g],\g^{(k)}]\subset \g^{(k+2)}=0 $.
  Furthermore, the set 
\[ (S\g^{(k+1)})\big[(\mathfrak g/\g^{(k+1)}),(\mathfrak g/\g^{%
  (k+1)})\big] ]\approx S[\g,\g]
\] is a basis of the vector space $ (\mathfrak
  g/\g^{(k+1)})/\big[(\mathfrak g/\g^{(k+1)}),(\mathfrak g/\g^{%
  (k+1)})\big]\approx \g/[\g,\g] $.  Since $ \g/\g^{(k+1)} $ is nilpotent
  of class of nilpotency~$ k $, the induction hypothesis applies to
  the Lie algebra~$ \g/\g^{(k+1)} $ and its subset~$ S\g^ {(k+1)} $,
  so~$ S\g^{(k+1)} $ is a minimal set of generators of~$ \g/\g^{(k+1)}
  $.  Thus the Lie subalgebra~$ \g_S $ generated by~$ S $ projects
  onto~$ \g/\g^{(k+1)} $ under the quotient mapping~$ \g \mapsto \g/\g^
  {(k+1)} $.  Thus it suffices to show that~$ S $ generates a set
  spanning~$ \g^{(k+1)} $.  Let~$ T\subset \g_S $ be a finite set
  projecting to a spanning set of~$ \g^{(k)}/\g^{(k+1)} $.  Then~$ \g^
  {(k)}=\<T\>+\g^{(k+1)} $.  By definition~$ [\<S\>,\<T\>]\subset \g_S
  $.  It follows that
  \[
    \g^{(k+1)}= [\<S\>, \g^{(k)}]= [ \<S\>, \<T\>+\g^{(k+1)}]\subset [\<S\>,
    \<T\>] + \g^{(k+2)}= [\<S\>, \<T\>]\subset \g_S.
  \]

  Finally observe that~$ S $ is minimal, since any proper subset of~$
  S $ does not project to a basis of~$ \g/[\g,\g] $, i.e. does not
  generate~$ \g/[\g,\g] $.  A fortiori it does not generate~$ \g $.

  Assume~$ S $ is a minimal set of generators of~$ \g $.  Then~$ S $
  projects to a generating set of~$ \g/[\g,\g] $, that is a finite
  spanning set for the vector space~$ \g/[\g,\g] $.  Let~$ S_1\subset
  S $ be a subset projecting to a basis of~$ \g/[\g,\g] $.  By the
  above~$ S_1 $ generates~$ \g $.  By the minimality of~$ S $ we
  have~$ S=S_1 $, concluding the proof.
\end{proof}

Let~$ S=\{X_1,\dots, X_j\} $ be a minimal generating set for the
$k$-step nilpotent Lie algebra~$ \g $.  For every~$ (i_1,\dots,i_k)\in \{1,\dots,j\}^k
$ define the Lie~$ k $-fold product
\[
  S_{i_1,i_2,\dots,i_k}:=[X_{i_1},[X_{i_2}, \dots, [X_{i_{k-1}},X_{i_k}]\dots]]
\]
and let~$ V_k(S)\subset\g^{(k)} $ be the linear span of the set of~$ k
$-fold products.

\begin{Lemma}
  \label{lem:aop_nil:3} Let~$ \g $ be a~$ k $-step nilpotent Lie
  algebra and~$ S $ a minimal generating set for~$ \g $.  Then~$ \g^{(k)}=
  V_k(S) $.
\end{Lemma}
\begin{proof}
  The proof is by induction on~$ k $.  When~$ \g $ is abelian the
  statement is obvious.  Suppose that the Lemma is true for all
  nilpotent Lie algebras of class~$ \ell <k $ and let~$ \g $ be of
  class~$ k $ and~$ S $ be a minimal generating set for~$ \g $. By the
  previous Lemma, the projection~$ \bar S=\{\bar X_1,\dots, \bar X_j\}
  $ of~$ S $ into~$ \g/\g^{(k)} $ is a set of generators for~$ \g/\g^{%
  (k)} $.  Then, by the induction hypothesis, the set of~$ (k-1) $-fold
  products of the elements~$ \bar S $ span~$ \g^{(k-1)}/\g^{(k)} $.
  It follows that~$ \g^{(k-1)}= \g^{(k)} + V_{k-1}(S) $.  Since~$ V_
  {k-1}(S)\subset\g^{(k-1)} $, this gives
  \[
    \g^{(k)}=[\g,\g^{(k-1)}]= [\<S\>+\g^{(2)}, \g^{(k)} + V_{k-1}(S) ]=
    [\<S\>, V_{k-1}(S) ]
  \]
  The proof is concluded by the observations that~$ [\<S\>, V_{k-1}(S)
  ] \subset V_{k}(S) $ and that, as~$ \g $ is of class~$ k $, the
  opposite inclusion is true by definition of~$ \g^{(k)} $.
\end{proof}

\begin{Lemma}
  \label{lem:aop_nil:4} Let~$ S=\{X_1,\dots, X_j\} $ and~$ S'=\{X'_1,\dots,
  X'_j\} $ be subsets of a nilpotent~$ k $-step Lie algebra~$ \g $
  such that~$ X_i=X'_i\mod [\g,\g] $ for all~$ i=1,\dots,j $. Then for
  every~$ (i_1,\dots,i_k)\in \{1,\dots,j\}^k $ we have
  \[
    S_{i_1,i_2,\dots,i_k}= S'_{i_1,i_2,\dots,i_k}
  \]
\end{Lemma}
\begin{proof}

  Again the proof is by induction on the class of nilpotency.  The
  statement being trivially true in the abelian case, we assume that
  the Lemma is true for all nilpotent Lie algebras of class~$ \ell <k $.
  Let~$ \g $ be of class~$ k $.  Let~$ X_i'=X_i+Y_i $ with~$ Y_i\in [\g,\g]
  $, for all~$ i=1,\dots,j $.  By the induction hypothesis, applied to
  the algebra~$ \g/\g^{(k)} $ for all~$ (i_2,\dots,i_k)\in \{1,\dots,j\}^
  {k-1} $ , the elements~$ S_{i_2,\dots,i_k} $ and~$ S'_{i_2,\dots,i_k}
  $ coincide modulo~$ \g^{(k)} $, i.e.~differ by an element~$ \alpha_{i_2,\dots,i_k}\in
  \g^{(k)} $.  Thus we have
  \begin{align*}
    S'_{i_1, i_2,\dots,i_k}&= [X'_{i_1}, S'_{i_2,\dots,i_k}] = [X_{i_1}+
    Y_{i_1}, S_{i_2,\dots,i_k} + \alpha_{i_2,\dots,i_k}]\\
    & = [X_{i_1}, S_{i_2,\dots,i_k} ]=S_{i_1, i_2,\dots,i_k}
  \end{align*}
  concluding the proof.
\end{proof}

\footnotesize

\bibliography{aop_biblio}
\bibliographystyle{amsalpha}

\end{document}